\documentclass[a4paper, 11pt]{article}

\usepackage{fullpage}

\usepackage[english]{babel}
\usepackage[T1]{fontenc}
\usepackage[utf8]{inputenc}
\usepackage{amsmath, amsfonts, amssymb, amsthm}
\usepackage{mathrsfs}
\usepackage{mathtools}
\usepackage{bm}
\usepackage{xfrac}
\usepackage{comment}
\usepackage{todonotes}
\usepackage[symbol]{footmisc}

\usepackage{booktabs}
\usepackage[colorlinks=true, linkcolor=blue, citecolor=red]{hyperref}
\usepackage{enumitem}
\usepackage{graphicx}
\usepackage{subcaption}

\usepackage[linesnumbered,ruled,vlined]{algorithm2e}
\SetAlFnt{\small}

\renewcommand{\thefootnote}{\arabic{footnote}}

\newcommand{\N}{\mathbb{N}}

\newtheorem{theorem}{Theorem}[section]

\newtheorem{proposition}[theorem]{Proposition}
\newtheorem{corollary}[theorem]{Corollary}

\newtheorem{definition}[theorem]{Definition}

\newtheorem{remark}[theorem]{Remark}
\usepackage{hyperref}

\usepackage{float}
\usepackage{todonotes}
\usepackage{algorithm2e}
\newlist{checklist}{itemize}{2}
\setlist[checklist]{label=$\square$}

\hypersetup{      
    urlcolor=black
}
\renewcommand{\thefootnote}{\fnsymbol{footnote}}

\newcommand{\vol}{\mathrm{vol}}

\title{Certified and accurate computation of function space norms of deep neural networks}
\author{
	Johannes Gründler\thanks{Faculty of Mathematics, University of Vienna, Kolingasse 14--16, 1090 Vienna, Austria (\href{mailto:a12332806@unet.univie.ac.at}{a12332806@unet.univie.ac.at}, \href{mailto:moritz.maibaum@univie.ac.at}{moritz.maibaum@univie.ac.at}, \href{mailto:philipp.petersen@univie.ac.at}{philipp.petersen@univie.ac.at})} \thanks{Equal contribution}
	\and
	Moritz Maibaum\footnotemark[1] \footnotemark[2]
	\and
	Philipp Petersen\footnotemark[1]
}

\begin{document}
	\maketitle
	
	\renewcommand{\thefootnote}{\arabic{footnote}}
	\setcounter{footnote}{0}
	
	\begin{abstract}
		Neural network methods for PDEs require reliable error control in function space norms. However, trained neural networks can typically only be probed at a finite number of point values. 
		Without strong assumptions, point evaluations alone do not provide enough information to derive tight deterministic and guaranteed bounds on function space norms. In this work, we move beyond a purely black-box setting and exploit the {neural} network structure directly. We present a framework for the certified and accurate computation of integral quantities of neural networks, including Lebesgue and Sobolev norms, by combining interval arithmetic enclosures on axis-aligned boxes with adaptive marking/refinement and quadrature-based aggregation. 
		On each box, we compute guaranteed lower and upper bounds for function values and derivatives, and propagate these local certificates to global lower and upper bounds for the target integrals. 
		Our analysis provides a general convergence theorem for such certified adaptive quadrature procedures and instantiates it for function values, Jacobians, and Hessians, yielding certified computation of $L^p$, $W^{1,p}$, and $W^{2,p}$ norms. We further show how these ingredients lead to practical certified bounds for PINN interior residuals. Numerical experiments illustrate the accuracy and practical behavior of the proposed methods.
	\end{abstract}
	
	
	
	\section{Introduction}
	Deep neural networks have become a standard tool across engineering, computer science, and applied mathematics, driven by their empirical success and their ability to approximate high-dimensional nonlinear maps \cite{petersenzech}. In recent years, they have also been used extensively for the numerical solution of partial differential equations (PDEs), most prominently in physics-informed neural networks (PINNs) \cite{raissi2019pinns} and related approaches such as deep Galerkin methods and operator learning \cite{sirignano2018dgm,lu2021deeponet,kovachki2021neuraloperator}.
	
	From the viewpoint of numerical analysis, these methods suggest a paradigm that can complement and sometimes replace classical discretizations such as finite element methods (FEM{s}) based on spline spaces \cite{verfurth2013posteriori}. In the neural-network-based setting, the ansatz space is no longer a hand-crafted finite-dimensional space, but instead a trainable class of functions. Crucially, for many applications, it is essential that these methods be complemented with \emph{reliable and certifiable error control}. A natural approach, rooted in PDE theory and numerical analysis, is to measure errors in function space norms and to bound residual quantities through integrals of derivatives. For spline-based methods, such quantities can typically be accessed and computed directly \cite{babuska1978posteriori,ainsworth1997posteriori,verfurth2013posteriori}.

	In comparison, neural networks are much more of a black box: after training, they are essentially available only as \emph{queryable} objects, i.e., one can evaluate a neural network (and, via automatic differentiation, its derivatives) at chosen points, but global information in function space norms remains inaccessible.
	Current practice in neural-network-based PDE solvers relies predominantly on pointwise evaluations on (random) samples, and error assessment is therefore typically phrased in probabilistic terms (e.g.\ via statistical learning theory arguments such as Rademacher complexity bounds), yielding guarantees only ``with high probability'' \cite{shalev2014understanding,neyshabur2018bounds}.
	Moreover, it has been shown that pointwise sampling information alone can be fundamentally insufficient to obtain certified uniform bounds or function space norm information because neural networks can represent highly localized functions; see, e.g., \cite{grohsvoigtlaendergap}.
	
	In this work, we address this challenge by enabling certified and accurate computation of integrals and function space norms of neural network functions. Concretely, we seek to compute (higher-order) Sobolev norms of neural networks. We provide a precise description of our contributions and results in the next section.

\subsection{Our contribution}
    We present AdaQuad, the first framework for computing certified lower and upper
bounds on Lebesgue $L^p$, Sobolev $W^{1,p}$, and $W^{2,p}$ norms of trained deep
neural networks. Our approach combines three ingredients that have not previously
been brought together:
\begin{enumerate}
  \item interval-arithmetic-based enclosures of neural network outputs and their
    derivatives on axis-aligned boxes, building on but going beyond pointwise
    verification methods \cite{Gowal2018,Zhang2018,Singh2019} and local derivative
    certification \cite{Zhang2019RecurJac,Sharifi2024,Laurel2022};

  \item adaptive partition refinement using D\"{o}rfler marking
    \cite{Dorfler1996}---imported from the adaptive finite element literature into
    neural network analysis for the first time; and

  \item quadrature-based aggregation that propagates local certificates to global
    integral bounds, extending the classical self-validating quadrature paradigm
    \cite{CorlissRall1987} from univariate smooth functions to multivariate neural
    network integrands.
\end{enumerate}
Our main theoretical result (Theorem~\ref{thm: main conv}) establishes a general
convergence theorem for certified adaptive quadrature procedures, guaranteeing that
the gap between upper and lower bounds vanishes at a geometric rate under
refinement. We then instantiate this principle for several integral computations:
\begin{itemize}
  \item In Section~\ref{subsec:functionValues}, we derive upper and lower bounds for function
    values. Algorithm~\ref{alg: nn fval bounds} computes certified bounds for neural network
    outputs on boxes, and Corollary~\ref{corr:LpNormsconvergenceOfQuadratureRule} shows that the resulting $L^p$
    norm estimates are accurate. Moreover, we give an efficient way to check whether
    a ReLU network is affine on a box (Proposition~\ref{prop:AffLinPieceCheck}),
    enabling exact integration on such boxes (see Remark~\ref{rem:ExactIntegration}).

  \item In Section~\ref{subsec:jacobianBounds}, we develop an algorithm for pointwise Jacobian
    bounds. Algorithm~\ref{alg: nn jac bounds} provides certifiable upper and lower bounds on
    boxes, and Corollary~\ref{corr:mainW1p} establishes accurate computation of
    $W^{1,p}$ norms.

  \item In Section~\ref{subsec:hessianBounds}, we develop the analogous construction for
    Hessian bounds. Algorithm~\ref{alg: nn hess bounds} yields certified upper and lower
    bounds on boxes, and Corollary~\ref{corr:mainW2p} proves accurate computation of
    $W^{2,p}$ norms.

  \item In Corollary~\ref{corr:mainEnergy}, we show how these ingredients lead to
    practical certified bounds for PINN interior residuals---providing the missing
    computational ingredient for rigorous a posteriori error control in
    neural-network-based PDE solvers
    \cite{Hillebrecht2025,Ernst2025,DeRyck2024}.

  \item In Section~\ref{sec:numerical_experiments}, we present one- and two-dimensional
    numerical experiments, including an elliptic PINN residual example, that
    illustrate certified computation of upper and lower bounds on $L^p$,
    $W^{1,p}$, $W^{2,p}$ norms, and energy norms, confirming the geometric
    convergence rates predicted by our theory.
\end{itemize}

	\begin{figure}[htbp]
		\centering
		\centering
		\includegraphics[width =\textwidth]{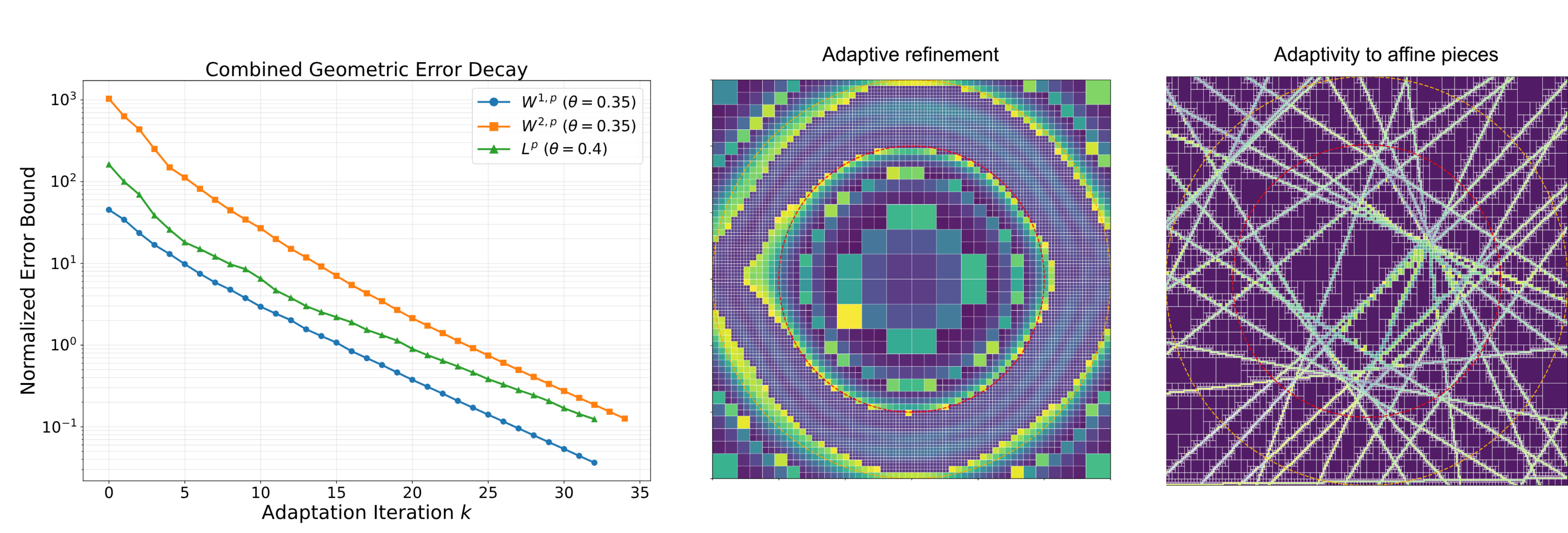}
		\caption{Overview of the adaptive refinement strategies. \textbf{Left:} 
			Certified relative gap of the proposed algorithms for the computation of $L^p$, $W^{1,p}$, and $W^{2,p}$ norms of a smoothed disc function (here $p= 2$). The plotted quantity is our upper bound minus the lower bound, normalized by the final norm estimate. \textbf{Middle:} Heatmap of the local $W^{2,p}$ norm error for a deep architecture with tanh activation function approximating a smoothed disc function. The partition is computed by $\mathrm{AdaQuad}$ (Algorithm~\ref{alg: adaptive integration}). \textbf{Right:} Adaptive partition after 30 refinement steps of $\mathrm{AdaQuad}$ applied to a random ReLU neural network.}
		\label{fig:intro_overview}
	\end{figure}

	\subsection{Related work}

    We summarize related work on \emph{certificates} (i.e., deterministic upper/lower bounds) for neural network function values, derivatives, or derived quantities, and position our contribution relative to the existing literature.

    \paragraph{Neural network verification and bound propagation.}
A large body of work computes certified output enclosures for neural networks.
Interval Bound Propagation (IBP) \cite{Gowal2018,Mirman2018} propagates input
intervals layer-by-layer, while CROWN and auto\_LiRPA \cite{Zhang2018,Xu2020}
propagate linear bounds backward for tighter enclosures. DeepPoly
\cite{Singh2019} employs abstract interpretation with polyhedra-interval
domains, and $\alpha,\beta$-CROWN \cite{Wang2021} achieves state-of-the-art
performance by combining optimized bound propagation with branch-and-bound.
Kapela et al.\ \cite{Kapela2024} import doubleton and affine arithmetic from
validated numerics to mitigate the wrapping effect in IBP. All these methods
target pointwise output properties---typically adversarial robustness---and do
not address derivatives or integral quantities.
 
\paragraph{Certified derivative bounds.}
RecurJac \cite{Zhang2019RecurJac} computes element-wise certified Jacobian bounds
over local input regions via recursive refinement. Shi et al.\ \cite{Shi2022}
generalize this to compute local Lipschitz constants via bound propagation on the
chain-rule backward graph. LipSDP \cite{Fazlyab2019} formulates global Lipschitz
estimation as a semidefinite program but produces only a single scalar bound. For
second-order derivatives, Entesari et al.\ \cite{Entesari2024} derive
compositional curvature bounds propagated layer-by-layer, while Sharifi and
Fazlyab \cite{Sharifi2024} propose derivative-preserving reachability analysis
providing analytical bounds on both gradients and Hessians---the closest precursor
to our local derivative enclosures. Laurel et al.\ \cite{Laurel2022} present a
general framework for abstractly interpreting higher-order automatic
differentiation, yielding certified first- and second-derivative bounds via
dual-number abstractions. All of these methods provide pointwise or region-wise
derivative certificates but do not aggregate local bounds into global integral
quantities such as Sobolev norms.
 
\paragraph{Branch-and-bound verification and domain splitting.}
Branch-and-bound (BaB) methods partition the problem domain to tighten
verification bounds. Bunel et al.\ \cite{Bunel2020} unify NN verification under a
BaB/MIP framework with activation-based splitting. ReluVal \cite{Wang2018ReluVal}
uses input-domain bisection with interval arithmetic. MN-BaB \cite{Ferrari2022}
and GCP-CROWN \cite{Zhang2022GCP} integrate multi-neuron relaxations and cutting
planes within BaB. GenBaB \cite{Shi2024GenBaB} extends BaB to general nonlinear
activations. Our adaptive refinement shares the surface-level mechanism of domain
splitting but pursues a structurally different objective: BaB verification prunes
verified subdomains and seeks early termination, whereas our framework aggregates
certified bounds across all subdomains to compute integrals. Boetius et al.\
\cite{Boetius2025} is the closest structural analog, aggregating bounds over
partitions to compute probability certificates, but does not compute function
norms, uses different splitting heuristics, and does not provide AFEM-style convergence
theory.
 
\paragraph{Validated numerics and rigorous integration.}
Self-validating quadrature combining interval arithmetic with adaptive subdivision
was pioneered by Corliss and Rall \cite{CorlissRall1987} for univariate functions.
Petras \cite{Petras2002} and Johansson \cite{Johansson2018} extend this to
piecewise-analytic integrands with arbitrary-precision ball arithmetic. Taylor
model methods \cite{BerzMakino1998,BerzMakino2003} provide a multivariate
framework for rigorous integration via polynomial-remainder pairs, with
demonstrated feasibility up to approximately eight variables. Foundational
references include Moore et al.\ \cite{Moore2009}, Tucker \cite{Tucker2011}, and
Rump \cite{Rump2010}. Computer-assisted proofs in PDEs \cite{Nakao2019} use validated numerics to prove existence of PDE solutions near finite element approximations. A key ingredient in these proofs is the certified computation of Sobolev norms of the approximate solution; however, existing methods rely on the piecewise-polynomial structure of finite element spaces and do not extend to neural network ansatz functions. Our
framework brings validated quadrature into the neural network setting, exploiting
the compositional and piecewise-affine structure of trained networks and using
D\"{o}rfler marking \cite{Dorfler1996} rather than simple bisection for optimal
adaptive refinement.
 
\paragraph{PINN error estimation.}
A priori error analysis \cite{DeRyck2024,Lorenz2024,Doumeche2025} establishes
convergence rates and approximation bounds but cannot certify a specific trained
network. A posteriori methods are more directly relevant: Hillebrecht and Unger
\cite{Hillebrecht2022,Hillebrecht2025} derive rigorous prediction error bounds via
semigroup theory but evaluate residual norms only at discrete points. Ernst et
al.\ \cite{Ernst2025} use Riesz representations on simpler domains, requiring
auxiliary finite element solves. Berrone et al.\ \cite{Berrone2022} define
reliable estimators for variational PINNs with mesh-based test functions. Most
closely related, $\partial$-CROWN \cite{Eiras2024} extends CROWN bound propagation
to compute certified worst-case ($L^\infty$) PINN residual bounds over continuous
domains. Mukherjee et al.\ \cite{Mukherjee2026} establish generalization bounds
connecting residual control to solution error. Tanaka and Yatabe \cite{Tanaka2026}
use interval arithmetic to verify PINN sub/super-solutions for ODEs but address
pointwise differential inequalities, not integral norms. Our work provides the
missing ingredient: certified $L^p$ and Sobolev norms of PINN residuals, computed
directly from the network structure via interval arithmetic and adaptive
quadrature, without requiring auxiliary PDE solves, mesh-based test functions, or
restriction to pointwise bounds.
 
\paragraph{Function-space norms of neural networks.}
Theoretical works define function-space norms---Barron norms \cite{E2022}, spectral
Barron spaces \cite{LiaoMing2023}---and prove approximation rates in Sobolev norms
\cite{Guhring2020,SiegelXu2023,DeVore2021}, but none provides algorithms for
computing these norms for trained networks. Parameter-space measures---spectral
norms \cite{Bartlett2017}, path norms \cite{Neyshabur2015,Gonon2024}---are
computable but characterize weight matrices, not the input-output function in $L^p$
or $W^{k,p}$. To the best of our knowledge, no prior work computes certified
function-space norms of trained deep neural networks with deterministic guarantees.

	\section{Preliminaries}
	
	In this section, we introduce preliminary notions that are used throughout the manuscript.

	\subsection{Function spaces and measures}
	All integrals in this work are taken with respect to the Lebesgue measure.
	For a measurable set $A \subset \mathbb{R}^d$, we denote by $\vol(A)$ the Lebesgue measure of $A$.
	The main objects of interest are the Lebesgue and Sobolev norms defined below.
	
	\begin{definition}\label{def:lebesgueSob}
		Let $1 \leq p < \infty$ and $d, m, k \in \mathbb{N}$. For $\Omega\subset \mathbb{R}^d$, we define the Lebesgue space of vector-valued functions $L^p(\Omega; \mathbb{R}^m)$ as the set of measurable maps $\Phi: \Omega \to \mathbb{R}^m$
		with finite norm,
		\begin{equation}
			\label{def: sobolev norm}
			||\Phi ||_{L^p(\Omega; \mathbb{R}^m)} \coloneqq \left( \sum_{i=1}^m \int_\Omega |\Phi_i(x)|^p \,dx\right)^{1/p}
		\end{equation}
		where functions that agree almost everywhere are identified.
		We define the Sobolev space of vector-valued functions $W^{k,p}(\Omega; \mathbb{R}^m)$ as the set of $\Phi \in L^p(\Omega; \mathbb{R}^m)$ such that, for every multi-index $\alpha \in \mathbb{N}_0^d$ with $|\alpha| \leq k$, the weak derivative $D^\alpha\Phi=(D^\alpha\Phi_1,\dotsc, D^\alpha\Phi_m)$ belongs to $L^p(\Omega; \mathbb{R}^m)$. We define the Sobolev norm
		\begin{equation}
			||\Phi||_{W^{k,p}(\Omega; \mathbb{R}^m)}\coloneqq\left(\sum_{i=1}^m  \sum_{|\alpha|\leq k}\int_\Omega |D^\alpha\Phi_i(x)|^p\, dx\right)^{1/p}
		\end{equation}
		where $|\alpha|=\alpha_1 + \dots + \alpha_d$.
	\end{definition}
	
	\begin{remark}
		For $1 \leq p < \infty$, we use the convention that $W^{0,p}(\Omega; \mathbb{R}^m)=L^p(\Omega; \mathbb{R}^m)$ and that $D^\alpha$ is the identity map for $\alpha = 0 \in \mathbb{N}_0^d$.
	\end{remark}

	\subsection{Neural networks}
	
	We consider standard feedforward neural networks. To fix notation, we include the following definition from \cite{petersenzech}.
	
	\begin{definition}[Feedforward neural network]\label{def: neural network}
		A feedforward neural network $\Phi:\mathbb{R}^{d_0}\to\mathbb{R}^{d_{L+1}}$ consists of $L\in\mathbb{N}$ hidden layers, layer widths $d_0,\dots,d_{L+1}\in\mathbb{N}$, an activation function $\sigma:\mathbb{R}\to\mathbb{R}$, and weight matrices
		$W^{(\ell)}\in\mathbb{R}^{d_{\ell+1}\times d_\ell}$ and bias vectors $b^{(\ell)}\in\mathbb{R}^{d_{\ell+1}}$ for $\ell = 0, \dots, L$
		such that
		\begin{align*}
			{x}^{(0)} &:= x, \nonumber \\
			z^{(\ell)} &:= W^{(\ell-1)}x^{(\ell-1)}+b^{(\ell-1)}, \quad &\text{(pre-activation)} \\
			x^{(\ell)} &:= \sigma(z^{(\ell)}),  \quad &\text{(activation)} \\
			x^{(L+1)} &:= W^{(L)}x^{(L)}+b^{(L)}\nonumber, 
		\end{align*}
		we have
		\[
		\Phi(x)=x^{(L+1)}\qquad\text{for all }x\in\mathbb{R}^{d_0}.
		\]
		We call $L$ the depth, $w_\Phi:=\max_{\ell=1,\dots,L} d_\ell$ the {width}, and $(\sigma; d_0,\dots,d_{L+1})$ the {architecture} of the neural network~$\Phi$.
	\end{definition}
	
	When bounding derivatives of a neural network in Section~\ref{sec: results}, we iterate backward through the network. This motivates the following definition of a tail network.
	
	\begin{definition}\label{def: tail nn}
		Let $\Phi: \mathbb{R}^{d_0} \to \mathbb{R}^{d_{L+1}}$ be a neural network. For $\ell = 0,\dotsc, L$, define maps
		$$ T_\ell: \mathbb{R}^{d_\ell} \to \mathbb{R}^{d_{\ell+1}}, \quad  T_\ell (x) = W^{(\ell)}x + b^{(\ell)},$$
		with $W^{(\ell)}$ and $b^{(\ell)}$ from $\Phi$. For $\ell =0,\dotsc, L$, we define the tail neural networks $\Phi^{(\ell)}: \mathbb{R}^{d_\ell} \to \mathbb{R}^{d_{L+1}}$ by
		$$ \Phi^{(\ell)}(x) = T_{L}\circ \sigma \circ T_{L-1}\circ  \dots \circ \sigma \circ T_\ell(x). $$
		In particular, we have $\Phi^{(0)}=\Phi$ and $\Phi^{(\ell)}(x^{(\ell)}) = \Phi(x)$.
	\end{definition}
	
	\begin{remark}
		For a neural network $\Phi: \mathbb{R}^{d_0} \to \mathbb{R}^{d_{L+1}}$, it is useful to view the layers as functions: the pre-activations are $z^{(\ell)}(u)=W^{(\ell-1)}x^{(\ell-1)}(u)+b^{(\ell-1)}$, and the activations are $x^{(\ell)}(u)=\sigma(z^{(\ell)}(u))$ for $\ell=1,\dots,L$, $u \in \mathbb{R}^{d_0}$, with $x^{(0)}(u)=u$.
	\end{remark}

	The Jacobian and Hessian of a neural network involve diagonal matrices with activation derivatives evaluated at pre-activations on the diagonal. We define these matrices below.
	
	\begin{definition}\label{def: diag D}
		Let $\Phi : \mathbb{R}^{d_0} \to \mathbb{R}^{d_{L+1}}$ be a neural network with $L$ hidden layers and activation function $\sigma \colon \mathbb{R} \to \mathbb{R}$.
		For each $\ell=1,\dots,L$ and each vector $u=(u_1,\dots,u_{d_\ell})\in\mathbb{R}^{d_\ell}$, define the diagonal matrix $D^{(\ell)}(u)\in\mathbb{R}^{d_\ell\times d_\ell}$ by
		\[
		D^{(\ell)}(u) = \operatorname{diag}\!\big(\sigma'(u_1),\dots,\sigma'(u_{d_\ell})\big),
		\]
		i.e.,
		\[
		D^{(\ell)}(u)_{ii} = \sigma'(u_i), \qquad D^{(\ell)}(u)_{ij}=0 \ \ (i\neq j),
		\]
		for $u \in \mathbb{R}^{d_\ell}$.
	\end{definition}

	\begin{definition}\label{def:ReLUActivationFunction}
		If $\sigma : \mathbb{R} \to \mathbb{R}$ is defined by $\sigma(x) = \max\{x, 0\}$, we call it a rectified linear unit (ReLU) activation function.
	\end{definition}

	\subsection{Interval analysis}\label{subsec: interval analysis}
	Here, we introduce definitions from interval analysis (see, e.g., \cite{moore2009introduction,rump2010verification}) that are required in the subsequent sections. We begin with the system of intervals.

	\begin{definition}[\cite{rump2010verification}]
		We call
		\begin{equation}
			\mathbb{IR} = \{ [\underline{X}, \overline{X}]\subset \mathbb{R} \, : \,   \underline{X}\leq \overline{X}\}
		\end{equation}
		the system of intervals. For $n,m\in\mathbb{N}$, elements of $\mathbb{IR}^n$ are called interval vectors (or boxes), and elements of $\mathbb{IR}^{n \times m}$ are called interval matrices.
	\end{definition}
	When computing with intervals, we replace arithmetic operations on numbers by the corresponding interval operations.
	\begin{definition}[\cite{rump2010verification}]\label{def:IntArit}
		Given two intervals $X,Y \in \mathbb{IR}$, the operations $\circ \in \{+,-,\cdot,/\}$ are defined by
		\begin{equation}
			X \circ Y = \{ x \circ y \, : \, x\in X, y \in Y\},
		\end{equation}
		where $0\notin Y$ is assumed in the case of division.
	\end{definition}
	Interval arithmetic can be expressed purely in terms of endpoints.

	\begin{proposition}[\cite{moore1966interval}]\label{prop: endpoint formulas}
		Given two intervals $X= [\underline{X}, \overline{X}]$ and $Y=[\underline{Y}, \overline{Y}]$, we have:
		\begin{itemize}
			\item For addition
			$$ X+Y = [\underline{X} + \underline{Y}, \,\overline{X} + \overline{Y}]$$
			
			\item For subtraction
			$$ X-Y = [\underline{X}-\overline{Y},\, \overline{X}-\underline{Y}] $$
			
			\item For multiplication, let $S=\{\underline{X}\underline{Y}, \underline{X}\overline{Y}, \overline{X}\underline{Y}, \overline{X}\overline{Y}\}$; then
			$$ X \cdot Y =  [\min S, \max S] $$
			
			\item For division (assuming $0 \notin [\underline{Y}, \overline{Y}]$),
			$$
			X/Y = [\underline{X}, \overline{X}]/[\underline{Y}, \overline{Y}] 
			= [\underline{X}, \overline{X}]\cdot[1/\overline{Y}, 1/\underline{Y}] .
			$$
		\end{itemize}
	\end{proposition}
	\begin{remark}\label{rem:IntervalLinAlg}
		Products of interval matrices with interval matrices (or interval vectors) are computed by replacing the arithmetic operations in real matrix products by interval arithmetic operations. In this way, endpoint formulas for interval matrix products follow from Proposition~\ref{prop: endpoint formulas}.
	\end{remark}
	
	Next, we define the width of an interval.
	\begin{definition}[\cite{moore2009introduction}]
		Let $X = [\underline{X}, \overline{X}]$ be an interval. Then
		\begin{equation}
			w(X) = \overline{X} - \underline{X}
		\end{equation}
		is called the width of the interval $X$. Let $n,m\in \mathbb{N}$. For an interval vector $X=(X_1, \dotsc,X_n)\in \mathbb{IR}^n$ we define
		$$ w(X) = \max_{i=1,\dotsc, n} w(X_i) $$
		and for an interval matrix $A\in \mathbb{IR}^{n\times m}$ we define
		$$ w(A) = \max_{\substack{i=1,\dots,n \\ j=1,\dotsc,m}} w(A_{ij}), $$
		where the $A_{ij}$'s are the entries of the interval matrix.
	\end{definition}
	
	\begin{remark}[\cite{moore2009introduction}]\label{rem: width linearity}
		For $X,Y\in\mathbb{IR}$ the width satisfies
		$$ w(aX + bY) = |a|w(X) + |b|w(Y) $$
		for $a,b\in \mathbb{R}$.
	\end{remark}
	
	Next, we define the magnitude, absolute value and powers of an interval.
	
	\begin{definition}[\cite{rump2010verification}, \cite{moore2009introduction}]\label{def: mag abs power}
		Let $X= [\underline{X}, \overline{X}]$ be an interval. We define:
        \begin{itemize}
            \item The magnitude of the interval $X$ by
            $$ \mathrm{mag}(X) = \max\{|\underline{X}|, \,|\overline{X}|\}. $$
            We have $\mathrm{mag}(X) = \max\{|x| \, : x \in X\}.$
            \item The absolute value of $X$ by
            $$ |X| =  \begin{cases}
				[\underline{X}, \overline{X}], & \text{if } \underline{X}\geq 0 \\
				[0, \max\{-\underline{X}, \overline{X}\}] & \text{if }\underline{X}<0, \overline{X}\geq 0 \\
				[-\overline{X}, -\underline{X}] & \text{if }\underline{X}<0,\overline{X}< 0.
			\end{cases}$$
            This is the united interval extension of the absolute value $x \mapsto |x|$ on $\mathbb{R}$.
            \item The interval $X$ to the power of $\gamma \geq 0$ by
            $$ X^\gamma = [\underline{X}^{\gamma}, \overline{X}^{\gamma}]  $$
            provided $0 \leq \underline X$ (so that $x\mapsto x^\gamma$ is well-defined and monotone on $X$). This is the united interval extension of the $x \mapsto x^\gamma$ on $\mathbb{R}$.
        \end{itemize}
	\end{definition}

	Next, we define a norm for interval matrices. 
	
	\begin{definition}[\cite{moore2009introduction}]\label{def: interval matrix norm}
		Let $m,n\in \mathbb{N}$ and $A \in \mathbb{IR}^{m\times n}$. We define
		\begin{equation}
			||A|| = \max_{i=1,\dotsc, m} \sum_{j=1}^n\mathrm{mag}(A_{ij}).
		\end{equation}
		Moreover, this norm is an interval extension of the maximum row-sum norm $||.||_{\infty,\infty}$ and for every real matrix $B$ contained in the interval matrix $A$, we have $||B||_{\infty, \infty} \leq ||A||$.
	\end{definition}

	We state some very useful inequalities for the width and matrix norm.
	
	\begin{proposition}
		Let $m,n,d\in \mathbb{N}$. Then
		\begin{equation}\label{eq: width subadditivity}
			w(A+B) \leq w(A) + w(B)
		\end{equation}
		and
		\begin{equation}\label{eq: matrix width inequality}
			w(AB) \leq ||A||w(B) + ||B||w(A)
		\end{equation}
		as well as
		\begin{equation}\label{eq: matrix magnitude inequality}
			||AB|| \leq ||A|| \, ||B||
		\end{equation}
		holds for every $A \in \mathbb{IR}^{m\times n}$ and $B \in \mathbb{IR}^{n \times d}$.
	\end{proposition}

    Let us define the range of a function.

    \begin{definition}
        Let $d,m\in \mathbb{N}$ and $f: \mathbb{R}^d  \to \mathbb{R}^m$. We define the range of $f$ on $A$
        as
        $$ f(A) = \{f(x) \, : \, x \in A\} $$
        for $A \subset \mathbb{R}^d$.
    \end{definition}
	
	When we bound the range of a function over an interval vector, we obtain for each output component a lower and an upper bound, and these two numbers together define an interval.
    
	\begin{definition}
		A function that maps intervals, interval vectors, interval matrices, etc. to intervals, interval vectors, interval matrices, etc. is called an interval function.
	\end{definition}

    If an interval function agrees with the true range of a function, we call it a united interval extension

    \begin{definition}[\cite{moore2009introduction}]
        Let $d,m\in \mathbb{N}$, $f:\mathbb{R}^d \to \mathbb{R}^m$ and $F: \mathbb{IR}^d  \to \mathbb{IR}^m$. If
        $$ f(K) = F(K) $$
        for $K \in \mathbb{IR}^d$, we call $F$ the united interval extension of $f$.
    \end{definition}
	
	We will mostly encounter interval functions mapping interval vectors to intervals, interval vectors, or interval matrices. A desirable property is preservation of set inclusion, called inclusion isotonicity.
	
	\begin{definition}[\cite{moore2009introduction}]
		Let $d,m\in \mathbb{N}$ and $F: \mathbb{IR}^d  \to \mathbb{IR}^m$. If, for all $K',K\in\mathbb{IR}^d$ with $K' \subset K$ we have $F(K') \subset F(K)$, then $F$ is called inclusion isotonic.
	\end{definition}
	
	If an interval function is inclusion isotonic and contains the range of a non-interval function $f$, it is called an interval enclosure of $f$.
	
	\begin{definition}[\cite{moore2009introduction}]\label{def: int enclosure}
		Let $d,m \in \mathbb{N}$ and $\Phi: \mathbb{R}^d \to \mathbb{R}^m$. We call a function $F_\Phi: \mathbb{IR}^d \to \mathbb{IR}^m$ an interval enclosure of $\Phi$ if it is inclusion isotonic and
		$$ \Phi(K) \subset F_\Phi(K) $$
		for all $K\in\mathbb{IR}^d$.
	\end{definition}
	
	We now define the concept of an interval extension. This is an interval function that agrees with a non-interval function on singletons.
	
	\begin{definition}[\cite{moore2009introduction}]\label{def: int extension}
		Let $d,m \in \mathbb{N}$ and $\Phi: \mathbb{R}^d \to \mathbb{R}^m$. We call a function $F_\Phi: \mathbb{IR}^d \to \mathbb{IR}^m$ an interval extension of $\Phi$ if
		$$ \Phi(\{x\}) = F_\Phi(\{x\}) $$
		for all $x \in \mathbb{R}^d$.
	\end{definition}
	
	An inclusion-isotonic interval extension is also an interval enclosure. This is the Fundamental Theorem of Interval Analysis.
	
	\begin{theorem}[\cite{moore1966interval}]\label{thm: fund thm of IA}
		Let $d,m \in \mathbb{N}$ and $\Phi: \mathbb{R}^d \to \mathbb{R}^m$. If $F_\Phi: \mathbb{IR}^d  \to \mathbb{IR}^m$ is an inclusion-isotonic interval extension of $\Phi$, then it is an interval enclosure of $\Phi$, i.e.
		$$ \Phi(K) \subset F_\Phi(K) $$
		for all $K\in\mathbb{IR}^d$.
	\end{theorem}
	
	\begin{remark}
		By the Fundamental Theorem of Interval Analysis, every inclusion-isotonic interval extension of a function $\Phi$ is an interval enclosure. However, the converse is false, since an interval enclosure need not agree with $\Phi$ on singletons.
	\end{remark}

	The concept of Hölder continuity will be crucial for studying the convergence of adaptive quadrature in the next section.
	
	\begin{definition}
		Let $d,m \in \mathbb{N}$, $\Omega \in \mathbb{IR}^d$, and $F: \mathbb{IR}^d  \to \mathbb{IR}^m$. We call $F$ Hölder continuous on $\Omega$ with exponent $\gamma \in (0,1]$ and constant $C>0$ if
		\begin{equation}\label{eq: hölder continuity}
			w(F(K)) \leq C w(K)^\gamma
		\end{equation}
		for all $K\in\mathbb{IR}^d$ with $K\subset \Omega$. In the case of $\gamma = 1$ we call $F$ Lipschitz continuous.
	\end{definition}
	
	\begin{remark}\label{rem: hölder continuity}
		For a function $\Phi: \Omega \to \mathbb{R}^m$, classical Hölder continuity
		$$ \|\Phi(x) - \Phi(y)\|_\infty \leq C \|x-y\|_\infty^\gamma $$
		for all $x,y \in \Omega$ implies \eqref{eq: hölder continuity} for the map $K \mapsto \Phi(K)$ if it has values in $\mathbb{IR}^m$.
	\end{remark}

    Hölder continuity allows us to bound the width of an interval enclosure by the width of the box. In Algorithm~\ref{alg: hölder strategy} we use this as a technique to control the local error.

	\section{Adaptive quadrature}
	
	Our aim is to design an adaptive quadrature that {yields, after} $n\in \mathbb{N}$ steps{,} two values $\mathrm{Q}_n, \eta_n \in \mathbb{R}$ for $f:\Omega \to \mathbb{R}$ with
	\begin{equation}
		\int_\Omega f(x) \, dx = \mathrm{Q}_n + \mathrm{R}_n
	\end{equation}
	where $\eta_n$ is {an upper bound on} the error $|\mathrm{R}_n|$ and converges to zero, i.e.
	\begin{equation}\label{eq: goal}
		|\mathrm{R}_n| \leq \eta_n \quad \text{and} \quad \eta_n \to 0 \text{  for  }n \to \infty.
	\end{equation}
	To bound the error of the quadrature, note that we have
	$$\int_\Omega f(x) \, dx \in f(\Omega) \vol(\Omega) $$
	for a continuous function $f$. If the quadrature lies within the interval $f(\Omega)\vol(\Omega)$, its width rigorously bounds the error.
	
	\begin{definition}\label{def: quadrature rule}
		Let $d\in \mathbb{N}$ and $\Omega \in \mathbb{IR}^d$. We call a function $\mathcal{I}$ with arguments $f:\Omega \to \mathbb{R}$ and interval vector $K \subset \Omega$ in $\mathbb{IR}^d$ mapping into $\mathbb{R}$, defined by
		$$ \mathcal{I}(f,K) = \sum_{i=1}^N w_i f(x_i), $$
		with parameters $w_1,\dotsc,w_N \in \mathbb{R}$ and $x_1,\dotsc,x_N \in K$ a quadrature rule.
		We say that $\mathcal{I}$ is exact for a class of functions $\mathcal{F}$ if
		$$ \mathcal{I}(f,K) = \int_K f(x)\,dx \qquad \text{for all } f \in \mathcal{F}. $$
		If $\mathcal{P}_m$ denotes the space of polynomials on $\mathbb{R}^d$ of total degree
		at most $m$, then we say that $\mathcal{I}$ has {degree of exactness} (or {precision}) $m$
		if it is exact for $\mathcal{P}_m$ but not exact for $\mathcal{P}_{m+1}$.
	\end{definition}
	
	If $\mathcal{I}$ is a positive-weight quadrature rule {that} is exact for constants, then $\mathcal{I}(f,\Omega)\in f(\Omega)\vol(\Omega)$ and we obtain the width
	$$\tilde{\eta}_\Omega = w\big(f(\Omega)\vol(\Omega) \big) = w(f(\Omega))\vol(\Omega)$$
	as an upper bound for $|\mathrm{R}_n|$.
	The computation of
	$$ w(f(\Omega)) = \overline{f(\Omega)} - \underline{f(\Omega)} $$
	requires knowledge of the maximum $\overline{f(\Omega)}$ and minimum $\underline{f(\Omega)}$ of $f$ on $\Omega$. Instead of applying computationally costly global optimization methods, we seek an interval enclosure $F: \mathbb{IR}^d \to  \mathbb{IR}$ of $f$. This yields
	$$\int_\Omega f(x)\, dx \in F(\Omega) \vol(\Omega) $$
	and 
	$$\eta_\Omega = w(F(\Omega))\vol(\Omega)$$
	satisfying $|\mathrm{R}_n| \leq \eta_\Omega$.
	The interval-valued map $F$ contains the range of $f$ and therefore yields rigorous error bounds. This motivates the following definition of a problem instance.
	
	\begin{definition}
		We call the triple
		\begin{equation}
			P = (f, F, \Omega)
		\end{equation}
		consisting of a function $f:\Omega \to \mathbb{R}$, an interval enclosure $F:\mathbb{IR}^d \to \mathbb{IR}$ of $f$ and a domain $\Omega\in \mathbb{IR}^d$ a problem instance.
	\end{definition}
	
	We have seen above how a suitable interval-valued function $F$ can yield an upper bound $\eta_\Omega$ to the error of the quadrature. Let us now turn to the second part of Equation \eqref{eq: goal} and motivate how to achieve convergent error decay. To this end, let $n \in \mathbb{N}$ and $\mathcal{P}_n\subset \mathbb{IR}^d $ be a partition of $\Omega$. We then have
	$$ \int_\Omega f(x) \, dx \in \sum_{K \in \mathcal{P}_n} F(K) \vol(K)  $$
	with the {right-hand side} having width $\sum_{K \in \mathcal{P}_n} w(F(K))\vol(K)$. Setting 
	\begin{equation}
		\eta_K = w(F(K))\vol(K)
	\end{equation}
	and
	\begin{equation}\label{eq: partition error bound}
		\eta_n =  \sum_{K \in \mathcal{P}_n} \eta_K
	\end{equation}
	yields $|\mathrm{R}_n| \leq \eta_n$, provided
	$$ \mathrm{Q}_n = \sum_{K \in \mathcal{P}_n} \mathcal{I}(f,K) $$
	with local quadratures satisfying $\mathcal{I}(f,K) \in F(K)\vol(K)$. Considering Equation \eqref{eq: partition error bound}, we hope to create partitions $\mathcal{P}_n$ such that the local error bounds $\eta_K$ converge to zero as the boxes get smaller. This motivates the definition of a state instance.
	
	\begin{definition}
		We call the triple
		\begin{equation}
			S = (\mathcal{P}, \mathcal{Q}, \mathcal{E})
		\end{equation}
		a state instance. The components are:
		\begin{itemize}
			\item A partition $\mathcal{P} \subset \mathbb{IR}^d$ of $\Omega$ into interval vectors.
			\item A set $\mathcal{Q} \subset \mathcal{P}\times \mathbb{R}$ of pairs $\mathcal{Q}=\{(K,Q_K)\,:\,K\in\mathcal{P}\}$ assigning a quadrature value $Q_K$ to each $K\in\mathcal{P}$.
			\item A set $\mathcal{E} \subset \mathcal{P}\times \mathbb{R}_{\geq 0}$ of pairs $\mathcal{E}=\{(K,\eta_K)\,:\,K\in\mathcal{P}\}$ assigning an error bound $\eta_K$ to each $K\in\mathcal{P}$.
		\end{itemize}
		We write $1^T\mathcal{Q} := \sum_{(K,Q_K)\in \mathcal{Q}} Q_K$ and $1^T\mathcal{E} := \sum_{(K,\eta_K)\in \mathcal{E}} \eta_K$.
	\end{definition}
	
	We can create partitions either by uniformly subdividing the problem domain into smaller boxes or iteratively refining a partition by choosing only a subset of boxes to subdivide. Inspired by adaptive finite element methods (see, e.g. \cite{verfurth2013posteriori}) we follow the latter approach of adaptively generating partitions. This introduces an additional selection process which is called marking. We will collect the local quadrature, marking and refinement rule in an algorithm instance.
	
	\begin{definition}
		We call the triple
		\begin{equation}
			A = (\mathcal{I}, \mathcal{M}, \mathcal{R})
		\end{equation}
		an algorithm instance. The components are:
		\begin{itemize}
			\item A quadrature rule $\mathcal{I}$ with arguments $f:\Omega \to \mathbb{R}$ and $K \in \mathbb{IR}^d$ with $K\subset \Omega$ which returns a real number.
			\item A marking rule $\mathcal{M}$ with argument $\mathcal{E}\subset \mathcal{P} \times \mathbb{R}_{\geq 0}$ that returns a subset $\widetilde{\mathcal{P}}\subset \mathcal{P}$.
			\item A refinement rule $\mathcal{R}$ taking as argument a box $K\in \mathbb{IR}^d$ and returning a partition $\mathcal{P}_K \subset \mathbb{IR}^d$ of $K$ into interval vectors.
		\end{itemize}
	\end{definition}
	
	We initialize the first state as the trivial partition and the corresponding error bound and quadrature.
	
	\vspace{1em}
	\SetKwInOut{Input}{Input}
	\SetKwInOut{Output}{Output}
	\begin{algorithm}[H]
		\LinesNotNumbered
		\DontPrintSemicolon
		\caption{$\mathrm{Init}_A(P)$}
		\label{alg: initialize state}
		\Input{Problem instance $P=(f,F, \Omega)$; algorithm instance $A=(\mathcal{I}, \mathcal{M}, \mathcal{R})$}
		\Output{State instance $S_0$}
		
		\BlankLine
		$\mathcal{P}_0 = \{\Omega\}$ \quad (trivial partition)\;
		$\mathcal{Q}_0 = \{(\Omega, \mathcal{I}(f,\Omega))\}$ \quad (initial quadrature)\; 
		$\mathcal{E}_0 = \{(\Omega, w(F(\Omega))\vol(\Omega))\}$ \quad (initial error bound)\;
		$S_0 = (\mathcal{P}_0, \mathcal{Q}_0, \mathcal{E}_0)$\;
		\Return $S_0$\;
	\end{algorithm}
	\vspace{1em}

	Based on a given state $S_n$ at iteration $n\in \mathbb{N}$, we can compute the next state by applying marking and refinement strategies followed by updating the partition, quadratures, and error bounds.
	
	\vspace{1em}
	\SetKwInOut{Input}{Input}
	\SetKwInOut{Output}{Output}
	\begin{algorithm}[H]
		\LinesNotNumbered
		\DontPrintSemicolon
		\caption{$\mathrm{Step}_A(S_n, P)$}
		\label{alg: step}
		\Input{State instance $S_n=(\mathcal{P}_n, \mathcal{Q}_n, \mathcal{E}_n)$; problem instance $P=(f,F, \Omega)$; algorithm instance $A=(\mathcal{I}, \mathcal{M}, \mathcal{R})$}
		\Output{State instance $S_{n+1}$}
		
		\BlankLine
		$\widetilde{\mathcal{P}}_n = \mathcal{M}(\mathcal{E}_n)$ \quad (marked partition elements)\;
		$\widetilde{\mathcal{Q}}_n = \{(K,Q_K)\in\mathcal{Q}_n : K\in\widetilde{\mathcal{P}}_n\}$ \quad (marked quadratures)\;
		$\widetilde{\mathcal{E}}_n = \{(K,\eta_K)\in\mathcal{E}_n : K\in\widetilde{\mathcal{P}}_n\}$ \quad (marked error bounds)\;
		$\mathcal{P}_{n+1, U} = \mathcal{P}_n \setminus \widetilde{\mathcal{P}}_n$ \quad (unmarked partition elements)\;
		$\mathcal{Q}_{n+1, U} = \mathcal{Q}_n \setminus \widetilde{\mathcal{Q}}_n$ \quad (unmarked quadratures)\;
		$\mathcal{E}_{n+1, U} = \mathcal{E}_n \setminus \widetilde{\mathcal{E}}_n$ \quad (unmarked error bounds)\;
		$\mathcal{P}_{n+1, R} = \emptyset$ \quad (initialize refined partition)\;
		\For{$K \in \widetilde{\mathcal{P}}_n$}{
			$\mathcal{P}_K = \mathcal{R}(K)$ \quad (refine marked partition element)\;
			$\mathcal{P}_{n+1, R} = \mathcal{P}_{n+1, R} \cup \mathcal{P}_K$\quad (update refined partition)\;
		}
		$\mathcal{Q}_{n+1, R} = \{(K, \mathcal{I}(f,K))\}_{K \in \mathcal{P}_{n+1,R}}$ \quad (update refined quadratures)\;
		$\mathcal{E}_{n+1, R} = \{(K, w(F(K))\vol(K))\}_{K \in \mathcal{P}_{n+1,R}}$\quad (update refined error bounds)\;
		$\mathcal{P}_{n+1} = \mathcal{P}_{n+1, R} \cup \mathcal{P}_{n+1, U}$ \quad (next partition)\;
		$\mathcal{Q}_{n+1} = \mathcal{Q}_{n+1, R} \cup \mathcal{Q}_{n+1, U}$\quad (next quadratures)\;
		$\mathcal{E}_{n+1} = \mathcal{E}_{n+1, R} \cup \mathcal{E}_{n+1, U}$\quad (next error bounds)\;
		$S_{n+1} = (\mathcal{P}_{n+1}, \mathcal{Q}_{n+1}, \mathcal{E}_{n+1})$ \quad (next state instance)\;
		\Return $S_{n+1}$\;
	\end{algorithm}
	\vspace{1em}
	
	Combining the initialization and the step algorithm, we obtain the adaptive quadrature method.
	
	\vspace{1em}
	\SetKwInOut{Input}{Input}
	\SetKwInOut{Output}{Output}
	\begin{algorithm}[H]
		\LinesNotNumbered        
		\DontPrintSemicolon
		\caption{$\mathrm{AdaQuad}_n(P,A)$}
		\label{alg: adaptive integration}
		\Input{Problem instance $P=(f,F, \Omega)$; algorithm instance $A=(\mathcal{I}, \mathcal{M}, \mathcal{R})$}
		\Output{Approximation $\mathrm{Q}_n$ to integral $\int_\Omega f(x)dx$ with error bound $\eta_n$.}
		
		\BlankLine
		$S_0 = \mathrm{Init}_A(P)$ \quad (initialize state instance)\;
		\For{$t =0,\dotsc, n-1$}{
			$S_{t+1} = \mathrm{Step}_A(S_t,P)$ \quad (update state instance)\;
		}
		$(\mathcal{P}_n,\mathcal{Q}_n,\mathcal{E}_n) = S_n$\;
		$\mathrm{Q}_n = 1^T \mathcal{Q}_n$\quad (sum local quadratures)\;
		$\eta_n = 1^T \mathcal{E}_n$ \quad (sum local error bounds)\;
		\Return $(\mathrm{Q}_n, \eta_n)$
	\end{algorithm}
	\vspace{1em}
	
	Next, we will make rigorous what conditions the problem and algorithm instances must satisfy in order for $\mathrm{AdaQuad}$ to satisfy Equation \eqref{eq: goal}.

	\subsection{Certification and convergence}
	
	Let us start by giving conditions on the problem and algorithm instance which are sufficient to certify the adaptive quadrature, i.e. satisfy the first part of Equation \eqref{eq: goal}.
	
	\begin{proposition}\label{prop: certificate}
		Let $P=(f,F,\Omega)$ and $A=(\mathcal{I}, \mathcal{M}, \mathcal{R})$ be a problem and algorithm instance, respectively. If $f$ is continuous and $\mathcal{I}$ satisfies
		\begin{equation}\label{eq: quadrature}
			\mathcal{I}(f,K) \in F(K)\vol(K),
		\end{equation}
		then $(\mathrm{Q}_n , \eta_n) = \mathrm{AdaQuad}_n(P,A)$ satisfy
		$$ |\mathrm{R}_n|  \leq \eta_n $$
		for every $n\in \mathbb{N}$.
	\end{proposition}
	
	\begin{proof}
		Since $f$ is continuous, it maps connected compact sets to connected compact sets and because it is also real-valued $f(K)\in \mathbb{IR}$ for all $K \subset\Omega$ in $\mathbb{IR}^d$. Thus, we can write $f(K)= [\underline{f(K)}, \overline{f(K)}]$ where the endpoints are the minimum and maximum of $f$ over $K$. By the monotonicity of the integral{,} we have
		$$\underline{f(K)}\vol(K) = \int_K \underline{f(K)}\, dx \leq \int_{K} f(x) \, dx \leq \int_K \overline{f(K)}\, dx = \overline{f(K)}\vol(K), $$
		showing 
		$$\int_K f(x) \, dx  \in f(K)\vol(K). $$
		By assumption we have $\mathcal{I}(f,K)\in F(K)\vol(K)$. It follows that
		$$ \int_\Omega f(x) \, dx, \mathrm{Q}_n \in \sum_{K \in \mathcal{P}_n} F(K) \vol(K),$$
		by the additivity of the integral and definition of $\mathrm{Q}_n$. By definition
		$$ \eta_n = \sum_{K \in \mathcal{P}_n} w(F(K)) \vol(K) $$
		which is equal to 
		$$ w\left(\sum_{K \in \mathcal{P}_n} F(K) \vol(K) \right), $$
		finishing the proof.
	\end{proof}
	
	
	\begin{definition}
		Let $P=(f,F,\Omega)$ and $A=(\mathcal{I},\mathcal{M}, \mathcal{R})$ be a problem and {algorithm} instance, respectively. If $P$ and $A$ satisfy the assumptions of Proposition~\ref{prop: certificate}, we call them a certifiable pair.
	\end{definition}    
	
	Next, we want to derive a sufficient condition that yields the convergence of the global error to zero, i.e. the second part of Equation \eqref{eq: goal}. The following proposition gives properties for the marking and refinement rule to satisfy in order to obtain linear convergence for the error decay. These assumptions are standard in the proof of convergence of adaptive finite element methods (see, e.g. \cite[Sec.~1.14]{verfurth2013posteriori}).
	
	\begin{proposition}\label{prop: convergence}
		Let $\theta,\rho\in(0,1)$ as well as $P=(f,F,\Omega)$ and $A=(\mathcal{I},\mathcal{M}, \mathcal{R})$ be a problem and algorithm instance, respectively. If for a partition $\mathcal{P}\subset \mathbb{IR}^d$ of $\Omega$ and $\mathcal{E} \subset \mathcal{P}\times \mathbb{R}_{\geq 0}$ the inequality
		\begin{equation}\label{eq: bulk condition}
			1^T\widetilde{\mathcal{E}} \geq \theta 1^T\mathcal{E}
		\end{equation}
		is satisfied for $\widetilde{\mathcal{E}}=\{(K, \eta_K)\in\mathcal{E} \, : \, K \in \widetilde{\mathcal{P}}\}$ with $\widetilde{\mathcal{P}}=\mathcal{M}(\mathcal{E})$; and $\mathcal{P}_K = \mathcal{R}(K)$ satisfies
		\begin{equation}\label{eq: local error reduction}
			\sum_{K' \in \mathcal{P}_K} w(F(K'))\vol(K') \leq \rho w(F(K))\vol(K) 
		\end{equation}
		for every $K\in\mathbb{IR}^d$ with $K \subset \Omega$, then for $n\in\mathbb{N}$ and $(\mathrm{Q}_n, \eta_n)=\mathrm{AdaQuad}_n(P,A)$ we have
		\begin{equation}
			\eta_n \to 0 \quad \text{as}\quad n\to \infty \quad \text{and}\quad \eta_{n+1} \leq q\eta_n
		\end{equation}
		with $q = 1 - \theta(1-\rho)$ satisfying $q\in(0,1)$.
	\end{proposition}
	
	\begin{proof}
		Let $(\eta_n)_{n\in \mathbb{N}_0}$ with $(\mathrm{Q}_n, \eta_n)= \mathrm{AdaQuad}_n(P,A)$ for $n \in \mathbb{N}_0$. Let $n\in \mathbb{N}_0$ and $S_n=(\mathcal{P}_n, \mathcal{Q}_n, \mathcal{E}_n)$ be a state instance. From the definition of $\mathrm{Step}_A(S_n,P)$ and $\eta_n$ it follows that
		$$ \eta_{n+1} = \eta_n - \eta_{M} +  \eta_{R},$$
		where $\eta_n = 1^T \mathcal{E}_n$, $\eta_M = 1^T\widetilde{\mathcal{E}}_n$ with $\widetilde{\mathcal{E}}_n = \{(K, \eta_K)\in\mathcal{E}_n \, : \, K \in \widetilde{\mathcal{P}}\}$ and $\widetilde{\mathcal{P}}_n=\mathcal{M}(\mathcal{E}_n)$, $\eta_R = 1^T\mathcal{E}_{n+1,R}$ with $\mathcal{E}_{n+1,R} = \{(K, w(F(K))\vol(K))\}_{K \in \mathcal{P}_{n+1, R}}$ and $\mathcal{P}_{n+1, R} = \bigcup_{K \in \widetilde{\mathcal{P}}_n} \mathcal{R}(K)$. Since the marking rule returns subsets of $\mathcal{P}_n$, inequality \eqref{eq: bulk condition} implies $\eta_M \in [\theta \eta_n, \eta_n]$. Moreover, inequality \eqref{eq: local error reduction} implies $\eta_R \in [0,\rho \eta_M]$. Inserting the upper bound for $\eta_{R}$ yields
		\begin{align*}
			\eta_{n} - \eta_{M} + \eta_{R}  &\leq \eta_n  - \eta_{M} +\rho\eta_{M}.\\
			\intertext{Factoring out $-\eta_{M}$ and inserting the lower bound for $\eta_{M}$ yields}
			\eta_n  - \eta_{M} +\rho\eta_{M} &\leq \eta_n - \theta(1-\rho)\eta_n, \\
			\intertext{and finally}
			\eta_n - \theta(1-\rho)\eta_n & = \big(1-\theta(1-\rho)\big)\eta_n.
		\end{align*}
		Since $\rho, \theta \in (0,1)$ it follows that $q \in (0,1)$ where $q=1- \theta(1-\rho)$. This finishes the proof.
	\end{proof}
	

	\begin{definition}
		Let $P=(f,F,\Omega)$ and $A=(\mathcal{I},\mathcal{M}, \mathcal{R})$ be a problem  and algorithm instance, respectively. If $P$ and $A$ satisfy the assumptions of Proposition~\ref{prop: convergence}, we call them a converging pair.
	\end{definition}

	\subsection{Application to Lebesgue and Sobolev norms}
    
	We want to compute Lebesgue and Sobolev norms of vector-valued functions with the $\mathrm{AdaQuad}$ algorithm. Firstly, we define suitable integrands. 
	\begin{proposition}\label{prop: lebesgue sobolev integrand}
		Let $d,m\in \mathbb{N}$, $\Omega \in \mathbb{IR}^d$, $k\in \mathbb{N}_0$, $1\leq p < \infty$ and $\Phi \in W^{k,p}(\Omega;\mathbb{R}^m)$. Define
		\begin{equation}
			f_{\Phi,k,p}(x) = \sum_{i=1}^m\sum_{|\alpha| \leq k} |D^\alpha \Phi_i(x)|^p .
		\end{equation}
		Then
		\begin{equation}
			\left( \int_\Omega f_{\Phi,k,p}(x)\, dx\right)^{1/p} = ||\Phi||_{W^{k,p}(\Omega; \mathbb{R}^m)}
		\end{equation}
		holds.
	\end{proposition}
	\begin{proof}
		The result follows directly from Definition~\ref{def:lebesgueSob} and the linearity of the integral.
	\end{proof}
	
	Next, we want to construct an interval enclosure for $f_{\Phi,k,p}$ from a box enclosure for $\Phi$.
	
	\begin{proposition}\label{prop: lebesgue sobolev interval enclosure}
		Let $d,m\in \mathbb{N}$, $\Omega \in \mathbb{IR}^d$, $k\in \mathbb{N}_0$, $1\leq p < \infty$ and $\Phi \in W^{k,p}(\Omega;\mathbb{R}^m)$. Furthermore, let for every $\alpha \in \mathbb{N}^d_0$ with $|\alpha|\leq k$, $i\in\{1,\dotsc, m\}$ the function $F_{\Phi,\alpha,i}: \mathbb{IR}^d  \to \mathbb{IR}$ be an interval enclosure of $D^\alpha\Phi_i$. Define
		\begin{equation}
			F_{\Phi,k,p}(K) = \sum_{i=1}^m\sum_{|\alpha|\leq k} |F_{\Phi,\alpha, i}(K)|^p
		\end{equation}
		for $K\in\mathbb{IR}^d$, $K\subset \Omega$. Then $ F_{\Phi,k,p}$ is an interval enclosure of $f_{\Phi,k,p}$.
	\end{proposition}
	\begin{proof}
		Let $K\in \mathbb{IR}$, $K \subset \Omega$. We have 
		$$ D^\alpha\Phi_i(K) \subset F_{\Phi,\alpha, i}(K) $$
		by assumption.
		By Definition~\ref{def: mag abs power} we have
		$ |D^\alpha\Phi_i(x)| \in |F_{\Phi,\alpha, i}(K)|$
		for every $x\in K$, which yields
		$$ |D^\alpha\Phi_i(K)| \subset |F_{\Phi,\alpha, i}(K)|. $$
		Moreover,  by Definition~\ref{def: mag abs power} we have
		$ |D^\alpha\Phi_i(x)|^p \in |F_{\Phi,\alpha, i}(K)|^p$
		for every $x\in K$, which yields
		$$ |D^\alpha\Phi_i(K)|^p \subset |F_{\Phi,\alpha, i}(K)|^p. $$
		Finally, for $\alpha,\beta \in \mathbb{N}_0^d$ satisfying $|\alpha|,|\beta|\leq k$, $i,j\in\{1,\dotsc, m\}$ we have
		$ |D^\alpha\Phi_i(x)|^p + |D^\beta\Phi_j(x)|^p \in |F_{\Phi,\alpha, i}(K)|^p + |F_{\Phi,\beta,j}(K)|^p$
		for every $x\in K$ by Proposition~\ref{prop: endpoint formulas}, which yields
		$$ |D^\alpha\Phi_i(K)|^p + |D^\beta\Phi_j(K)|^p \subset |F_{\Phi,\alpha, i}(K)|^p + |F_{\Phi,\beta,j}(K)|^p. $$
		Summing over all $|\alpha|\leq k$ and $i\in\{1,\dotsc,m\}$ yields
		$$ f_{\Phi,k,p}(K) \subset F_{\Phi,k,p}(K). $$
		This finishes the proof.
	\end{proof}
	
	Next, we show that $F_{\Phi,k,p}$ is Hölder continuous, if the interval extensions $F_{\Phi,\alpha,i}$ of $D^\alpha \Phi_i$ are Hölder continuous.
	
	\begin{proposition}\label{prop: hölder lebesgue sobolev interval enclosure}
		Let $d,m\in \mathbb{N}$, $\Omega \in \mathbb{IR}^d$, $k\in \mathbb{N}_0$, $1\leq p < \infty$ and $\Phi \in W^{k,p}(\Omega;\mathbb{R}^m)$. Furthermore, let for every $\alpha \in \mathbb{N}^d_0$ with $|\alpha| \leq k$, $i\in\{1,\dotsc, m\}$ the function $F_{\Phi,\alpha,i}: \mathbb{IR}^d \to \mathbb{IR}$ be an interval enclosure of $D^\alpha\Phi_i$ and be Hölder continuous on $\Omega$ with constant $C_{\alpha, i}> 0$ and exponent $\gamma_{\alpha, i}\in(0,1]$. Then $F_{\Phi,k,p}$ is a Hölder continuous interval enclosure of $f_{\Phi,k,p}$ with exponent $\min_{\alpha, i} \gamma_{\alpha, i}$.
	\end{proposition}
	\begin{proof}
		It remains to be shown that $F_{\Phi,k,p}$ is Hölder continuous. We have
		\begin{align*}
			w\left( \sum_{i=1}^m\sum_{|\alpha|\leq k} |F_{\Phi,\alpha, i}(K)|^p \right) &= \sum_{i=1}^m\sum_{|\alpha|\leq k}w( |F_{\Phi,\alpha, i}(K)|^p  ) \\
			\intertext{by Remark~\ref{rem: width linearity}. The map $t\mapsto t^p$ for $t\in[0,M]$, $M>0$ is Lipschitz continuous with constant $L_p = pM^{p-1}$. Therefore, by Remark~\ref{rem: hölder continuity} we have}
			\sum_{i=1}^m\sum_{|\alpha|\leq k}w( |F_{\Phi,\alpha, i}(K)|^p  ) &\leq \sum_{i=1}^m\sum_{|\alpha|\leq k} L_p w( |F_{\Phi,\alpha, i}(K)|  ), \\
			\intertext{with $M=\max_{\alpha,i}\overline{|F_{\Phi,\alpha, i}(\Omega)|}$. By Definition~\ref{def: mag abs power} taking the absolute value of an interval does not increase its width. This yields}
			\sum_{i=1}^m\sum_{|\alpha|\leq k} L_p w( |F_{\Phi,\alpha, i}(K)|  )
			& \leq
			\sum_{i=1}^m\sum_{|\alpha|\leq k} L_p w( F_{\Phi,\alpha, i}(K)  ). \\
			\intertext{By assumption the  $F_{\Phi,\alpha, i}(K)$ are Hölder continuous. Thus, we obtain}
			\sum_{i=1}^m\sum_{|\alpha|\leq k} L_p w( F_{\Phi,\alpha, i}(K) )
			& \leq
			\sum_{i=1}^m\sum_{|\alpha|\leq k} L_p C_{\alpha,i} w( K )^{\gamma_{\alpha,i}}. \\
			\intertext{Let $\tilde{\gamma} = \min_{\alpha,i} \gamma_{\alpha,i}$ be the smallest of the Hölder exponents. Factoring out $w(K)^{\tilde{\gamma}}$ yields}
			\sum_{i=1}^m\sum_{|\alpha|\leq k} L_p C_{\alpha,i} w( K )^{\gamma_{\alpha,i}}
			& =
			L_p \left(\sum_{i=1}^m\sum_{|\alpha|\leq k} C_{\alpha,i} w(K)^{\gamma_{\alpha, i} - \tilde{\gamma}}\right) w( K )^{\tilde{\gamma}} \\
			\intertext{We have $w(K) \leq w(\Omega)$ since $K\subset \Omega$. This yields}
			L_p \left(\sum_{i=1}^m\sum_{|\alpha|\leq k} C_{\alpha,i} w(K)^{\gamma_{\alpha, i} - \tilde{\gamma}}\right) w( K )^{\tilde{\gamma}}
			& \leq
			L_p \left(\sum_{i=1}^m\sum_{|\alpha|\leq k} C_{\alpha,i} w(\Omega)^{\gamma_{\alpha, i} - \tilde{\gamma}}\right) w( K )^{\tilde{\gamma}}
		\end{align*}
		and finishes the proof.
	\end{proof}

	\subsection{Application to energy norms}
	Let us now compute the integral of the squared residuals of boundary-value problems. 
	This is especially relevant when solving partial differential equations with neural networks (see \cite{raissi2019pinns}). Here the loss function is taken to be the mean squared residual of the differential equation in the interior and on the boundary evaluated on a set of collocation points. To evaluate the generalization error it is common to simply increase the number of collocation points and scale by the volume of the domain (Monte Carlo integration). We show how $\mathrm{AdaQuad}$ can compute generalization errors for PINNs with rigorous error bounds. Consider the following boundary-value problem (see \cite[Chapter 6]{evans2022partial}). Let $\Omega \in \mathbb{IR}^d$ for $d \in \mathbb{N}$ and consider the boundary-value problem
	\begin{equation}
		\begin{cases}\label{eq:BVP}
			\mathcal{D}u= g & \text{in } \mathring{\Omega} \\
			u = 0 & \text{on } \partial\Omega
		\end{cases}
	\end{equation}
	where $\mathring{\Omega}$ denotes the interior of $\Omega$ and $u: \Omega \to \mathbb{R}$ is the unknown. Here the function $g: \mathring{\Omega} \to \mathbb{R}$ is given and $\mathcal{D}$ denotes a second-order partial differential operator given in nondivergence form
	\begin{equation}\label{eq:DiffOperator}
		\mathcal{D}\Phi(x) = - \sum_{i,j=1}^d a_{ij}(x) \Phi_{x_i x_j}(x)  + \sum_{i=1}^d b_i(x) \Phi_{x_i}(x) + c(x)\Phi(x),
	\end{equation}
	for given coefficient functions $a_{ij}, b_i,c$ for $i,j=1, \dotsc, d$ and trial function $\Phi$. As for Lebesgue and Sobolev norms, let us give the a suitable integrand that allows $\mathrm{AdaQuad}$ to compute Lebesgue norms of the differential equations residual.

    \begin{proposition}\label{prop: energy integrand}
        Let $d \in \mathbb{N}$, $1 \leq p < \infty$ and consider the boundary-value problem \eqref{eq:BVP} with $\mathcal{D}\Phi -g \in L^p(\Omega; \mathbb{R})$. Define
        \begin{equation}
            r_{\Phi, p}(x) = |\mathcal{D}\Phi - g|^p.
        \end{equation}
        Then
        \begin{equation*}
        \left(\int_\Omega r_{\Phi, p}(x) \, dx\right)^{1/p} = ||\mathcal{D}\Phi - g||_{L^p(\Omega; \mathbb{R})}
        \end{equation*}
        holds.
    \end{proposition}

    Next, we show that replacing all real functions with interval enclosures in $r_{\Phi,p}$, yields an interval enclosure $R_{\Phi,p}$.
	
	\begin{proposition}\label{prop: energy interval enclosure}
		Let $d \in \mathbb{N}$, $\Omega \in \mathbb{IR}^d$, $\Phi: \Omega \to \mathbb{R}$ and for every $i,j=0,\dotsc,d$ the function $F_{\Phi, i,j} : \mathbb{IR}^d  \to \mathbb{IR}$ be an interval enclosure of $\Phi_{x_ix_j}$, where $\Phi_{x_0}=\Phi$. Furthermore, for $i,j=1,\dotsc, d$ let $A_{ij},B_i,C,G: \mathbb{IR}^d \to \mathbb{IR}$ be interval enclosures of $a_{ij},b_i,c,g: \mathring{\Omega} \to \mathbb{R}$, respectively. Define
		\begin{equation}
			R_{\Phi,p}(K) = \left|- \sum_{i,j=1}^d A_{ij}(K) F_{\Phi,i,j}(K)+ \sum_{i=1}^d B_i(K) F_{\Phi,i,0}(K) + C(K)F_{\Phi,0,0}(K) -G(K)\right|^p
		\end{equation}
		for $K\in \mathbb{IR}^d$ with $K \subset \Omega$. Then $R_{\Phi, p}$ is an interval enclosure of $r_{\Phi, p}$.
	\end{proposition}
    \begin{proof}
        Let $K \in \mathbb{IR}^d$ with $K\subset\Omega$ and $x \in K$. By assumption, for $i,j=1,\dotsc, d$ we have
        $a_{ij}(x)\in A_{ij}(K)$, $b_i(x) \in B_i(K)$, $c(x) \in C(K)$, $g(x) \in G(K)$ and likewise $\Phi_{x_i x_j}(x)\in F_{\Phi,i,j}(K)$, $\Phi_{x_i}(x) \in F_{\Phi,i,0}(K)$, $\Phi(x) \in F_{\Phi,0,0}(K)$.
        Since interval addition and multiplication are inclusion isotonic interval extensions of the corresponding real operations, the interval expression inside the absolute value contains the real expression evaluated at $x$. Applying $\mathrm{abs}$ and then $(\cdot)^p$ preserves inclusion (with $p\geq 0$) by Definition~\ref{def: mag abs power}, hence $r_{\Phi,p}(x)\in R_{\Phi,p}(K)$.
        Therefore $R_{\Phi,p}$ is an interval enclosure of $r_{\Phi,p}$.
    \end{proof}
	
	Next we show that if the coefficient and partial derivative interval enclosures are Hölder continuous, then $F_{\Phi,\mathcal{D}}$ is also Hölder continuous.
	
	\begin{proposition}\label{prop:DiffOpIntervalEnclosure}
		Let $d \in \mathbb{N}$, $\Omega \in \mathbb{IR}^d$, $\Phi: \Omega \to \mathbb{R}$ and for every $i,j=0,\dotsc,d$ the function $F_{\Phi, i,j} : \mathbb{IR}^d \to \mathbb{IR}$ be an interval enclosure of $\Phi_{x_ix_j}$ and be Hölder continuous on $\Omega$, where $\Phi_{x_0}=\Phi$. Furthermore, for $i,j=1,\dotsc, d$ let $A_{ij},B_i,C: \mathbb{IR}^d \to \mathbb{IR}$ be interval enclosures of $a_{ij},b_i,c$, respectively, and assume that these interval enclosures are Hölder continuous on $\Omega$. Then $F_{\Phi, \mathcal{D}}$ is a Hölder continuous interval enclosure of $\mathcal{D}\Phi$.
	\end{proposition}
	
	\begin{proof}
		In Proposition~\ref{prop: energy interval enclosure} we already showed that $F_{\Phi, \mathcal{D}}$ is an interval enclosure of $\mathcal{D}\Phi$. It remains to be shown that $F_{\Phi, \mathcal{D}}$ is Hölder continuous.
		In the proof of Proposition~\ref{prop: hölder lebesgue sobolev interval enclosure} we already showed that sums of Hölder continuous interval functions on $\Omega \in \mathbb{IR}^d$ are again Hölder continuous. Thus, it suffices to show that the product is also Hölder continuous. Let $F_1,F_2: \mathbb{IR}^d \to \mathbb{IR}$ be inclusion isotonic and Hölder continuous on $\Omega$ with constants $C_1,C_2 > 0$ and exponent $\gamma_1, \gamma_2 \in (0,1]$, respectively. Let $K \in \mathbb{IR}^d, \, K \subset \Omega$. By Equation~\eqref{eq: matrix width inequality} we have
		
		$$w(F_1(K)F_2(K)) \leq \mathrm{mag}(F_1(K))w(F_2(K)) + \mathrm{mag}(F_2(K))w(F_1(K)). $$
		{Applying Hölder continuity yields}
		$$\mathrm{mag}(F_1(K))w(F_2(K)) + \mathrm{mag}(F_2(K))w(F_1(K)) \leq \mathrm{mag}(F_1(K))C_2 w(K)^{\gamma_2} + \mathrm{mag}(F_2(K))C_1 w(K)^{\gamma_1} $$
		{Since $F_1,F_2$ are inclusion isotonic we have $\mathrm{mag}(F_i(K)) \leq \mathrm{mag}(F_i(\Omega))$ for $i=1,2$. Let $C = \max\{C_1\mathrm{mag}(F_2(\Omega)), C_2 \mathrm{mag}(F_1(\Omega)\}$. This yields}
		$$ 
		\mathrm{mag}(F_1(K))C_2 w(K)^{\gamma_2} + \mathrm{mag}(F_2(K))C_1 w(K)^{\gamma_1}
		\leq C (w(K)^{\gamma_1} + w(K)^{\gamma_2}).$$
		taking $\gamma=\min_{i=1,2} \gamma_i$ yields
		$$ C (w(K)^{\gamma_1} + w(K)^{\gamma_2}) \leq C (w(\Omega)^{\gamma_1-\gamma} + w(\Omega)^{\gamma_2 -\gamma})w(K)^\gamma. $$
		We thus showed that the product of inclusion isotonic Hölder continuous interval functions is again Hölder continuous. This completes the proof.
	\end{proof}
	
	Now we give an interval enclosure for the squared residual.
	
	\begin{proposition}\label{prop: hölder energy}
		Let $d \in \mathbb{N}$, $1\leq p < \infty$, $\Omega \in \mathbb{IR}^d$, $\Phi: \Omega \to \mathbb{R}$ and for every $i,j=0,\dotsc,d$ the function $F_{\Phi, i,j} : \mathbb{IR}^d \to \mathbb{IR}$ be an interval enclosure of $\Phi_{x_ix_j}$ and be Hölder continuous on $\Omega$, where $\Phi_{x_0}=\Phi$. Furthermore, for $i,j=1,\dotsc, d$ let $A_{ij},B_i,C,G: \mathbb{IR}^d \to \mathbb{IR}$ be interval enclosures of $a_{ij},b_i,c,g$, respectively, and assume that these interval enclosures are Hölder continuous on $\Omega$. Define
		\begin{equation}
			R_{\Phi,p}(K) = | F_{\Phi,\mathcal{D}}(K) - G(K)|^p
		\end{equation}
		for $K \in \mathbb{IR}^d, \, K \subset \Omega$. Then $R_{\Phi, p}$ is a Hölder continuous interval enclosure of
		$$ r_{\Phi,p}(x) = |\mathcal{D}\Phi(x) - g(x)|^p $$
		where $x \in \Omega$.
	\end{proposition}
	\begin{proof}
		In Proposition~\ref{prop:DiffOpIntervalEnclosure} we showed that under the given assumptions $F_{\Phi,\mathcal{D}}$ is a Hölder continuous interval enclosure of $\mathcal{D}\Phi$. Since $G$ is a Hölder continuous interval enclosure of $g$, by Definition~\ref{def:IntArit} we have that $F_{\Phi,\mathcal{D}} - G$ is an interval enclosure of $\mathcal{D}\Phi - g$. Since finite sums of Hölder continuous interval functions on compact domains are again Hölder continuous (as shown in the proof of Proposition~\ref{prop: hölder lebesgue sobolev interval enclosure}), we have that $F_{\Phi,\mathcal{D}} - G$ is Hölder continuous. To show that $ | F_{\Phi,\mathcal{D}}(K) - G(K)|^p$ is Hölder continuous and an interval enclosure of $|\mathcal{D}\Phi(x) - g(x)|^p$, we can proceed analogously to the proof of Proposition~\ref{prop: hölder lebesgue sobolev interval enclosure}.
	\end{proof}
	
	\begin{remark}
		Assume the coefficient functions are compositions of arithmetic operations and elementary functions (such as $\exp$, $\sin$, $\sqrt{}$, etc.). Then replacing arithmetic operations with interval arithmetic and elementary functions with interval enclosures yields an interval enclosure of the coefficient function (see \cite{moore2009introduction}).
	\end{remark}

	\section{Results}\label{sec: results}
	
	Previously, we saw how to use $\mathrm{AdaQuad}$ to compute Lebesgue and Sobolev norms of vector-valued functions. In this section, we explicitly construct the necessary interval enclosures for the function values and partial derivatives up to order two of a neural network. This yields a problem instance and we specify an algorithm instance such that the pair is certifiable and converging. We start by giving a marking rule for our algorithm instance. This is essentially the Dörfler marking strategy (\cite{dorfler1996convergent}).
	
	\vspace{1em}
	\SetKwInOut{Input}{Input}
	\SetKwInOut{Output}{Output}
	\begin{algorithm}[H]\label{alg: dörfler strategy}
		\LinesNotNumbered
		\DontPrintSemicolon
		\caption{$\mathrm{D\ddot{o}rfler}_\theta(\mathcal{E})$}
		\Input{
			Set $\mathcal{E}=\{(K, \eta_K)\}_{K\in \mathcal{P}}\subset \mathcal{P}\times\mathbb{R}_{\geq 0}$ with index set $\mathcal{P}$; $\theta \in (0,1)$.
		}
		\Output{
			Subset $\widetilde{\mathcal{P}} \subset \mathcal{P}$.
		}
		\BlankLine
		
		$\widetilde{\mathcal{P}} = \emptyset$ \quad (initialize set of marked partition elements)\;
		$\widetilde{\mathcal{E}} = \emptyset$ \quad (initialize set of marked error bounds)\;
		\While{$1^T \widetilde{\mathcal{E}} < \theta 1^T \mathcal{E}$}{
			${\eta}_{\max} = \max \{\eta_K \, : \, (K, \eta_K)\in\mathcal{E}, K \in \mathcal{P}\setminus \widetilde{\mathcal{P}}\}$ \quad (take the largest unmarked error bound)\;
			\ForEach{$K \in \mathcal{P} \setminus \widetilde{\mathcal{P}}$}{
				
				\If{$\eta_K = {\eta}_{\max}$}{
					$\widetilde{\mathcal{P}} = \widetilde{\mathcal{P}} \cup \{K\}$\quad (add the corresponding element to the marked set)\;
					$\widetilde{\mathcal{E}} = \widetilde{\mathcal{E}} \cup \{(K,\eta_K)\}$ \quad (add the corresponding tuple to the error bounds set)\;
				}
			}
		}
		\Return $\widetilde{\mathcal{P}}$\;
	\end{algorithm}
	\vspace{1em}
	
	Next, we give a refinement method {that} requires a Hölder continuous interval enclosure of the integrand.
	
	\vspace{1em}
	\begin{algorithm}[H]\label{alg: hölder strategy}
		\LinesNotNumbered
		\DontPrintSemicolon
		\caption{$\mathrm{H\ddot{o}lder}_{\rho}(K, F, \gamma,C)
			$}
		\Input{
			Interval vector $K=(K_1, \dotsc, K_d)\in \mathbb{IR}^d$; Hölder continuous $F:\mathbb{IR}^d \to \mathbb{IR}$ with constant $C>0$ and exponent $\gamma \in (0,1]$; $\rho \in (0,1)$.
		}
		\Output{
			Partition $\mathcal{P}_{K}\subset \mathbb{IR}^d$ of $K$.
		}
		\BlankLine
		$\eta_K = w(F(K))\vol(K)$\;
		\For{$i = 1,\dotsc, d$}{
			$l_i = \overline{K}_i - \underline{K}_i$\;
			$m_i = \left \lceil l_i\left(\frac{C \vol(K) }{\rho \eta_K} \right)^{1/\gamma} \right \rceil$\;
			\For{$j=1,\dotsc, m_i$}{
				$K_{i,j} = [\underline{K}_i + (j-1)l_i/m_i, \;\underline{K}_i + j l_i / m_i]$\;
				
			}
		}
        $\mathcal{P}_K = \left\{(K_{1,j_1}, \dotsc, K_{d,j_d}) : j_i \in \{1,\dotsc,m_i\},\; i = 1,\dotsc,d \right\}$\;
		\Return $\mathcal{P}_{K}$\;
	\end{algorithm}
	\vspace{1em}
	
	We show that the $\mathrm{D\ddot o rfler}_\theta$ and $\mathrm{H\ddot older}_\rho$ marking and refinement strategies together with a quadrature rule satisfying \eqref{eq: quadrature} are certifiable and converging for all problem instances with Hölder continuous interval enclosures.
	
	\begin{theorem}\label{thm: main conv}
		Define the algorithm instance $A_{\theta, \rho} = (\mathcal{I}, \mathrm{D\ddot orfler}_\theta, \mathrm{H\ddot older}_\rho)$ where $\mathcal{I}$ is a positive-weight quadrature rule that is exact for constants and $\theta, \rho \in (0,1)$. Let $P=(f,F,\Omega)$ where $f$ is continuous, $F$ is an interval enclosure of $f$ and is Hölder continuous on $\Omega$, and $\Omega \in \mathbb{IR}^d$. Then $P$ and $A_{\theta,\rho}$ are certifiable and converging.
	\end{theorem}
	
	\begin{proof}
		By assumption $F$ is an interval enclosure of $f$ and therefore
		$$ f(K) \vol(K) \subset F(K)\vol(K)$$
		for every $K \in \mathbb{IR}^d, \, K \subset \Omega$. By assumption $\mathcal{I}$ is a positive-weight quadrature rule which is exact for constants. This yields $\mathcal{I}(f,K) \in f(K)\vol(K)$ for every $K \in \mathbb{IR}^d, \, K \subset \Omega$ which proves certifiability by Proposition~\ref{prop: certificate}. Towards convergence of the pair $(P, A_{\theta,\rho})$, the $\mathrm{D\ddot orfler}_\theta$ marking strategy satisfies inequality \eqref{eq: bulk condition} by construction. It remains to be shown that $\mathrm{H\ddot older}_\rho$ satisfies \eqref{eq: local error reduction}.  
		Let $K \in \mathbb{IR}^d, \, K \subset \Omega$, and $C>0$, $\gamma\in(0,1]$ be the Hölder constant and exponent of $F$, respectively. Note that
		$ l_im_i^{-1} \leq \left(\frac{\rho \eta_K}{C \vol(K)}\right)^{1/\gamma}, $
		since $\lceil a^{-1}\rceil ^{-1} \leq a$ for every $a>0$. Thus 
		$$w(K') \leq \left(\frac{\rho \eta_K}{C \vol(K)}\right)^{1/\gamma}$$
		for $K' \in \mathcal{P}_K$ and $\mathcal{P}_K = \mathrm{H \ddot older}_{\rho}(K,F,\gamma,C)$.
		We have
		\begin{align*}
			\sum_{K' \in \mathcal{P}_K} w(F(K'))\vol(K')   &\leq \sum_{K' \in \mathcal{P}_K} Cw(K')^\gamma\vol(K' ) \\
			\intertext{by Hölder continuity of $F$. Inserting the upper bound for $w(K')$ yields}
			\sum_{K' \in \mathcal{P}_K} Cw(K')^\gamma\vol(K' ) &\leq \sum_{K' \in \mathcal{P}_K} C\left( \left(\frac{\rho \eta_K}{C \vol(K)}\right)^{1/\gamma}\right)^\gamma \vol(K').\\
			\intertext{After simplifying and inserting $\eta_K = w(F(K))\vol(K)$ we obtain}
			\sum_{K' \in \mathcal{P}_K} C\left( \left(\frac{\rho \eta_K}{C \vol(K)}\right)^{1/\gamma}\right)^\gamma \vol(K') &= \rho w(F(K))\vol(K).
		\end{align*}
		Thus, by Proposition~\ref{prop: convergence} the pair $P$ and {$A_{\theta,\rho}$} are converging. This finishes the proof.
	\end{proof}
	
	\begin{remark}\label{rem:HalfRefinement}
		Let $d\in \mathbb{N}$ and $\mathrm{Half}:\mathbb{IR}^d \to (\mathbb{IR}^d)^{2^d}$ be the refinement strategy that divides a box into two halves along every axis. This refinement strategy does not satisfy Equation~\eqref{eq: local error reduction} for $P=(f,F,\Omega)$ with Hölder continuous $F$ with constant $C>0$ and exponent $\gamma \in (0,1]$ in general. Define
		\begin{equation}
			\bar{\eta}_K = Cw(K)^\gamma \vol(K)
		\end{equation}
		for $K\in \mathbb{IR}^d$. This is an upper bound to $\eta_K = w(F(K))\vol(K)$
		by Hölder continuity of $F$. Moreover, for $K\in \mathbb{IR}^d$ the refinement $\mathcal{P}_K = \mathrm{Half}(K)$ satisfies
		\begin{equation}
			\sum_{K' \in \mathcal{P}_K} \bar{\eta}_{K'} \leq \frac{1}{2^\gamma} \bar{\eta}_K,
		\end{equation}
		since $w(K')=\tfrac{1}{2}w(K)$ for $K' \in \mathcal{P}_K$.
	\end{remark}

	\subsection{Neural network function value bounds}\label{subsec:functionValues}

	For a neural network $\Phi: \mathbb{R}^{d_0} \to \mathbb{R}^{d_{L+1}}$ we obtain bounds on the output $\Phi(K)$ for interval vectors $K\in \mathbb{IR}^{d_0}$ by replacing arithmetic operations with interval arithmetic operations and the activation function $\sigma:\mathbb{R} \to \mathbb{R}$ by an interval enclosure $\Sigma : \mathbb{IR} \to \mathbb{IR}$.
	
	\vspace{1em}
	\begin{algorithm}[H]\label{alg: nn fval bounds}
		\LinesNotNumbered
		\DontPrintSemicolon
		\caption{$\mathrm{Fval_{\Phi,\ell,\Sigma}}(K)$}
		\Input{
			Interval vector $K\in \mathbb{IR}^{d_0}$; neural network $\Phi : \mathbb{R}^{d_0} \to \mathbb{R}^{d_{L+1}}$ of depth $L$ with activation function $\sigma : \mathbb{R} \to \mathbb{R}$; layer $\ell \in\{ 1, \dotsc, L+1\}$; interval enclosure $\Sigma : \mathbb{IR} \to \mathbb{IR}$ of $\sigma$.
		}
		\Output{
			Interval vector $Z^{(\ell)}\in \mathbb{IR}^{d_{\ell}}$ containing the pre-activations of layer $\ell$ for inputs from $K$.
		}
		\vspace{1em}
		${X}^{(0)} = K$\;
		\For{$\ k = 0, \dotsc, \ell-1$}{
			${Z}^{(k+1)} = W^{(k)} X^{(k)} + b^{(k) }$\;
			${X}^{(k+1)} = \Sigma({Z}^{(k+1)})$\;
		}
		\Return $Z^{(\ell)}$\;
	\end{algorithm}
	\vspace{1em}
	
	\begin{remark}
		The algorithm $\mathrm{Fval}_{\Phi,\ell, \Sigma}(K)$ returns bounds for the function values of the neural network $\Phi$ when choosing $\ell=L+1$ and bounds on the pre-activations of the $\ell$-th hidden layer when $\ell < L+1$, since no activation function is applied to the last layer, i.e. $x^{(L+1)} = z^{(L+1)}$.
	\end{remark}

	We want to show that for $\ell=1,\dotsc,L+1$ the function $\mathrm{Fval}_{\Phi, \ell, \Sigma}$ is a Hölder continuous interval enclosure of $\Phi$ {provided} the activation function allows a Hölder continuous interval enclosure.
	\begin{remark}\label{rem:monotonicSigma}
		If the activation function $\sigma : \mathbb{R} \to \mathbb{R}$ is monotonic, we can compute its range on an interval exactly only by evaluating the endpoints accordingly.
	\end{remark}
	
	We start by showing $\mathrm{Fval}_{\Phi, \ell, \Sigma}$ is an interval extension.
	
	\begin{proposition}\label{prop: fval interval extension}
		Let $\Phi : \mathbb{R}^{d_0} \to \mathbb{R}^{d_{L+1}}$ be a neural network and $\Sigma:\mathbb{IR} \to \mathbb{IR}$ be an interval extension of $\sigma : \mathbb{R} \to \mathbb{R}.$ Then for $\ell=1,\dotsc, L+1$ we have that $\mathrm{Fval}_{\Phi, \ell, \Sigma}$ is {an} interval extension of the pre-activation $z^{(\ell)}$, i.e.
		$$ z^{(\ell)}(\{x\}) = \mathrm{Fval}_{\Phi, \ell, \Sigma}(\{x\}) $$
		for all $x\in \Omega$.
	\end{proposition}
	
	\begin{proof}
		Note that matrix-interval-vector-multiplication is an interval extension of
		matrix-vector-multiplication. Moreover, $\Sigma$ is an interval extension of $\sigma$ and compositions of interval extensions are again interval extensions. Thus, it follows that $\mathrm{Fval}_{\Phi, \ell, \Sigma}$ is an interval extension of the pre-activation $z^{(\ell)}$ for all $\ell =1, \dots, L+1$.
	\end{proof}
	
	Next, we show that this interval extension is inclusion isotonic.
	
	\begin{proposition}\label{prop: fval inclusion isotonic}
		Let $\Phi : \mathbb{R}^{d_0} \to \mathbb{R}^{d_{L+1}}$ be a neural network with activation function $\sigma : \mathbb{R} \to \mathbb{R}$ and inclusion isotonic interval extension $\Sigma: \mathbb{IR} \to \mathbb{IR}$. Then for $\ell=1,\dotsc, L+1$ we have that $\mathrm{Fval}_{\Phi, \ell, \Sigma}$ is inclusion isotonic, i.e.
		\begin{equation}
			\mathrm{Fval}_{\Phi, \ell, \Sigma}(K') \subset \mathrm{Fval}_{\Phi, \ell, \Sigma}(K) \quad \text{if}\, K' \subset K
		\end{equation}
		for $K',K \in \mathbb{IR}^{d_0}, \, K',K \subset \Omega$.
	\end{proposition}
	
	\begin{proof}
		Notice that matrix-interval-vector-multiplication is inclusion isotonic and by assumption $\Sigma$ is inclusion isotonic as well. Moreover, the composition of inclusion isotonic functions is again inclusion isotonic. This finishes the proof.
	\end{proof}
	
	From the Fundamental Theorem of Interval Analysis it follows that for $\ell=0,\dotsc,L+1$ the function $\mathrm{Fval}_{\Phi, \ell, \Sigma}$ is an interval enclosure of $\Phi$.
	
	\begin{corollary}\label{cor: fval int encl}
		Let $\Phi : \mathbb{R}^{d_0} \to \mathbb{R}^{d_{L+1}}$ be a neural network with activation function $\sigma : \mathbb{R} \to \mathbb{R}$ and interval enclosure $\Sigma:\mathbb{IR}\to \mathbb{IR}$. Then for $\ell=1,\dotsc,L+1$ the function $\mathrm{Fval}_{\Phi, \ell, \Sigma}$ is an interval enclosure of $z^{(\ell)}$.
	\end{corollary}
	\begin{proof}
		Since by Proposition~\ref{prop: fval interval extension} and Proposition~\ref{prop: fval inclusion isotonic} the function $\mathrm{Fval}_{\Phi, \ell, \Sigma}$ is an inclusion isotonic interval extension of $z^{(\ell)}$. The Fundamental Theorem of Interval Analysis~\ref{thm: fund thm of IA} yields that $\mathrm{Fval}_{\Phi, \ell, \Sigma}$ is an interval enclosure of $z^{(\ell)}$, finishing the proof.
	\end{proof}

	\begin{proposition}\label{prop: fval hölder}
		Let $\Phi : \mathbb{R}^{d_0} \to \mathbb{R}^{d_{L+1}}$ be a neural network with activation function $\sigma : \mathbb{R} \to \mathbb{R}$ and Hölder continuous interval enclosure $\Sigma:\mathbb{IR}\to \mathbb{IR}$ having exponent $\gamma_\Sigma \in (0,1]$ and constant $C_\Sigma > 0$. Then for $\ell=1,\dotsc,L+1$ the function $\mathrm{Fval}_{\Phi, \ell, \Sigma}$ is Hölder continuous on $\Omega$, i.e., there exist $C > 0$ and $\gamma \in (0,1]$ such that
		$$ w(\mathrm{Fval}_{\Phi, \ell, \Sigma}(K)) \leq C w(K)^{\gamma} $$
		for all $K \in \mathbb{IR}^{d_0}$ with $K \subset \Omega$.
	\end{proposition}
	
	\begin{proof}
		Let $\ell\in\{1, \dotsc, L+1\}$, $K\in \mathbb{IR}^{d_0}, \, K \subset \Omega$. Let $Z^{(\ell)} = \mathrm{Fval}_{\Phi, \ell, \Sigma}(K)$. By definition we have
		\begin{align*}
			w(Z^{(\ell)}) &= w(W^{(\ell-1)}X^{(\ell-1)} + b^{(\ell-1)}) \\
			\intertext{and applying Equations~\eqref{eq: width subadditivity} and~\eqref{eq: matrix width inequality} yields}
			w(W^{(\ell-1)}X^{(\ell-1)} + b^{(\ell-1)}) &\leq ||W^{(\ell-1)}||w(X^{(\ell-1)}), \\
			\intertext{since the matrix $W^{(\ell -1)}$ and the vector $b^{(\ell-1)}$ have zero width. If $\ell=1$ we stop. If $\ell>1$, applying the Hölder continuity of $\Sigma$ yields}
			||W^{(\ell-1)}||w(X^{(\ell-1)}) &\leq ||W^{(\ell-1)}||C_\Sigma w(Z^{(\ell-1)})^{\gamma_\Sigma}.
		\end{align*}
        Iterating this recurrence and considering $K = Z^{(0)}$ yields $w(Z^{(\ell)}) \leq C w(K)^{\gamma}$ for some $C > 0$ and $\gamma = \gamma_\Sigma^{\ell-1} \in (0,1]$.
		This finishes the proof.
	\end{proof}
	
	We can now apply Theorem~\ref{thm: main conv} to compute Lebesgue norms of neural networks. 
	
	\begin{corollary}\label{corr:LpNormsconvergenceOfQuadratureRule}
		Let $\Phi : \mathbb{R}^{d_0} \to \mathbb{R}^{d_{L+1}}$ be a neural network with activation function $\sigma : \mathbb{R} \to \mathbb{R}$ and Hölder continuous interval enclosure $\Sigma:\mathbb{IR}\to \mathbb{IR}$ and $1\leq p < \infty$. 
		Let $\alpha = 0 \in \mathbb{N}^d_0$ and set
		$$ F_{\Phi,\alpha,i}(K) = \mathrm{Fval}_{\Phi, L+1, \Sigma}(K)_i,$$
		where $\mathrm{Fval}_{\Phi, L+1, \Sigma}(K)_i$ is the interval along the $i$-th axis.
		Then $f_{\Phi,0,p}$ is continuous and $F_{\Phi,0,p}$ is a Hölder continuous interval enclosure of $f_{\Phi,0,p}$. Thus, $P_{0,p}=(f_{\Phi, 0,p}, F_{\Phi, 0,p}, \Omega)$ and {$A_{\theta, \rho}$} are certifiable and convergent. In particular, we have
		$$ |\, ||\Phi||_{L^p(\Omega; \mathbb{R}^{d_{L+1}})} - \mathrm{Q}_n^{1/p}| \leq \eta_n^{1/p} $$
		and $\eta_n \leq \eta_0 q^n$ where $(\mathrm{Q}_n, \eta_n) = \mathrm{AdaQuad}_n(P_{0,p}, A_{\theta, \rho})$ for $n\in \mathbb{N}$ and $q=1-\theta(1-\rho)$.
	\end{corollary}
	\begin{proof}
		Corollary~\ref{cor: fval int encl} and Proposition~\ref{prop: fval hölder} yield that $F_{\Phi,0,p}$ with $F_{\Phi,0,i} = \mathrm{Fval}_{\Phi, L+1, \Sigma}(K)_i$, $i=1,\dotsc, d_{L+1}$ is a Hölder continuous interval enclosure of $f_{\Phi,0,p}$. By Theorem~\ref{thm: main conv} the pair $P_{0,p}$ and $A_{\theta,\rho}$ are certifiable and converging. The map $t \mapsto t^{1/p}$ is Hölder continuous for $1 \leq p < \infty$ and $t\geq 0$ with constant $1$ and exponent $1/p$. For $n\in \mathbb{N}$ and $(\mathrm{Q}_n, \eta_n) = \mathrm{AdaQuad}_n(P_{0,p}, A_{\theta,\rho})$ this yields
		$$ |\, ||\Phi||_{L^p(\Omega ; \mathbb{R}^{d_{L+1}})}  - \mathrm{Q}_n^{1/p}|\leq  |\, ||\Phi||_{L^p(\Omega ; \mathbb{R}^{d_{L+1}})}^p  - \mathrm{Q}_n|^{1/p} = |\mathrm{R}_n|^{1/p} \leq \eta_n^{1/p}. $$
		This finishes the proof.
	\end{proof}
	
	Since ReLU neural networks are piecewise affine linear, we can integrate them exactly on boxes which are contained in affine linear pieces. Moreover, taking the absolute value and raising the output of the neural network to the power of $\tilde{p}\in \mathbb{N}$ still allows us to integrate exactly via quadratures with degrees of exactness greater or equal $\tilde{p}$ (see Definition~\ref{def: quadrature rule}). Thus, the error on such boxes is exactly zero and no error bound is needed. Therefore, we must check if a box is inside an affine linear piece. In \cite{hanin2019deep} it was shown that activation pieces (activation regions) are contained in affine linear pieces. This gives us an efficient way to check whether or not a box is contained in such a piece. Firstly, we define activation patterns of a ReLU neural network.

	\begin{definition}[\cite{hanin2019deep}]\label{def:ActiPattern}
		Let $\Phi: \mathbb{R}^{d_0} \to \mathbb{R}^{d_{L+1}}$ be a ReLU neural network and $u \in \mathbb{R}^{d_0}$. The activation pattern of $\Phi$ for the input $u$ is defined as
		\begin{equation*}
			\mathfrak{A}_\Phi(u) = \{ \mathrm{sign}(x^{(\ell)}(u)) \, : \, \text{for }\ell=1,\dotsc, L\} \in \mathfrak{B}_\Phi
		\end{equation*}
		with $\mathfrak{B}_\Phi=  \{0,1\}^{d_1} \times \dots \times \{0,1\}^{d_L}$ and $\mathrm{sign}$ applied entrywise, mapping negative numbers to $-1$, $0$ to $0$, and positive numbers to $1$.
	\end{definition}
	
	Next, we define activation pieces of a ReLU neural network.
	
	\begin{definition}\label{def:ActiPiece}
		Let $\Phi: \mathbb{R}^{d_0} \to \mathbb{R}^{d_{L+1}}$ be a ReLU neural network and $a \in \mathfrak{B}_\Phi$. We define the activation piece as the set of inputs $u\in \mathbb{R}^{d_0}$ with activation pattern $a$:
		\begin{equation*}
			\mathfrak{R}_\Phi(a) = \{ u \in \mathbb{R}^{d_0} \, : \, \mathfrak{A}_\Phi(u) = a\}\subset \mathbb{R}^{d_0},
		\end{equation*}
		which is equal to $\mathfrak{A}_\Phi^{-1}(a)$.
	\end{definition}
	
	\begin{remark}
		Notice that if for a ReLU neural network $\Phi: \mathbb{R}^{d_0} \to \mathbb{R}^{d_{L+1}}$, set $C \subset \mathbb{R}^d$ and activation pattern $a\in \mathfrak{B}_\Phi$ we have $\mathfrak{A}_\Phi(C) =\{a\}$, it follows $C\subset \mathfrak{R}_\Phi(a)$.
	\end{remark}
	
	In principle we would need to check the activation pattern for every element in a given box $K \in \mathbb{IR}^{d_0}$, but Lemma 1 from \cite{hanin2019deep} shows that the activation pieces are convex.
	
	\begin{definition}
		Let $d \in \mathbb{N}$, $K= (K_1, \dotsc, K_d)\in \mathbb{IR}^d$ with $K_i = [\underline{K}_i, \overline{K}_i]$ for $i=1,\dotsc,d$ and $b \in \{0,1\}^d$. Then the vertices of the box are given by
		$$ v(b)_i = (1- b_i) \underline{K}_i + b_i \overline{K}_i, $$
		with $i=1,\dotsc,d$. Moreover,
		$$ \mathrm{vert}(K)  = \{ v(b) \in \mathbb{R}^d \, : \, b \in \{0,1\}^d \} $$
		is the set of vertices of $K$.
	\end{definition}
	
	It suffices to compare the activation patterns of the vertices of a box to have a sufficient condition whether or not it is inside an activation piece.
	
	\begin{proposition}\label{prop:ActiPieceCheck}
		Let $\Phi: \mathbb{R}^{d_0} \to \mathbb{R}^{d_{L+1}}$ be a ReLU neural network and $K \in \mathbb{IR}^{d_0}$. Then $\mathfrak{A}_\Phi(K)$ is a singleton if and only if $\mathfrak{A}_\Phi(\mathrm{vert}(K))$ is a singleton.
	\end{proposition}
	\begin{proof}
		Let $a \in \mathfrak{B}_\Phi$. Assume that $\mathfrak{A}_\Phi(K) = \{a\}$. Since $\mathrm{vert}(K) \subset K$ we have $\mathfrak{A}_\Phi(\mathrm{vert}(K)) = \{a\}$. Now let $\mathfrak{A}_\Phi (\mathrm{vert}(K))=\{a\}$. Since by \cite[Lemma 1]{hanin2019deep} the activation piece is convex, it follows $\mathfrak{A}_\Phi (K)=\{a\}$ from $\mathrm{conv}(\mathrm{vert}(K)) = K$.
	\end{proof}
	
	Since the activation pieces are contained in affine linear pieces, we obtain a sufficient condition by only comparing the activation patterns of the vertices.
	
	\begin{proposition}\label{prop:AffLinPieceCheck}
		Let $\Phi: \mathbb{R}^{d_0} \to \mathbb{R}^{d_{L+1}}$ be a ReLU neural network and $K \in \mathbb{IR}^{d_0}$. Then $\Phi$ is affine linear on $K$ if $\mathfrak{A}_\Phi(\mathrm{vert}(K))$ is a singleton.
	\end{proposition}
	\begin{proof}
		Let $K \in \mathbb{IR}^{d_0}$ and $\mathfrak{A}_\Phi(K)$ be a singleton. By Proposition~\ref{prop:ActiPieceCheck} the set $K$ is contained in an activation piece. By \cite[Lemma 3]{hanin2019deep} the activation pieces are contained in affine linear pieces. This completes the proof.
	\end{proof}

	\begin{figure}[htb]
		\centering
		\includegraphics[width=0.32\linewidth]{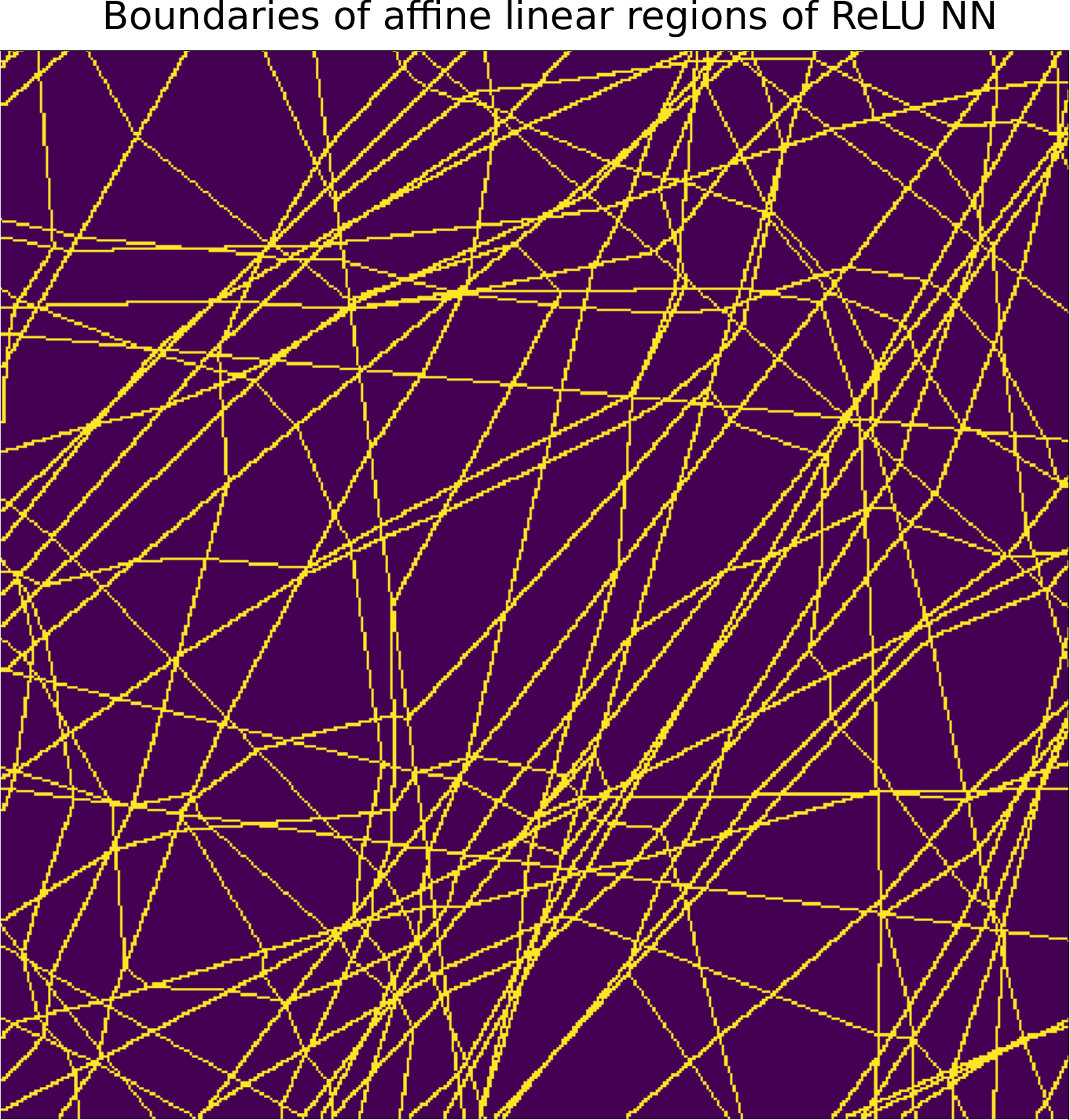} \ \includegraphics[width=0.32\linewidth]{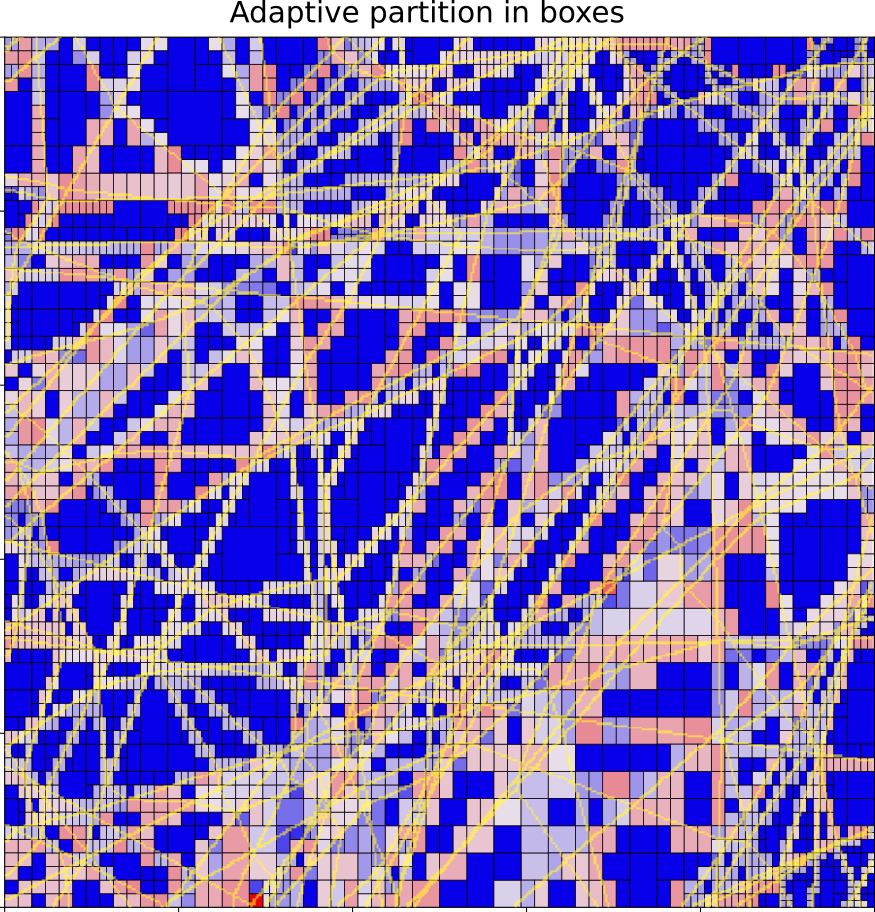} \ 
		\includegraphics[width=0.32\linewidth]{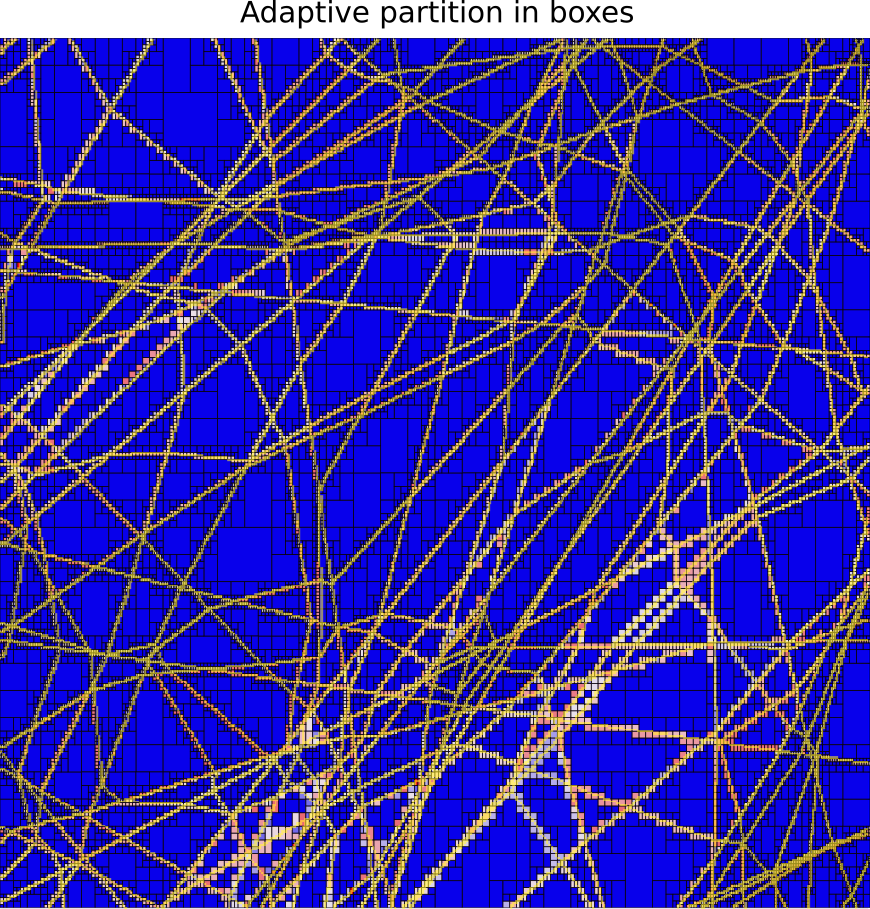}
		\caption{\textbf{Left: } Boundaries of affine linear pieces for a random ReLU network of width 40 and depth 5. \textbf{Middle: } Adaptive partition after 30 refinement steps of $\mathrm{AdaQuad}$ with color indicating local error bounds. The boundaries of affine linear regions are overlaid. \textbf{Right: } Adaptive partition after 40 refinement steps of $\mathrm{AdaQuad}$ with color indicating local error bounds. The boundaries of affine linear regions are overlaid.}
		\label{fig:relu_partition_deep}
	\end{figure}
	
	\begin{remark}\label{rem:ExactIntegration}
		We can use Proposition~\ref{prop:AffLinPieceCheck} to efficiently check whether or not a box is inside an affine linear piece of a ReLU neural network. If true, we can use a quadrature with degree of exactness equal to one and integrate without error. When taking the absolute value and a power of the neural network output, we can integrate exactly by using a quadrature rule with a degree of exactness higher than said power. Figure \ref{fig:relu_partition_deep} shows the effect of employing the activation region check in $\mathrm{AdaQuad}$. 
	\end{remark}
	\begin{remark}\label{rem:higher-orderReLU}
		In principle, the methods of this section can be extended to neural networks with a $ReLU^k$ activation function for $k \in \N$. These neural networks will then be piecewise polynomial and can be integrated exactly with appropriate quadrature rules whenever boxes lie fully within an activation region.
	\end{remark}

	
	\subsection{Neural network Jacobian bounds}\label{subsec:jacobianBounds}
	We aim to construct an interval enclosure for the Jacobian $\nabla\Phi (K)$ with $ K \in \mathbb{IR}^{d_{0}}$ for a neural network $\Phi : \mathbb{R}^{d_0} \to \mathbb{R}^{d_{L+1}}$.
	Recall that we have $  \Phi^{(0)} =\Phi$, implying
	\begin{equation}\label{eq: full tail jacobian}
		\nabla\Phi^{(0)}(x^{(0)}) = \nabla \Phi (x),
	\end{equation}
	where $\Phi^{(0)}$ is a tail neural network (see Definition~\ref{def: tail nn}) and $x^{(0)}=x\in \mathbb{R}^{d_0}$.
	Furthermore, for $\ell=0, \dotsc, L-1$ we have the recursion
	$$
	\Phi^{(\ell)} = \Phi^{(\ell+1)} \circ \sigma \circ T_{\ell}.
	$$
	This implies
	$$ \nabla  \Phi^{(\ell)}(u) = \nabla\Phi^{(\ell+1)}( \sigma( T_{\ell}(u))) \operatorname{diag}(\sigma' (T_\ell(u))) W^{(\ell)}$$
	for $u \in \mathbb{R}^{d_\ell}$ by the chain rule. In particular, for $u \in \mathbb{R}^{d_0}$ we have
	\begin{equation}\label{eq: shortest tail jac}
		\nabla \Phi^{(L)}(x^{(L)}(u)) = W^{(L)} 
	\end{equation}
	and
	\begin{equation}\label{eq: jacobian recursion}
		\nabla  \Phi^{(\ell)}(x^{(\ell)}(u)) = \nabla\Phi^{(\ell+1)}( x^{(\ell+1)}(u)) \operatorname{diag}(\sigma' (z^{(\ell+1)}(u))) W^{(\ell)}
	\end{equation}
	for $\ell=0, \dotsc, L-1$. Starting with Equation~\eqref{eq: shortest tail jac} and repeatedly applying Equation~\eqref{eq: jacobian recursion} allows us to compute $\nabla \Phi(u)$ by Equation~\eqref{eq: full tail jacobian}. To bound the Jacobian $\nabla \Phi(K)$ for $K \in \mathbb{IR}^{d_0}$ we replace the pre-activations $z^{(\ell)}(K)$ by $\mathrm{Fval}_{\Phi, \ell, \Sigma}(K)$ and carry out the matrix multiplications via interval arithmetic (see Proposition~\ref{prop: endpoint formulas}).
	
	\vspace{1em}
	\begin{algorithm}[H]\label{alg: nn jac bounds}
		\LinesNotNumbered
		\DontPrintSemicolon
		\caption{$\mathrm{Jac}_{\Phi,\ell, \Sigma, \Sigma'}(K)$}
		\Input{
			Interval vector $K\in \mathbb{IR}^{d_0}$; neural network $\Phi : \mathbb{R}^{d_0} \to \mathbb{R}^{d_{L+1}}$ of depth $L$ with activation function $\sigma : \mathbb{R} \to \mathbb{R}$; layer $\ell \in\{ 0, \dotsc, L\}$; interval enclosures $\Sigma,\Sigma' : \mathbb{IR} \to \mathbb{IR}$ of $\sigma$, $\sigma'$, respectively.
		}
		\Output{
			Interval matrix $J^{(\ell)}\in \mathbb{IR}^{d_{L+1} \times d_{\ell}}$ containing the Jacobians for the $\ell$-th tail neural network with inputs from $x^{(\ell)}(K)$.
		}
		\vspace{1em}
		$J^{(L)} = W^{(L)}$\;
		\For{$\ k = L-1, \dotsc, \ell$}{
			$Z^{(k+1)} = \mathrm{Fval}_{\Phi, k+1, \Sigma}(K)$\;
			$J^{(k)} = J^{(k+1)} \operatorname{diag}(\Sigma'(Z^{(k+1)}))  W^{(k)}$\;
		}
		\Return $J^{(\ell)}$\;
	\end{algorithm}
	\vspace{1em}
	
	We now want to show that for a neural network $\Phi : \mathbb{R}^{d_0} \to \mathbb{R}^{d_{L+1}}$ the function $\mathrm{Jac}_{\Phi,\ell, \Sigma, \Sigma'}$ is an interval extension of the Jacobian of the $\ell$-th tail neural network for inputs from $x^{(\ell)}(K)$ and $K \in \mathbb{IR}^{d_0}$ for $\ell = 0, \dotsc, L$.

	\begin{proposition}\label{prop: jac interval extension}
		Let $\Phi : \mathbb{R}^{d_0} \to \mathbb{R}^{d_{L+1}}$ be a neural network with differentiable activation function $\sigma : \mathbb{R} \to \mathbb{R}$ and interval extensions $\Sigma,\Sigma':\mathbb{IR}\to \mathbb{IR}$ of $\sigma,\sigma'$, respectively. Then for $\ell=0,\dotsc, L$ we have that $\mathrm{Jac}_{\Phi,\ell, \Sigma, \Sigma'}$ is {an} interval extension {of} the Jacobian of {the} $\ell$-th tail neural network evaluated at $x^{(\ell)}$, i.e.
		$$ \nabla \Phi^{(\ell)}(x^{(\ell)}(\{u\})) = \mathrm{Jac}_{\Phi,\ell, \Sigma, \Sigma'}(\{u\}) $$
		for all $u\in \mathbb{R}^{d_0}$.
	\end{proposition}
	
	\begin{proof}
		Let $u \in \mathbb{R}^{d_0}$. The result follows directly from the fact that for $\ell=0, \dotsc, L$ we can compute $\nabla \Phi^{(\ell)}(x^{(\ell)}(u))$ via Equation \eqref{eq: shortest tail jac} and recursion \eqref{eq: jacobian recursion}{, and that} $\Sigma'$ as well as $\mathrm{Fval}_{\Phi,k,\Sigma}$ are interval extensions for $k=L-1, \dotsc, \ell$.
	\end{proof}

	Next, we show that this interval extension is inclusion isotonic.
	
	\begin{proposition}\label{prop: jac inclusion isotonic}
		Let $\Phi : \mathbb{R}^{d_0} \to \mathbb{R}^{d_{L+1}}$ be a neural network with differentiable activation function $\sigma : \mathbb{R} \to \mathbb{R}$ and inclusion isotonic interval extensions $\Sigma,\Sigma':\mathbb{IR}\to \mathbb{IR}$ of $\sigma,\sigma'$, respectively. Then for $\ell=0,\dotsc, L$ we have that $\mathrm{Jac}_{\Phi,\ell, \Sigma, \Sigma'}$ is inclusion isotonic, i.e.
		\begin{equation}
			\mathrm{Jac}_{\Phi,\ell, \Sigma, \Sigma'}(\tilde{K}) \subset \mathrm{Jac}_{\Phi,\ell, \Sigma, \Sigma'}(K) \quad \text{if}\quad \tilde{K} \subset K
		\end{equation}
		for $\tilde{K},K \in \mathbb{IR}^{d_0}$.
	\end{proposition}
	\begin{proof}
		Let $\tilde{K},K \in \mathbb{IR}^{d_0}$ satisfy $\tilde{K} \subset K$. We have
		$$ \mathrm{Jac}_{\Phi,L, \Sigma, \Sigma'}(\tilde{K}) = \mathrm{Jac}_{\Phi,L, \Sigma, \Sigma'}({K}) = W^{(L)}. $$
		Now assume
		\begin{equation}\label{eq: jac induction hypothesis}
			\mathrm{Jac}_{\Phi,\ell+1, \Sigma, \Sigma'}(\tilde{K}) \subset \mathrm{Jac}_{\Phi,\ell+1, \Sigma, \Sigma'}(K)
		\end{equation}
		for an arbitrary but fixed $\ell \in \{0, \dotsc, L-1\}$. Note that by definition we have
		\begin{equation}\label{eq: algo jac recursion}
			\mathrm{Jac}_{\Phi,\hat{\ell}, \Sigma, \Sigma'}(\hat{K}) = 
			\mathrm{Jac}_{\Phi,\hat{\ell}+1, \Sigma, \Sigma'}(\hat{K})  \operatorname{diag}(\Sigma'(\mathrm{Fval}_{\Phi , \hat{\ell}+1,\Sigma}(\hat{K}))) W^{(\hat{\ell})}
		\end{equation}
		for $\hat{K} \in \mathbb{IR}^{d_0}$ and $\hat{\ell} = 0, \dotsc, L-1$. 
		This yields
		\begin{align*}
			\mathrm{Jac}_{\Phi,\ell, \Sigma, \Sigma'}(\tilde{K}) & = \mathrm{Jac}_{\Phi,{\ell}+1, \Sigma, \Sigma'}(\tilde{K})  \operatorname{diag}(\Sigma'(\mathrm{Fval}_{\Phi , {\ell}+1,\Sigma}(\tilde{K}))) W^{(\ell)}.
		\end{align*}
		Applying Equation~\eqref{eq: jac induction hypothesis} and the inclusion isotonicity of $\Sigma'$ and $\mathrm{Fval}_{\Phi, \ell+1, \Sigma}$ yields
		\begin{align*}
			\mathrm{Jac}_{\Phi,{\ell}+1, \Sigma, \Sigma'}(\tilde{K})  \operatorname{diag}(\Sigma'(\mathrm{Fval}_{\Phi , {\ell}+1,\Sigma}(\tilde{K}))) W^{(\ell)} &\subset \mathrm{Jac}_{\Phi,{\ell}+1, \Sigma, \Sigma'}({K})  \operatorname{diag}(\Sigma'(\mathrm{Fval}_{\Phi , {\ell}+1,\Sigma}({K}))) W^{(\ell)}. 
		\end{align*}
		Finally, we obtain
		\begin{align*}
			\mathrm{Jac}_{\Phi,{\ell}+1, \Sigma, \Sigma'}({K})  \operatorname{diag}(\Sigma'(\mathrm{Fval}_{\Phi , {\ell}+1,\Sigma}({K}))) W^{(\ell)} &= \mathrm{Jac}_{\Phi,{\ell}, \Sigma, \Sigma'}({K}) 
		\end{align*}
		by Equation~\eqref{eq: algo jac recursion}. This finishes the proof.
	\end{proof}
	
	From the Fundamental Theorem of Interval Analysis it follows that for $\ell=0,\dotsc,L$ the function $\mathrm{Jac}_{\Phi,\ell, \Sigma, \Sigma'}$ is an interval enclosure of $\nabla\Phi^{(\ell)}(x^{(\ell)}(K))$ for neural network $\Phi$ and $K \in \mathbb{IR}^d$.
	
	\begin{corollary}\label{cor: jac int encl}
		Let $\Phi : \mathbb{R}^{d_0} \to \mathbb{R}^{d_{L+1}}$ be a neural network with differentiable activation function $\sigma : \mathbb{R} \to \mathbb{R}$ and interval enclosures $\Sigma,\Sigma':\mathbb{IR}\to \mathbb{IR}$ of $\sigma,\sigma'$, respectively. Then for $\ell=0,\dotsc,L$ the function $\mathrm{Jac}_{\Phi,\ell, \Sigma, \Sigma'}(K)$ is an interval enclosure of $\nabla\Phi^{(\ell)}(x^{(\ell)}(K))$ for $K\in \mathbb{IR}^{d_0}$.
	\end{corollary}
	\begin{proof}
		Since by Proposition~\ref{prop: jac interval extension} and Proposition~\ref{prop: jac inclusion isotonic} the function $\mathrm{Jac}_{\Phi,\ell, \Sigma, \Sigma'}$ is an inclusion isotonic interval extension of $\nabla\Phi^{(\ell)}(x^{(\ell)}(.))$. The Fundamental Theorem of Interval Analysis (see \ref{thm: fund thm of IA}) yields that $\mathrm{Jac}_{\Phi,\ell, \Sigma, \Sigma'}$ is an interval enclosure of $\nabla\Phi^{(\ell)}(x^{(\ell)}(.))$, finishing the proof.
	\end{proof}

    \begin{remark}\label{rem: local lipschitz from jac}
		Algorithm~\ref{alg: nn jac bounds} also yields computable certified local Lipschitz constants for $\Phi$ on any box $K \in \mathbb{IR}^{d_0}$. By the classical result that the Lipschitz constant of a differentiable map on a convex set is bounded by the supremum of the operator norm of its Jacobian, and since $\mathrm{Jac}_{\Phi,0,\Sigma,\Sigma'}(K)$ is an interval enclosure of $\nabla\Phi$ on $K$ by Corollary~\ref{cor: jac int encl} together with Definition~\ref{def: interval matrix norm}, the quantity
		$$L_K := \|\mathrm{Jac}_{\Phi,0,\Sigma,\Sigma'}(K)\|$$
		satisfies $\sup_{x \in K}\|\nabla\Phi(x)\|_{\infty,\infty} \leq L_K$.
		Hence $L_K$ is a certified local Lipschitz constant of $\Phi$ on $K$.
		This makes the Hölder refinement strategy $\mathrm{H\ddot{o}lder}_\rho$ (Algorithm~\ref{alg: hölder strategy}) practically implementable for $L^{p}$ norm computation: one sets $\gamma = 1$ and uses $L_K$ as the local Lipschitz constant $C$ for each box, replacing the global---and potentially very conservative---estimate required by Proposition~\ref{prop: fval hölder}.
	\end{remark}
    
	We now want to show that this interval enclosure is Hölder continuous.
	
	\begin{proposition}\label{prop: jac hölder}
		Let $\Omega \in \mathbb{IR}^{d_0}$, $\Phi : \mathbb{R}^{d_0} \to \mathbb{R}^{d_{L+1}}$ be a neural network with differentiable activation function $\sigma : \mathbb{R} \to \mathbb{R}$ and Hölder continuous interval enclosures $\Sigma,\Sigma':\mathbb{IR}\to \mathbb{IR}$ of $\sigma,\sigma'$, respectively. Then for $\ell=0,\dotsc,L$ the function $\mathrm{Jac}_{\Phi,\ell, \Sigma, \Sigma'}$ is Hölder continuous on $\Omega$.
	\end{proposition}
	\begin{proof}
		Let $K\in \mathbb{IR}^{d_0}, \, K \subset \Omega$. We have $\mathrm{Jac}_{\Phi,L, \Sigma, \Sigma'}(K)=W^{(L)}$ and thus
		$$ w(\mathrm{Jac}_{\Phi,L, \Sigma, \Sigma'}(K)) = 0, $$
		which is in particular Hölder continuous for every exponent in $(0,1]$ and constant zero.
		Let $\ell \in \{0, \dotsc,L-1\}$ be arbitrary but fixed and assume
		\begin{equation}\label{eq: jac hölder induction hypothesis}
			w(\mathrm{Jac}_{\Phi,\ell+1, \Sigma, \Sigma'}(K)) \leq C_{\ell+1} w(K)^{\gamma_{\ell+1}}
		\end{equation}
		with $C_{\ell +1} >0$ and $\gamma_{\ell +1}\in (0,1]$.
		By Equation~\eqref{eq: algo jac recursion} we have
		\begin{align*}
			w(\mathrm{Jac}_{\Phi,\ell, \Sigma, \Sigma'}(K)) &=
			w\left( \mathrm{Jac}_{\Phi,{\ell}+1, \Sigma, \Sigma'}({K})  \operatorname{diag}(\Sigma'(\mathrm{Fval}_{\Phi , {\ell}+1, \Sigma}({K}))) W^{(\ell)}\right)
		\end{align*}
		and applying Equation~\eqref{eq: matrix width inequality} yields
		\begin{align*}
			w\left( \mathrm{Jac}_{\Phi,{\ell}+1, \Sigma, \Sigma'}({K})  \operatorname{diag}(\Sigma'(\mathrm{Fval}_{\Phi , {\ell}+1, \Sigma}({K}))) W^{(\ell)}\right) 
			&\leq ||\mathrm{Jac}_{\Phi,{\ell}+1, \Sigma, \Sigma'}({K})||w(\Delta^{(\ell+1)}(K)) \\ 
			&\quad + ||\Delta^{(\ell+1)}(K)|| w(\mathrm{Jac}_{\Phi,{\ell}+1, \Sigma, \Sigma'}({K})),
		\end{align*}
		where we define $\Delta^{(\ell+1)}(K) = \operatorname{diag}(\Sigma'(\mathrm{Fval}_{\Phi , {\ell}+1, \Sigma}({K}))) W^{(\ell)}$ for brevity.
		We have
		$$w(\Delta^{(\ell+1)}(K))  \leq ||\operatorname{diag}(\Sigma'(\mathrm{Fval}_{\Phi , {\ell}+1, \Sigma}({K})))|| w(W^{(\ell)}) + ||W^{(\ell)}||w(\operatorname{diag}(\Sigma'(\mathrm{Fval}_{\Phi , {\ell}+1, \Sigma}({K})))) 
		$$
		by Equation~\eqref{eq: matrix width inequality}. Considering $w(W^{(\ell)})=0$ and $w(\operatorname{diag}(\Sigma'(\mathrm{Fval}_{\Phi , {\ell}+1, \Sigma}({K})))) =w(\Sigma'(\mathrm{Fval}_{\Phi , {\ell}+1, \Sigma}({K})))$ we obtain
		\begin{align*}
			&||\operatorname{diag}(\Sigma'(\mathrm{Fval}_{\Phi , {\ell}+1, \Sigma}({K}))|| w(W^{(\ell)}) + ||W^{(\ell)}||w(\operatorname{diag}(\Sigma'(\mathrm{Fval}_{\Phi , {\ell}+1, \Sigma}({K})))) \\
			= &||W^{(\ell)}||w(\Sigma'(\mathrm{Fval}_{\Phi , {\ell}+1, \Sigma}({K}))).
		\end{align*}
		By assumption $\Sigma'$ is Hölder continuous and we showed in Proposition~\ref{prop: fval hölder} that $\mathrm{Fval}_{\Phi , {\ell}+1, \Sigma}({K})$ is Hölder continuous. Thus there exist $\tilde{C}_{\ell +1} > 0$ and $\tilde{\gamma}_{(\ell+1)} \in (0,1]$ independent of $K$ such that
		$$ ||W^{(\ell)}||w(\Sigma'(\mathrm{Fval}_{\Phi , {\ell}+1, \Sigma}({K}))) \leq \tilde{C}_{\ell +1}w(K)^{\tilde{\gamma}_{(\ell+1)}}. $$
		It remains to bound the matrix norms of $||\mathrm{Jac}_{\Phi,{\ell}+1, \Sigma, \Sigma'}({K})||$ and $||\Delta^{(\ell+1)}(K)||$. Since both $\mathrm{Jac}_{\Phi,{\ell}+1, \Sigma, \Sigma'}$ and $\Delta^{(\ell+1)}(K)$ are inclusion isotonic, we have
		$$||\mathrm{Jac}_{\Phi,{\ell}+1, \Sigma, \Sigma'}({K})|| \leq ||\mathrm{Jac}_{\Phi,{\ell}+1, \Sigma, \Sigma'}({\Omega})|| \quad \text{and} \quad ||\Delta^{(\ell+1)}(K)|| \leq ||\Delta^{(\ell+1)}(\Omega)||.$$
		Note, that the inclusion isotonicity of $\Delta^{(\ell+1)}$ follows from the inclusion isotonicity of $\mathrm{Fval}_{\Phi, \ell+1, \Sigma}$.
	\end{proof}

	We can now apply Theorem~\ref{thm: main conv} to compute the $W^{1,p}$ norm of neural networks for $1\leq p < \infty$.
	
	\begin{corollary}\label{corr:mainW1p}
		Let $1<p\leq \infty$, $\Phi : \mathbb{R}^{d_0} \to \mathbb{R}^{d_{L+1}}$ be a neural network with differentiable activation function $\sigma : \mathbb{R} \to \mathbb{R}$ and Hölder continuous interval enclosures $\Sigma,\Sigma':\mathbb{IR}\to \mathbb{IR}$ of $\sigma,\sigma'$, respectively.
		For $i=1, \dotsc, d_{L+1}$ and $\alpha \in\mathbb{N}^{d_0}_0$ set
		$$ F_{\Phi,\alpha,i}(K) = \mathrm{Fval}_{\Phi, L+1, \Sigma}(K)_i$$
		for $\alpha = 0$.
		For $|\alpha|=1$ with non-zero index $j$ set
		$$ F_{\Phi,\alpha,i}(K) = \mathrm{Jac}_{\Phi,0, \Sigma, \Sigma'}(K)_{ij}.$$
		Then $f_{\Phi,1,p}$ is continuous and $F_{\Phi,1,p}$ is a Hölder continuous interval enclosure of $f_{\Phi,1,p}$. Thus, $P_{1,p}=(f_{\Phi, 1,p}, F_{\Phi, 1,p}, \Omega)$ and $A_{\theta,\rho}$ are certifiable and convergent. In particular, we have
		$$ |\, ||\Phi||_{W^{1,p}(\Omega; \mathbb{R}^{d_{L+1}})} - \mathrm{Q}_n^{1/p}| \leq \eta_n^{1/p} $$
		and $\eta_n \leq \eta_0 q^n$ where $(\mathrm{Q}_n, \eta_n) = \mathrm{AdaQuad}_n(P_{1,p}, A_{\theta, \rho})$ for $n\in \mathbb{N}$ and $q=1-\theta(1-\rho)$.
	\end{corollary}
	\begin{proof}
		Corollary~\ref{cor: jac int encl} and Proposition~\ref{prop: jac hölder} yield that $F_{\Phi, 1,p}$ with $F_{\Phi,\alpha,i}$ is a Hölder continuous interval enclosure of $f_{\Phi,1,p}$. By Theorem~\ref{thm: main conv} the pair $P_{1,p}$ and $A_{\theta,\rho}$ are certifiable and converging. This finishes the proof.
	\end{proof}

	\subsection{Neural network Hessian bounds}\label{subsec:hessianBounds}
	
	We aim to construct an interval enclosure for the Hessian $\nabla^2\Phi (K)$ with $ K \in \mathbb{IR}^{d_{0}}$ for a neural network $\Phi : \mathbb{R}^{d_0} \to \mathbb{R}^{d_{L+1}}$.
	The tail neural network $\Phi^{(0)}$ satisfies $ \Phi^{(0)} =\Phi $, which implies
	\begin{equation}\label{eq: full tail hessian}
		\nabla^2\Phi^{(0)}(x^{(0)}) = \nabla^2 \Phi (x) 
	\end{equation}
	for $x^{(0)}=x\in \mathbb{R}^{d_0}$.
	Recall, for $\ell=0, \dotsc, L-1$ we have the recursion
	$$
	\Phi^{(\ell)} = \Phi^{(\ell+1)} \circ \sigma \circ T_{\ell}.
	$$
	For the $i$-th component of the neural network this implies
	\begin{align*}
		\nabla^2  \Phi_i^{(\ell)}(u) =& \,W^{(\ell)T} D^{(\ell+1)}(T_\ell(u)) \nabla^2\Phi^{(\ell+1)}_i( \sigma( T_{\ell}(u)))D^{(\ell+1)}(T_\ell(u))W^{(\ell)} \\
		&+W^{(\ell)T} \operatorname{diag}\left(\sigma'' (T_\ell(u)) \odot \nabla \Phi^{(\ell+1)}_i(\sigma (T_\ell(u)))\right) W^{(\ell)}
	\end{align*}
	for $u \in \mathbb{R}^{d_\ell}$ by the chain rule. In particular, for $u \in \mathbb{R}^{d_0}$ and $i=1, \dotsc, d_{L+1}$ we have
	\begin{equation}\label{eq: shortest tail hess}
		\nabla^2 \Phi^{(L)}_i(x^{(L)}(u)) = 0 \in \mathbb{R}^{d_{L}\times d_{L}}
	\end{equation}
	and
	\begin{equation}\label{eq: hessian recursion}
		\begin{aligned}
			\nabla^2  \Phi_i^{(\ell)}(x^{(\ell)}(u))
			&= W^{(\ell)T} D^{(\ell+1)}(z^{(\ell+1)}(u)) \nabla^2\Phi^{(\ell+1)}_i( x^{(\ell+1)}(u))
			D^{(\ell+1)}(z^{(\ell+1)}(u))W^{(\ell)} \\
			&\quad + W^{(\ell)T} \operatorname{diag}\!\left(\sigma'' (z^{(\ell+1)}(u)) \odot \nabla \Phi^{(\ell+1)}_i(x^{(\ell+1)}(u))\right) W^{(\ell)}.
		\end{aligned}
	\end{equation}
	for $\ell=0, \dotsc, L-1$ with $\odot$ denoting {entrywise} multiplication. Starting with Equation~\eqref{eq: shortest tail hess} and repeatedly applying Equation~\eqref{eq: hessian recursion} allows us to compute $\nabla^2 \Phi(u)$ by Equation~\eqref{eq: full tail hessian}. To bound the Hessian $\nabla^2 \Phi_i(K)$ for $K \in \mathbb{IR}^{d_0}$ we replace the pre-activations $z^{(\ell)}(K)$ by $\mathrm{Fval}_{\Phi, \ell, \Sigma}(K)$, tail Jacobians $\nabla \Phi^{(\ell+1)}_i(x^{(\ell+1)}(K))$ by $\mathrm{Jac}_{\Phi, \ell +1, \Sigma, \Sigma'}(K)$ and carry out the matrix multiplications via interval arithmetic (see Proposition~\ref{prop: endpoint formulas}).

	\vspace{1em}
	\begin{algorithm}[H]\label{alg: nn hess bounds}
		\LinesNotNumbered
		\DontPrintSemicolon
		\caption{$\mathrm{Hess}_{\Phi,i,\ell, \Sigma, \Sigma', \Sigma''}(K)$}
		\Input{
			Interval vector $K\in \mathbb{IR}^{d_0}$; neural network $\Phi : \mathbb{R}^{d_0} \to \mathbb{R}^{d_{L+1}}$ of depth $L$ with activation function $\sigma : \mathbb{R} \to \mathbb{R}$; output component $i \in \{1, \dotsc, d_{L+1}\}$; layer $\ell \in\{ 0, \dotsc, L\}$; interval enclosures $\Sigma,\Sigma',\Sigma'' : \mathbb{IR} \to \mathbb{IR}$ of $\sigma$, $\sigma'$, $\sigma''$, respectively.
		}
		\Output{
			Interval matrix $H^{(\ell)}\in \mathbb{IR}^{d_{\ell} \times d_{\ell}}$ containing the Hessians for $i$-th output component of the $\ell$-th tail neural network with inputs from $x^{(\ell)}(K)$.
		}
		\vspace{1em}
		$H^{(L)} = 0 \in \mathbb{R}^{d_L \times d_L}$\;
		\For{$\ k = L-1, \dotsc, \ell$}{
			$Z^{(k+1)} = \mathrm{Fval}_{\Phi, k+1,\Sigma}(K)$\;
			$J^{(k+1)}_i = \mathrm{Jac}_{\Phi, k+1, \Sigma, \Sigma'}(K)_{i:}$ (extract $i$-th row)\;
			$H^{(k)} = W^{(k)T}  \operatorname{diag}(\Sigma'(Z^{(k+1)})) H^{(k+1)} \operatorname{diag}(\Sigma'(Z^{(k+1)})) ~W^{(k)} + W^{(k)T}~\operatorname{diag}\left( \Sigma''(Z^{(k+1)}) \odot J^{(k+1)}_i\right)  W^{(k)}$\;
		}
		\Return $H^{(\ell)}$\;
	\end{algorithm}
	\vspace{1em}
	
	We now want to show that for a neural network $\Phi : \mathbb{R}^{d_0} \to \mathbb{R}^{d_{L+1}}$ the function $\mathrm{Hess}_{\Phi,i,\ell, \Sigma, \Sigma', \Sigma''}$ is an interval extension of the Hessian of the $i$-th component of the $\ell$-th tail neural network for inputs from $x^{(\ell)}(K)$ and $K \in \mathbb{IR}^{d_0}$, $i=1, \dotsc, d_{L+1}$, $\ell = 0, \dotsc, L$.

	\begin{proposition}\label{prop: hess interval extension}
		Let $\Phi : \mathbb{R}^{d_0} \to \mathbb{R}^{d_{L+1}}$ be a neural network with twice differentiable activation function $\sigma : \mathbb{R} \to \mathbb{R}$ and interval extensions $\Sigma,\Sigma',\Sigma'':\mathbb{IR}\to \mathbb{IR}$ of $\sigma,\sigma',\sigma''$, respectively. Then for $i=1, \dotsc, d_{L+1}$, $\ell=0,\dotsc, L$ we have that $\mathrm{Hess}_{\Phi,i,\ell, \Sigma, \Sigma', \Sigma''}$ is an interval extension of the Hessian of the $i$-th component of the $\ell$-th tail neural network evaluated at $x^{(\ell)}$, i.e.
		$$ \nabla^2 \Phi_i^{(\ell)}(x^{(\ell)}(\{u\})) = \mathrm{Hess}_{\Phi,i,\ell, \Sigma, \Sigma', \Sigma''}(\{u\}) $$
		for all $u\in \mathbb{R}^{d_0}$.
	\end{proposition}
	
	\begin{proof}
		Let $u \in \mathbb{R}^{d_0}$. The result follows directly from the fact that for $i=1, \dotsc, d_{L+1}$, $\ell=0, \dotsc, L$ we can compute $\nabla^2 \Phi_i^{(\ell)}(x^{(\ell)}(u))$ via Equation \eqref{eq: shortest tail hess} and recursion \eqref{eq: hessian recursion}, since $\Sigma'$, $\Sigma''$, $\mathrm{Fval}_{\Phi,k,\Sigma}$ and $\mathrm{Jac}_{\Phi,k}$ are interval extensions for $k=L-1, \dotsc, \ell$.
	\end{proof}

	Next, we show that this interval extension is inclusion isotonic.
	
	\begin{proposition}\label{prop: hess inclusion isotonic}
		Let $\Phi : \mathbb{R}^{d_0} \to \mathbb{R}^{d_{L+1}}$ be a neural network with twice differentiable activation function $\sigma : \mathbb{R} \to \mathbb{R}$ and inclusion isotonic interval extensions $\Sigma,\Sigma',\Sigma'':\mathbb{IR}\to \mathbb{IR}$ of $\sigma,\sigma', \sigma''$, respectively. Then for $i=1, \dotsc, d_{L+1}$, $\ell=0,\dotsc, L$ we have that $\mathrm{Hess}_{\Phi,i,\ell, \Sigma, \Sigma', \Sigma''}$ is inclusion isotonic, i.e.
		\begin{equation}
			\mathrm{Hess}_{\Phi,i,\ell, \Sigma, \Sigma', \Sigma''}(\tilde{K}) \subset \mathrm{Hess}_{\Phi,i,\ell, \Sigma, \Sigma', \Sigma''}(K) \quad \text{if}\quad \tilde{K} \subset K
		\end{equation}
		for $\tilde{K},K \in \mathbb{IR}^{d_0}$.
	\end{proposition}
	\begin{proof}
		Let $\tilde{K},K \in \mathbb{IR}^{d_0}$ satisfy $\tilde{K} \subset K$ and $i \in \{1, \dotsc, d_{L+1}\}$. We have
		$$ \mathrm{Hess}_{\Phi,i,L, \Sigma, \Sigma', \Sigma''}(\tilde{K}) = \mathrm{Hess}_{\Phi,i,L, \Sigma, \Sigma', \Sigma''}({K}) = 0. $$
		Now assume
		\begin{equation}\label{eq: hess induction hypothesis}
			\mathrm{Hess}_{\Phi, i,\ell+1, \Sigma, \Sigma', \Sigma''}(\tilde{K}) \subset \mathrm{Hess}_{\Phi, i,\ell+1, \Sigma, \Sigma', \Sigma''}(K)
		\end{equation}
		for an arbitrary but fixed $\ell \in \{0, \dotsc, L-1\}$. Note that by definition we have
		\begin{equation}\label{eq: algo hess recursion}
			\begin{aligned}
				\mathrm{Hess}_{\Phi, i,\hat{\ell}, \Sigma, \Sigma', \Sigma''}(\hat{K})
				&= W^{(\hat{\ell})T} \operatorname{diag}\left( \Sigma'(\mathrm{Fval}_{\Phi, \hat{\ell}+1}(\hat{K}))\right)  \mathrm{Hess}_{\Phi, i,\hat{\ell}+1, \Sigma, \Sigma', \Sigma''}(\hat{K}) \\
				&\quad  \operatorname{diag}\left( \Sigma'(\mathrm{Fval}_{\Phi, \hat{\ell}+1}(\hat{K}))\right) W^{(\hat{\ell})}\\
				&\quad + W^{(\hat{\ell})T}  \operatorname{diag}\left(\Sigma'' (\mathrm{Fval}_{\Phi, \hat{\ell}+1}(\hat{K})) \odot \mathrm{Jac}_{\Phi, \hat{\ell}+1, \Sigma, \Sigma'}(\hat{K})\right)  W^{(\hat{\ell})},
			\end{aligned}
		\end{equation}
		for $\hat{K}\in \mathbb{IR}^{d_0}$ and $\hat{\ell} = 0, \dotsc, L-1$. We proceed analogously to the proof of Proposition~\ref{prop: jac inclusion isotonic} and apply the inclusion isotonicity of $\Sigma'$, $\Sigma''$, $\mathrm{Fval}_{\Phi, \ell+1, \Sigma}$ and $\mathrm{Jac}_{\Phi,i, \ell+1, \Sigma, \Sigma'}$. This yields
		$$  \mathrm{Hess}_{\Phi, i,\ell, \Sigma, \Sigma', \Sigma''}(\tilde{K}) \subset \mathrm{Hess}_{\Phi, i,\ell, \Sigma, \Sigma', \Sigma''}(K) $$
		and finishes the proof.
	\end{proof}
	
	From the Fundamental Theorem of Interval Analysis it follows that for $i=1, \dotsc, d_{L+1}$, $\ell=0,\dotsc,L$ the function $\mathrm{Hess}_{\Phi,i,\ell, \Sigma, \Sigma', \Sigma''}$ is an interval enclosure of $\nabla^2\Phi_i^{(\ell)}(x^{(\ell)}(K))$ for neural network $\Phi$ and $K \in \mathbb{IR}^d$.

	\begin{corollary}\label{cor: hess int encl}
		Let $\Phi : \mathbb{R}^{d_0} \to \mathbb{R}^{d_{L+1}}$ be a neural network with twice differentiable activation function $\sigma : \mathbb{R} \to \mathbb{R}$ and interval enclosures $\Sigma,\Sigma',\Sigma'':\mathbb{IR}\to \mathbb{IR}$ of $\sigma,\sigma',\sigma''$, respectively. Then for $i=1, \dotsc, d_{L+1}$, $\ell=0,\dotsc,L$ the function $\mathrm{Hess}_{\Phi,i,\ell, \Sigma, \Sigma',\Sigma''}(K)$ is an interval enclosure of $\nabla^2\Phi_i^{(\ell)}(x^{(\ell)}(K))$ for $K\in \mathbb{IR}^{d_0}$.
	\end{corollary}
	\begin{proof}
		Since by Proposition~\ref{prop: hess interval extension} and Proposition~\ref{prop: hess inclusion isotonic} the function $\mathrm{Hess}_{\Phi,i,\ell, \Sigma, \Sigma',\Sigma''}$ is an inclusion isotonic interval extension of $\nabla^2\Phi_i^{(\ell)}(x^{(\ell)}(.))$. The Fundamental Theorem of Interval Analysis (Theorem \ref{thm: fund thm of IA}) yields that $\mathrm{Hess}_{\Phi,i,\ell, \Sigma, \Sigma', \Sigma''}$ is an interval enclosure of $\nabla^2\Phi_i^{(\ell)}(x^{(\ell)}(.))$, finishing the proof.
	\end{proof}
	
	We now want to show that this interval enclosure is Hölder continuous.
	
	\begin{proposition}\label{prop: hess hölder}
		Let $\Omega \in \mathbb{IR}^{d_0}$, $\Phi : \mathbb{R}^{d_0} \to \mathbb{R}^{d_{L+1}}$ be a neural network with twice differentiable activation function $\sigma : \mathbb{R} \to \mathbb{R}$ and Hölder continuous interval enclosures $\Sigma,\Sigma',\Sigma'':\mathbb{IR}\to \mathbb{IR}$ of $\sigma,\sigma',\sigma''$, respectively. Then for $i=1, \dotsc, d_{L+1}$, $\ell=0,\dotsc,L$ the function $\mathrm{Hess}_{\Phi,i,\ell, \Sigma, \Sigma', \Sigma''}$ is Hölder continuous on $\Omega$.
	\end{proposition}
	
	\begin{proof}
		We proceed analogously to the proof of Proposition~\ref{prop: jac hölder}. Let $K\in \mathbb{IR}^{d_0}, \, K \subset \Omega$, and $i \in \{1, \dotsc, d_{L+1}\}$. We have $\mathrm{Hess}_{\Phi,L, \Sigma, \Sigma', \Sigma''}(K)=0 \in \mathbb{R}^{d_L \times d_L}$, which is in particular Hölder continuous for every exponent in $(0,1]$ and constant zero. Let $\ell \in \{0, \dotsc,L-1\}$ be arbitrary but fixed and assume
		\begin{equation}\label{eq: hess hölder induction hypothesis}
			w(\mathrm{Hess}_{\Phi,i,\ell+1, \Sigma, \Sigma', \Sigma''}(K)) \leq C_{\ell+1} w(K)^{\gamma_{\ell+1}}.
		\end{equation}
		Considering Remark~\ref{rem: width linearity} and Equation~\eqref{eq: algo hess recursion}, we obtain
		\begin{equation*}
			\begin{aligned}
				w(\mathrm{Hess}_{\Phi, i,{\ell}, \Sigma, \Sigma', \Sigma''}({K}))
				&\leq w\bigg(W^{({\ell})T} \operatorname{diag}\left( \Sigma'(\mathrm{Fval}_{\Phi, {\ell}+1}({K}))\right)  \mathrm{Hess}_{\Phi, i,{\ell}+1, \Sigma, \Sigma', \Sigma''}({K}) \\
				&\quad  \operatorname{diag}\left( \Sigma'(\mathrm{Fval}_{\Phi, {\ell}+1}({K}))\right) W^{({\ell})}\bigg)\\
				&\quad + w\bigg(W^{({\ell})T} \operatorname{diag}\left(\Sigma'' (\mathrm{Fval}_{\Phi, {\ell}+1}({K})) \odot \mathrm{Jac}_{\Phi, {\ell}+1, \Sigma, \Sigma'}({K})\right)  W^{({\ell})}\bigg),
			\end{aligned}
		\end{equation*}
		Since the width of point matrices is zero, repeatedly applying Equation~\eqref{eq: matrix width inequality} and \eqref{eq: matrix magnitude inequality} yields
		\begin{equation*}
			\begin{aligned}
				& w\bigg(W^{({\ell})T} \operatorname{diag}\left( \Sigma'(\mathrm{Fval}_{\Phi, {\ell}+1}({K}))\right)  \mathrm{Hess}_{\Phi, i,{\ell}+1, \Sigma, \Sigma', \Sigma''}({K}) 
				\operatorname{diag}\left( \Sigma'(\mathrm{Fval}_{\Phi, {\ell}+1}({K}))\right) W^{({\ell})}\bigg) \\
				&\leq ||W^{({\ell})}||^2 ||\Sigma'(\mathrm{Fval}_{\Phi, {\ell}+1}({K}))||^2 w(\mathrm{Hess}_{\Phi, i,{\ell}+1, \Sigma, \Sigma', \Sigma''}({K})) \\
				&+ 2 ||W^{({\ell})}||^2 ||\Sigma'(\mathrm{Fval}_{\Phi, {\ell}+1}({K}))|| \,||\mathrm{Hess}_{\Phi, i,{\ell}+1, \Sigma, \Sigma', \Sigma''}({K}) ||\, w(\Sigma' (\mathrm{Fval}_{\Phi, {\ell}+1}({K}))).
			\end{aligned}
		\end{equation*}
		By Proposition~\ref{prop: fval hölder} and by Equation~\eqref{eq: hess induction hypothesis}, we have that $\Sigma'(\mathrm{Fval}_{\Phi,\ell+1})$ and $\mathrm{Hess}_{\Phi, i,{\ell}+1, \Sigma, \Sigma', \Sigma''}$ are Hölder continuous. By inclusion isotonicity of $\Sigma'(\mathrm{Fval}_{\Phi, \ell+1, \Sigma})$ and $\mathrm{Hess}_{\Phi, i,{\ell}+1, \Sigma, \Sigma', \Sigma''}$, we have that
		$$ ||W^{({\ell})}||^2 ||\Sigma'(\mathrm{Fval}_{\Phi, {\ell}+1}({K}))||^2  \leq ||W^{({\ell})}||^2 ||\Sigma'(\mathrm{Fval}_{\Phi, {\ell}+1}({\Omega}))||^2$$
		and 
		\begin{equation*}
			\begin{aligned}
				&2 ||W^{({\ell})}||^2 ||\Sigma'(\mathrm{Fval}_{\Phi, {\ell}+1}({K}))|| \,||\mathrm{Hess}_{\Phi, i,{\ell}+1, \Sigma, \Sigma', \Sigma''}({K}) || \\
				&\leq 2 ||W^{({\ell})}||^2 ||\Sigma'(\mathrm{Fval}_{\Phi, {\ell}+1}({\Omega}))|| \,||\mathrm{Hess}_{\Phi, i,{\ell}+1, \Sigma, \Sigma', \Sigma''}({\Omega}) ||,
			\end{aligned}
		\end{equation*}
		bounding the coefficients. Proceeding analogously for the second summand yields
		\begin{equation*}
			\begin{aligned}
				& w\bigg(W^{({\ell})T} \operatorname{diag}\left(\Sigma'' (\mathrm{Fval}_{\Phi, {\ell}+1}({K})) \odot \mathrm{Jac}_{\Phi, {\ell}+1, \Sigma, \Sigma'}({K})\right)  W^{({\ell})}\bigg) \\
				&\leq ||W^{(\ell)}||^2 w\left(\Sigma'' (\mathrm{Fval}_{\Phi, {\ell}+1}({K})) \odot \mathrm{Jac}_{\Phi, {\ell}+1, \Sigma, \Sigma'}({K})\right).
			\end{aligned}
		\end{equation*}
		Moreover, by Equation~\eqref{eq: matrix width inequality} we have
		\begin{equation*}
			\begin{aligned}
				&w\left(\Sigma'' (\mathrm{Fval}_{\Phi, {\ell}+1}({K})) \odot \mathrm{Jac}_{\Phi, {\ell}+1, \Sigma, \Sigma'}({K})\right) \\
				&\leq || \Sigma'' (\mathrm{Fval}_{\Phi, {\ell}+1}({K}))||\, w(\mathrm{Jac}_{\Phi, {\ell}+1, \Sigma, \Sigma'}({K})) \\
				& +|| \mathrm{Jac}_{\Phi, {\ell}+1, \Sigma, \Sigma'}({K})||\, w(\Sigma'' (\mathrm{Fval}_{\Phi, {\ell}+1}({K}))).
			\end{aligned}
		\end{equation*}
		Again, by inclusion isotonicity we can bound the coefficients from above independently of $K$:
		$$ || \Sigma'' (\mathrm{Fval}_{\Phi, {\ell}+1}({K}))|| \leq || \Sigma'' (\mathrm{Fval}_{\Phi, {\ell}+1}({\Omega}))|| $$
		and 
		$$ || \mathrm{Jac}_{\Phi, {\ell}+1, \Sigma, \Sigma'}({K})|| \leq || \mathrm{Jac}_{\Phi, {\ell}+1, \Sigma, \Sigma'}({\Omega})||.$$
		By Propositions~\ref{prop: fval hölder} and \ref{prop: jac hölder}, we have that $\Sigma'(\mathrm{Fval}_{\Phi,\ell+1})$ and $\mathrm{Hess}_{\Phi, i,{\ell}+1, \Sigma, \Sigma', \Sigma''}$ are Hölder continuous. Considering that sums of Hölder continuous functions on a compact domain are again Hölder continuous finishes the proof.    
	\end{proof}

	We can now apply Theorem~\ref{thm: main conv} to compute the $W^{2,p}$ norm of neural networks for $1\leq p < \infty$.
	
	\begin{corollary}\label{corr:mainW2p}
		Let $1\leq p < \infty$, $\Phi : \mathbb{R}^{d_0} \to \mathbb{R}^{d_{L+1}}$ be a neural network with twice differentiable activation function $\sigma : \mathbb{R} \to \mathbb{R}$ and Hölder continuous interval enclosures $\Sigma,\Sigma',\Sigma'':\mathbb{IR}\to \mathbb{IR}$ of $\sigma,\sigma',\sigma''$, respectively.
		For $i=1, \dotsc, d_{L+1}$ and $\alpha \in\mathbb{N}^{d_0}_0$ set
		$$ F_{\Phi,\alpha,i}(K) = \mathrm{Fval}_{\Phi, L+1, \Sigma}(K)_i$$
		for $\alpha = 0$.
		For $|\alpha|=1$ with non-zero index $j$ set
		$$ F_{\Phi,\alpha,i}(K) = \mathrm{Jac}_{\Phi,0, \Sigma, \Sigma'}(K)_{ij}.$$
		For $|\alpha|=2$ with non-zero indices $j,k$
		set
		$$ F_{\Phi,\alpha,i}(K) = \mathrm{Hess}_{\Phi,i,0, \Sigma, \Sigma',\Sigma''}(K)_{jk}.$$
		Then $f_{\Phi,2,p}$ is continuous and $F_{\Phi,2,p}$ is a Hölder continuous interval enclosure of $f_{\Phi,2,p}$. Thus, $P_{2,p}=(f_{\Phi, 2,p}, F_{\Phi, 2,p}, \Omega)$ and $A_{\theta, \rho}$ are certifiable and convergent. In particular, we have
		$$ |\, ||\Phi||_{W^{2,p}(\Omega; \mathbb{R}^{d_{L+1}})} - \mathrm{Q}_n^{1/p}| \leq \eta_n^{1/p} $$
		and $\eta_n \leq \eta_0 q^n$ where $(\mathrm{Q}_n, \eta_n) = \mathrm{AdaQuad}_n(P_{2,p}, A_{\theta, \rho})$ for $n\in \mathbb{N}$ and $q=1-\theta(1-\rho)$.
	\end{corollary}
	\begin{proof}
		Corollary~\ref{cor: hess int encl} and Proposition~\ref{prop: hess hölder} yield that $F_{\Phi,2,p}$ with $F_{\Phi,\alpha,i}$ is a Hölder continuous interval enclosure of $f_{\Phi,2,p}$. By Theorem~\ref{thm: main conv} the pair $P_{2,p}$ and $A_{\theta,\rho}$ are certifiable and converging. This finishes the proof.
	\end{proof}

	\subsection{Neural network residual bounds}
	
	Here we show how to combine the previously established neural network bounding techniques to obtain certified convergence of $\mathrm{AdaQuad}$ to the integral of the residual of a differential equation. Recall the boundary-value problem \eqref{eq:BVP} with differential operator in nondivergence form given by
	\begin{equation*}
		\mathcal{D}\Phi(x) = - \sum_{i,j=1}^d a_{ij}(x) \Phi_{x_i x_j}(x)  + \sum_{i=1}^d b_i(x) \Phi_{x_i}(x) + c(x)\Phi(x),
	\end{equation*}
	for given coefficient functions $a_{ij}, b_i,c$ for $i,j=1, \dotsc, d$ and trial function $\Phi$ as well as right-hand side $g$. We will show that for the residual $r_{\Phi,p}=|\mathcal{D}\Phi - g|^p$ and its enclosure $R_{\Phi, p}$ (see Proposition~\ref{prop: hölder energy}) together with the algorithm instance $A_{\theta,\rho}$ is certifiable and converging.

	\begin{corollary}\label{corr:mainEnergy}
		Let $1 < p < \infty$, $d\in \mathbb{N}$, $\Omega \in \mathbb{IR}^d$, $\mathcal{D}$ given by \eqref{eq:DiffOperator}, $d_0=d$, $d_{L+1}=1$ and $\Phi : \mathbb{R}^{d_0} \to \mathbb{R}^{d_{L+1}}$ be a neural network with twice differentiable activation function $\sigma : \mathbb{R} \to \mathbb{R}$ and Hölder continuous interval enclosures $\Sigma,\Sigma',\Sigma'':\mathbb{IR}\to \mathbb{IR}$ of $\sigma,\sigma',\sigma''$, respectively.
		Assume for $i,j=1,\dotsc,d_0$ interval enclosures $A_{ij},B_i,C,G: \mathbb{IR}^d \to \mathbb{IR}$ of $a_{ij},b_i,c,g$, respectively, which are Hölder continuous on $\Omega$. Moreover, let $R_{\Phi,p}$ be defined by replacing in 
		$r_{\Phi,p}$ the functions $a_{ij},b_i,c,g$ with $A_{ij}, B_i,C,G$ for $i,j=1,\dotsc,d$, and the function value, first- and second-order partial derivatives of the neural network by the Hölder continuous interval 
		enclosures $\mathrm{Fval}_{\Phi, L+1, \Sigma}$, $\mathrm{Jac}_{\Phi,0,\Sigma, \Sigma'}$, 
		$\mathrm{Hess}_{\Phi,1,0,\Sigma',\Sigma''}$, 
		respectively. Then $r_{\Phi,p}$ is continuous and $R_{\Phi,p}$ is a Hölder continuous interval enclosure of $r_{\Phi,p}$. 
		Thus, $P_{\mathcal{D},p}=(r_{\Phi,p}, R_{\Phi,p}, \Omega)$ and {$A_{\theta, \rho}$} are certifiable and convergent. 
		In particular, we have
		$$ |\, ||r_{\Phi,p}||_{L^{p}(\Omega; \mathbb{R})} - \mathrm{Q}_n^{1/p}| \leq \eta_n^{1/p} $$
		and $\eta_n\leq \eta_0 q^n$ where $(\mathrm{Q}_n, \eta_n) = \mathrm{AdaQuad}_n(P_{\mathcal{D},p}, A_{\theta, \rho})$ for $n\in \mathbb{N}$ and $q = 1 - \theta(1-\rho)$.
	\end{corollary}
	\begin{proof}
		By construction, $R_{\Phi,p}$ satisfies Proposition~\ref{prop: hölder energy} and is thus a Hölder continuous interval enclosure of $r_{\Phi,p}$. It is easy to see that therefore, $r_{\Phi,p}$ is Hölder continuous as well and, in particular, continuous. 
		By Theorem~\ref{thm: main conv} the pair $P_{\mathcal{D},p}$ and $A_{\theta,\rho}$ are certifiable and converging. This finishes the proof.
	\end{proof}

	\subsection{Numerical Experiments}
	\label{sec:numerical_experiments}
	
	This section presents a series of numerical experiments to validate the theoretical performance of the adaptive refinement and bounding algorithms. Our code is available under \url{https://github.com/RubenDeAngelo/Function-Space-Bounds-for-NN }.
	The primary quantity of interest is the gap between our upper and lower bounds.
	Let $\overline{B}, \underline{B} \in \mathbb{R}$ denote local upper and lower bounds computed by Algorithm~\ref{alg: adaptive integration} on a partition $\mathcal{P}$ of a domain $\Omega \in \mathbb{IR}^d$.
	For the norms considered here---$L^p$-, Sobolev-, and energy-norms---we evaluate bounds for the corresponding $p$-th powers.
	We define the global error $G \in \mathbb{R}$ by
	\begin{align*}
		G = \sum_{K \in \mathcal{P}} \big(\overline{B}(K) - \underline{B}(K)\big).
	\end{align*}
	The global error directly quantifies the error of our algorithm. 
	
	First, we provide a one-dimensional analysis of Sobolev (Corollary~\ref{corr:mainW1p}, \ref{corr:mainW2p}) and $L^p$ (Corollary~\ref{corr:LpNormsconvergenceOfQuadratureRule}) norm bounds in Sections~\ref{sec: Sobolev bounds for untrained and trained neural networks (1d)} and~\ref{sec: $L^p$ bounds for untrained and trained networks (1d)}.
	We analyze the convergence of the global error for both untrained and trained networks in deep and wide architectures for increasingly refined partitions.
	Next, in Section~\ref{sec: Sobolev and $L^p$ bounds for trained networks (2d)}, we study convergence of the global error and the local error for deep and wide architectures fitted to the two-dimensional smooth disk function~\eqref{eq: smooth disk}.
	Finally, in Section~\ref{sec:Bounding_Interior_Residuals_for_Elliptic_PINN}, we examine convergence of the global error and the local error for the energy norm in an elliptic PINN setting.
	
	For the $L^p$ norm bounds, we use the Hölder refinement strategy (Algorithm~\ref{alg: hölder strategy}). For Sobolev and energy-norm bounds, we use Dörfler marking (Algorithm~\ref{alg: dörfler strategy}) with uniform refinement, where each box is split along coordinate axes into $2^d$ subboxes (see \ref{rem:HalfRefinement}). Although this refinement strategy does not provide the controlled error reduction described in Equation~\eqref{eq: local error reduction}, it avoids the need to compute the Hölder constant and exponent. In the Lipschitz-continuous case, available estimates of the Lipschitz constant for neural networks are often overly conservative, which can lead to refinements with an excessive number of subdivisions.

	\subsubsection{Sobolev bounds for untrained and trained neural networks (1d)}
	\label{sec: Sobolev bounds for untrained and trained neural networks (1d)}
	To understand how well our bounds perform in generic settings, we study their behavior for many random initializations of two fixed architectures.
	
	Figures~\ref{fig: deep random network sobolev bound} and~\ref{fig: wide random network sobolev bound} show the mean normalized global errors for the Sobolev norm for untrained (random) tanh networks in deep and wide configurations computed over 100 randomly initialized neural networks. Here, we normalize by a high-precision estimate of the Sobolev norms computed with the Monte-Carlo method using 50000 samples.
	We study two architectures: a deep neural network with three hidden layers of 32 neurons, and a wide neural network with one hidden layer of 200 neurons. 
	Each run uses a unique reproducible seed and PyTorch's default Kaiming-uniform initialization.
	Since random initialization typically produces comparatively small curvature, the observed errors decrease regularly under refinement, and the separation between $W^{1,2}$ and $W^{2,2}$ is less pronounced. 
	Deeper architectures induce more curvature, which leads to larger and more clearly separated mean global errors for $W^{1,p}$ and $W^{2,p}$ than in the wide architecture.

	\begin{figure}[htb]
		\centering
		\begin{subfigure}[t]{0.48\linewidth}\vspace{0pt}
			\centering
			\includegraphics[width=\linewidth]{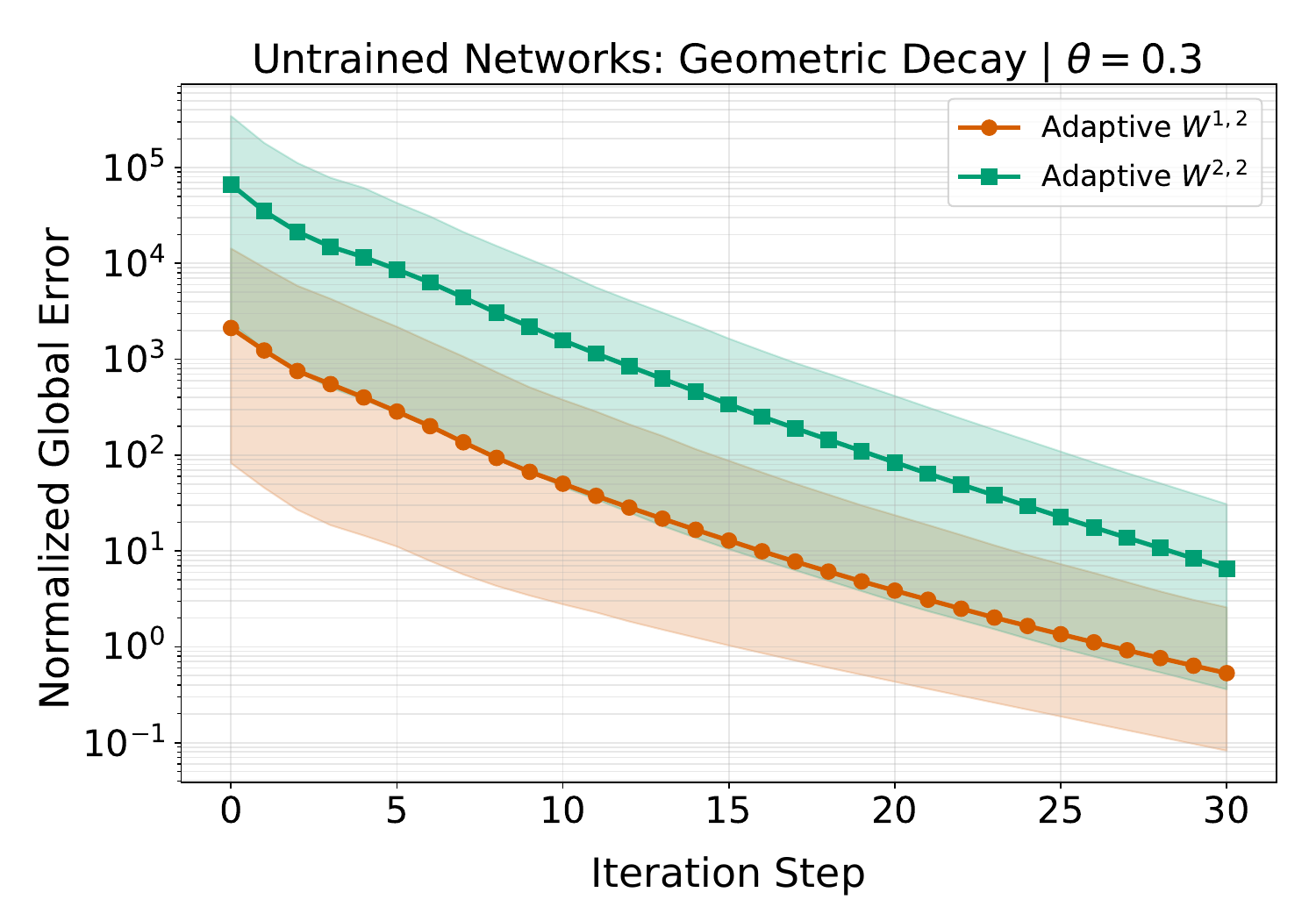}
			\caption{Deep architecture.}
			\label{fig: deep random network sobolev bound}
		\end{subfigure}
		\hfill
		\begin{subfigure}[t]{0.48\linewidth}\vspace{0pt}
			\centering
			\includegraphics[width=\linewidth]{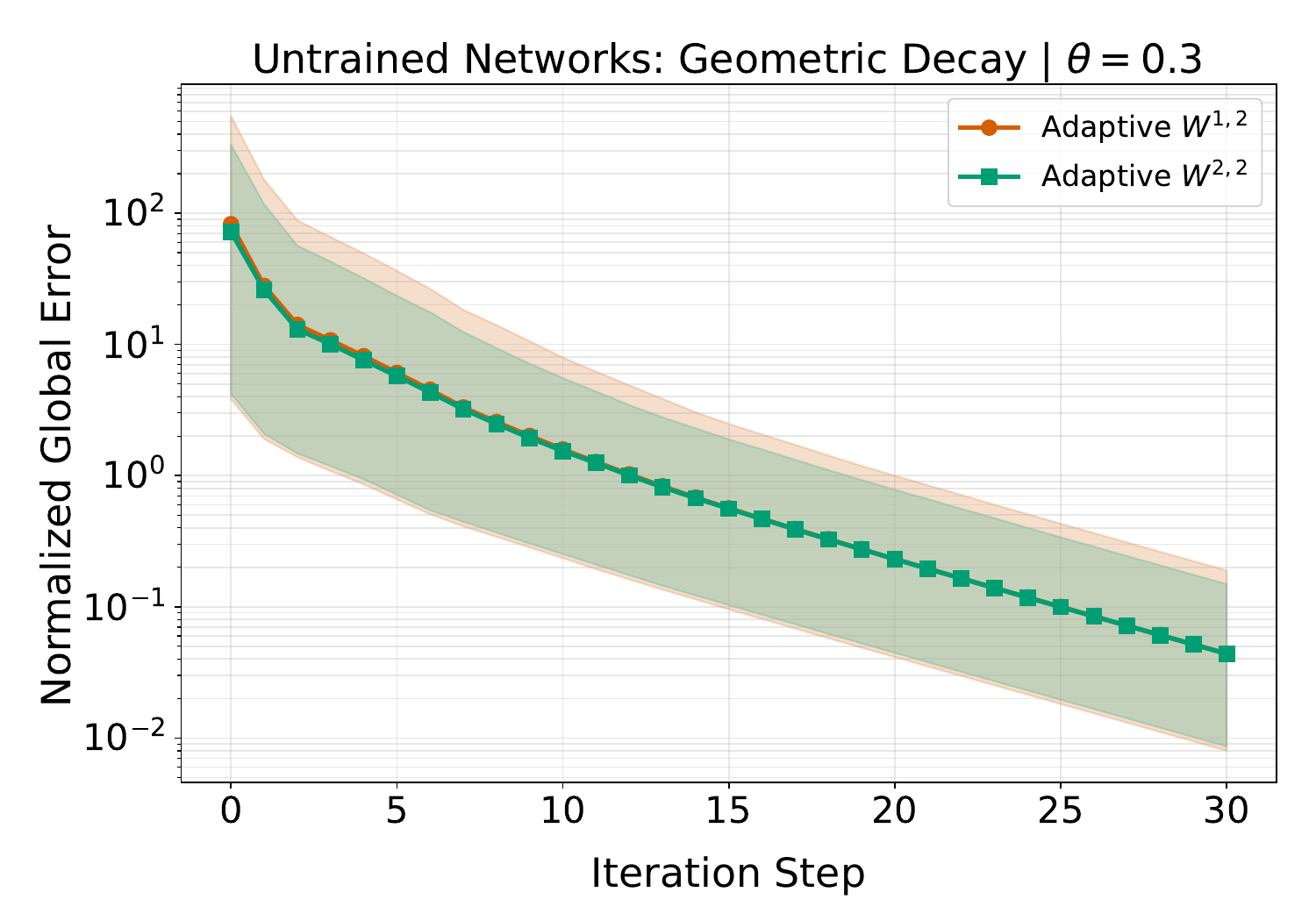}
			\caption{Wide architecture.}
			\label{fig: wide random network sobolev bound}
		\end{subfigure}
		\caption{Mean and $95\%$ confidence interval of the normalized global error for Sobolev norms for 100 untrained tanh networks, refined adaptively, comparing a deep (three hidden layers with 32 neurons) and a wide (one hidden layer with 200 neurons) architecture. }
		\label{fig: random network sobolev bound}
	\end{figure}
	
	\begin{figure}[htb]
		\centering
		\begin{subfigure}[t]{0.48\linewidth}\vspace{0pt}
			\centering
			\includegraphics[width=\linewidth]{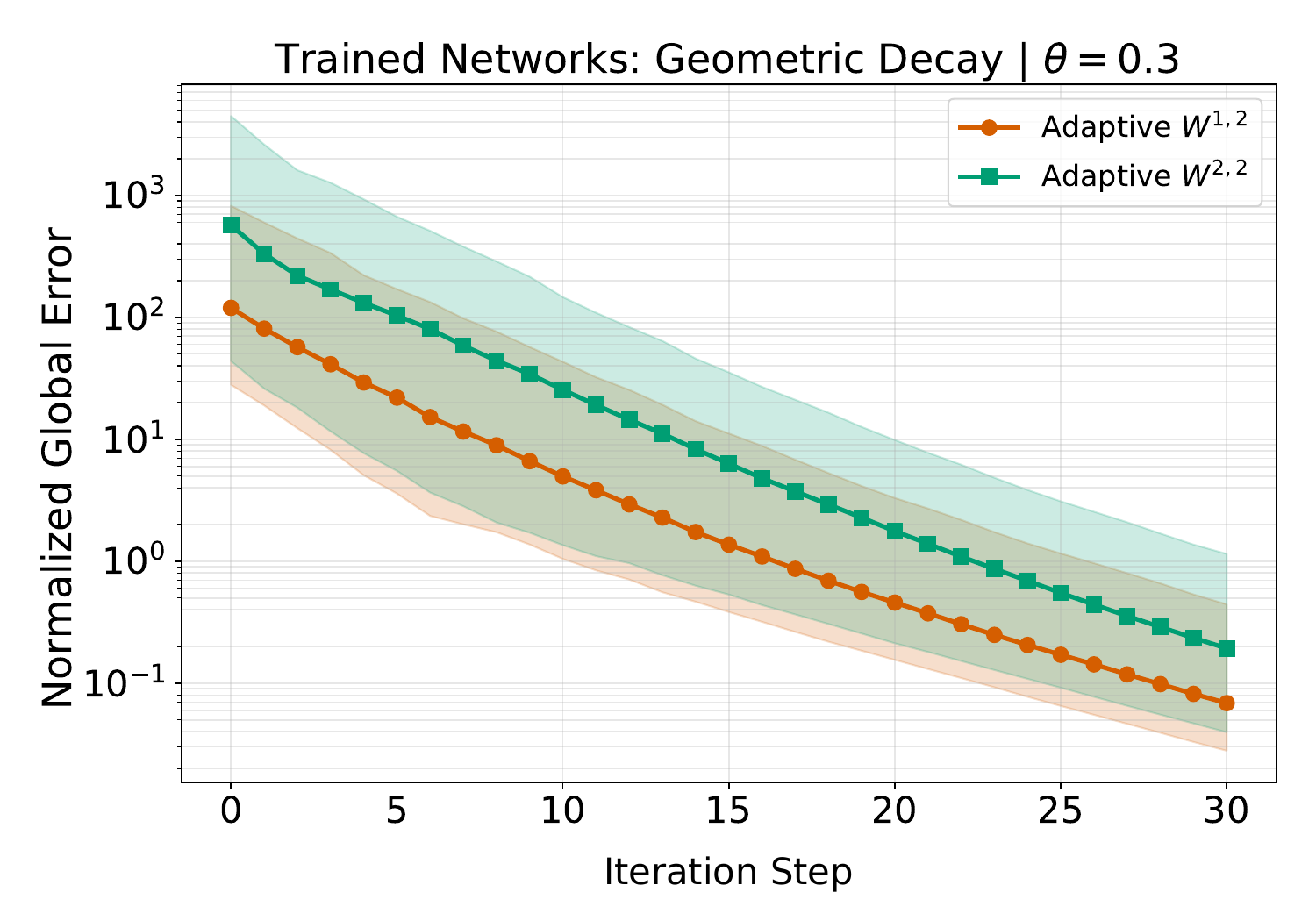}
			\caption{Deep architecture.}
			\label{fig: deep trained network sobolev bound}
		\end{subfigure}
		\hfill
		\begin{subfigure}[t]{0.48\linewidth}\vspace{0pt}
			\centering
			\includegraphics[width=\linewidth]{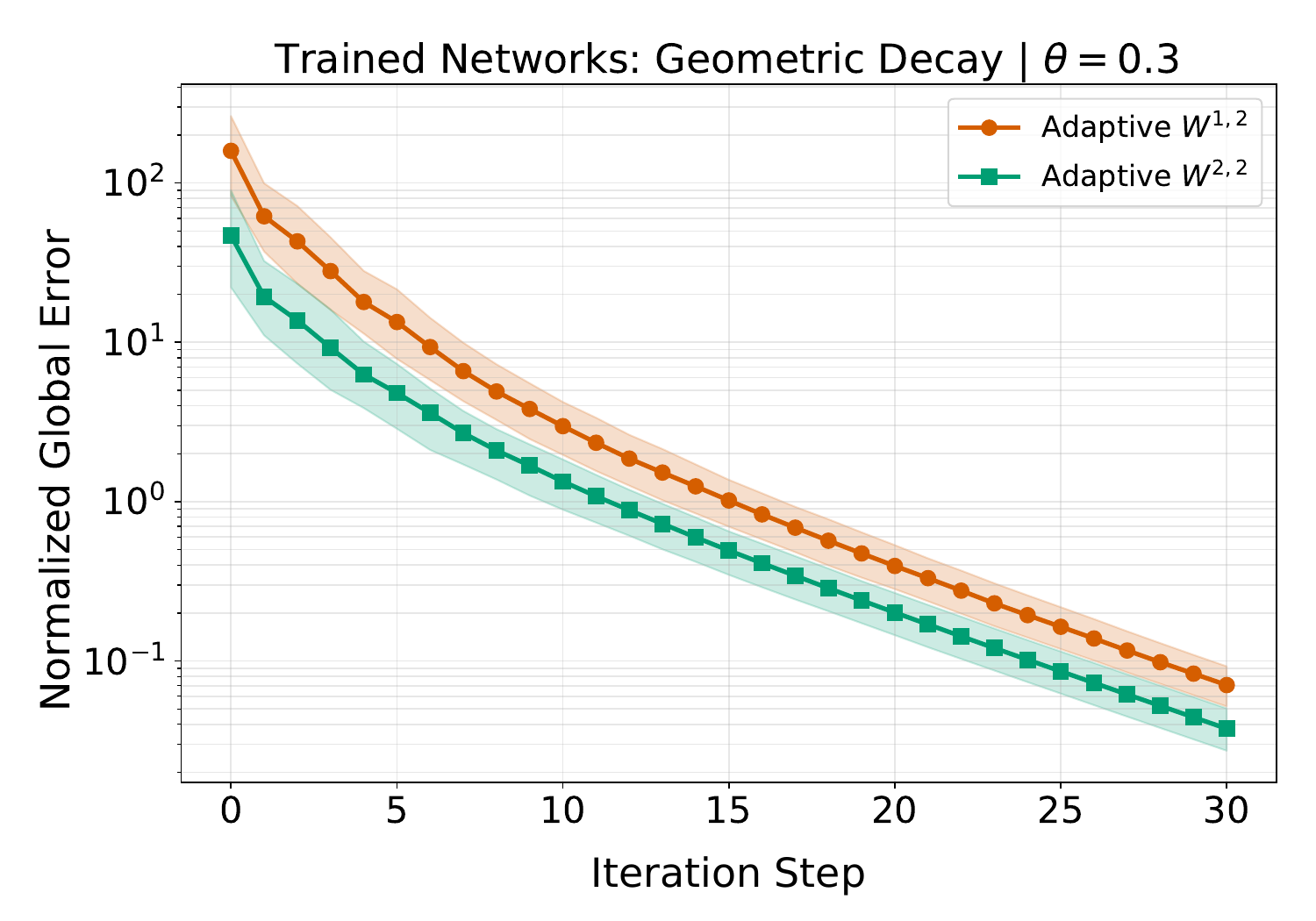}
			\caption{Wide architecture.}
			\label{fig: wide trained network sobolev bound}
		\end{subfigure}
		\caption{Mean and $95\%$ confidence interval of the normalized global error for Sobolev norms of 100 tanh networks, refined with Dörfler marking, comparing a deep (three hidden layers with 32 neurons) and a wide (one hidden layer with 200 neurons) architecture. The neural networks are trained to fit the Gaussian peak~\eqref{eq:gaussian_peak} function.}
		\label{fig: gaussian trained vs untrained}
	\end{figure}
	
	We then train deep and wide architectures using the same initializations as in Figure~\ref{fig: random network sobolev bound}.
	The neural networks are trained to fit the one-dimensional Gaussian-peak function
	\begin{align}
		\label{eq:gaussian_peak}
		f(x) = \exp\left(-\frac{(x - x_0)^2}{\sigma^2}\right),
	\end{align}
	for $x\in[0, 2\pi]$ and $x_0=\pi$. 
	Each network is trained with Adam using learning rate $10^{-3}$ for 2000 epochs. 
	Figures~\ref{fig: deep trained network sobolev bound} and~\ref{fig: wide trained network sobolev bound} show the resulting global errors for Sobolev norms in the Gaussian-peak experiment.
	The error values are again normalized by a Sobolev norm estimated via Monte Carlo with 50000 samples. 
	The normalized mean global error over 100 initializations decays geometrically for both architectures, as predicted by Proposition~\ref{prop: convergence}. 
	The magnitude of the normalized global error is initially higher by a factor of 10 for the deep architecture. Initially, the mean error drops steeply and becomes asymptotically geometric, since the Dörfler marking targets boxes with a high error around the Gaussian peak, reducing the global error significantly. 
	Once the adaptive refinement sufficiently resolves localized features, the spatial variance of the error decreases. This results in a more uniform distribution across the domain, at which point the global error reduction stabilizes at the predicted geometric rate. 
	
	\subsubsection{$L^p$ bounds for untrained and trained networks (1d)}
	\label{sec: $L^p$ bounds for untrained and trained networks (1d)}
	To complement the previous experiment for Sobolev norms, we conduct an analogous 100-run 
	statistical analysis for the $L^p$ norm bounds. 
	We maintain the same deep 
	and wide architectural dimensions (three hidden layers of 32 neurons, and one 
	hidden layer of 200 neurons, respectively), but employ the ReLU activation function to enable the $L^p$ norm bounding optimizations with Hölder refinement (Algorithm~\ref{alg: hölder strategy}). 
	We again randomly initialize 100 ReLU neural networks as before and fit them to the 1d Gaussian peak problem~\eqref{eq:gaussian_peak}, utilizing the Adam 
	optimizer over 2000 epochs without weight decay.
	Figure~\ref{fig: lp statistical 
		waterfall combined} displays the normalized mean global error and the 95\% confidence intervals of the 
	geometric decay for both the untrained and trained errors. 
	The normalization is carried out as in the previous experiments. 
	As observed in the untrained scenario (Figure~\ref{fig: lp untrained waterfall}), both architectures exhibit a steady reduction in the normalized global error, with the deep network initially displaying a higher gap. 
	Conversely, the trained networks (Figure~\ref{fig: lp trained waterfall}) demonstrate a highly consistent and stable geometric decay. 
	The deep architecture initially presents a higher normalized global error with a slightly wider deviation. 
	The difference between the deep and wide architecture is significantly reduced compared to the Sobolev experiment in Figure~\ref{fig: gaussian trained vs untrained}.

	\begin{figure}[htb]
		\centering
		\begin{subfigure}[t]{0.48\linewidth}\vspace{0pt}
			\centering
			\includegraphics[width=\linewidth]{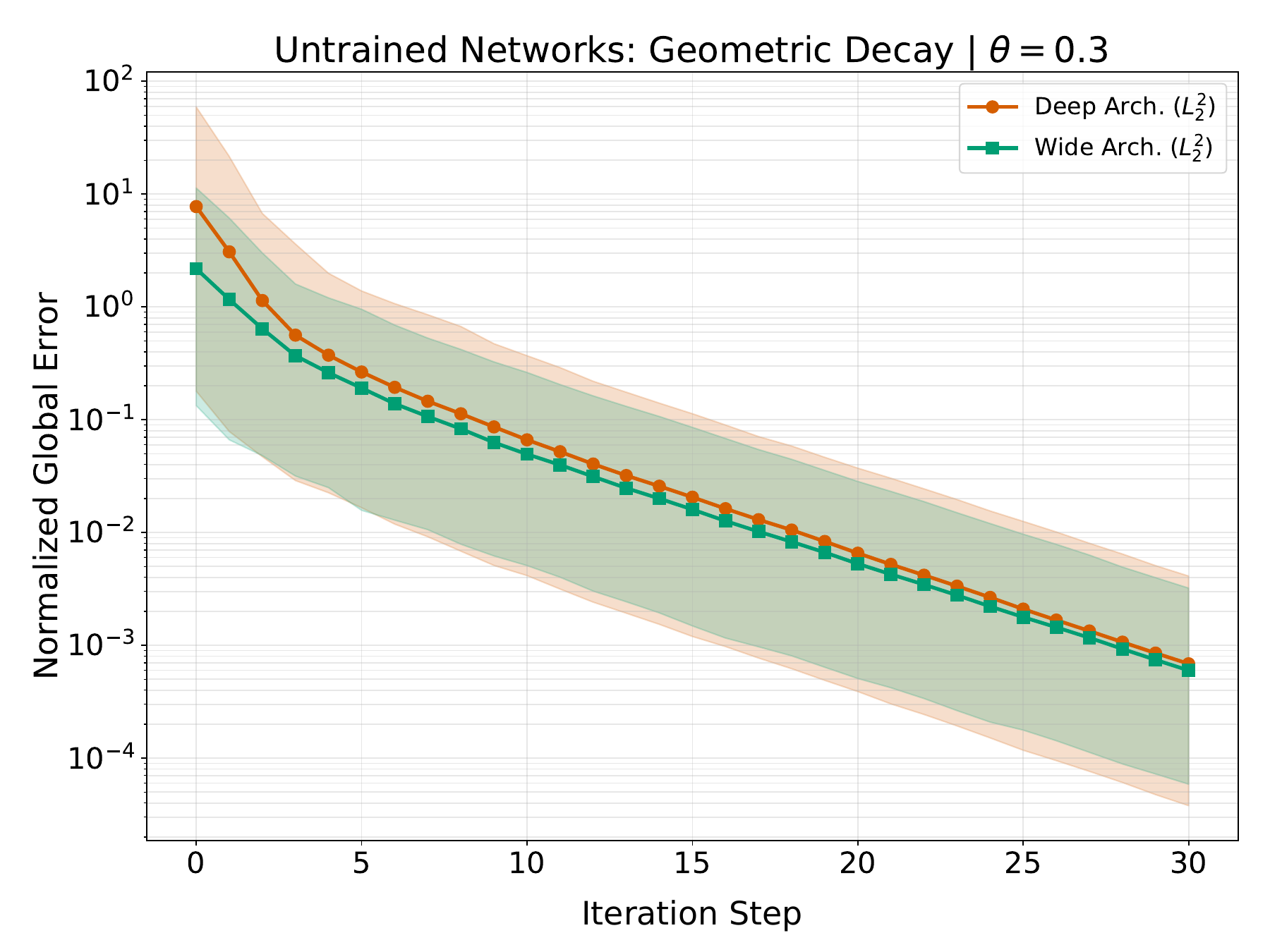}
			\caption{Untrained ReLU networks.}
			\label{fig: lp untrained waterfall}
		\end{subfigure}
		\hfill
		\begin{subfigure}[t]{0.48\linewidth}\vspace{0pt}
			\centering
			\includegraphics[width=\linewidth]{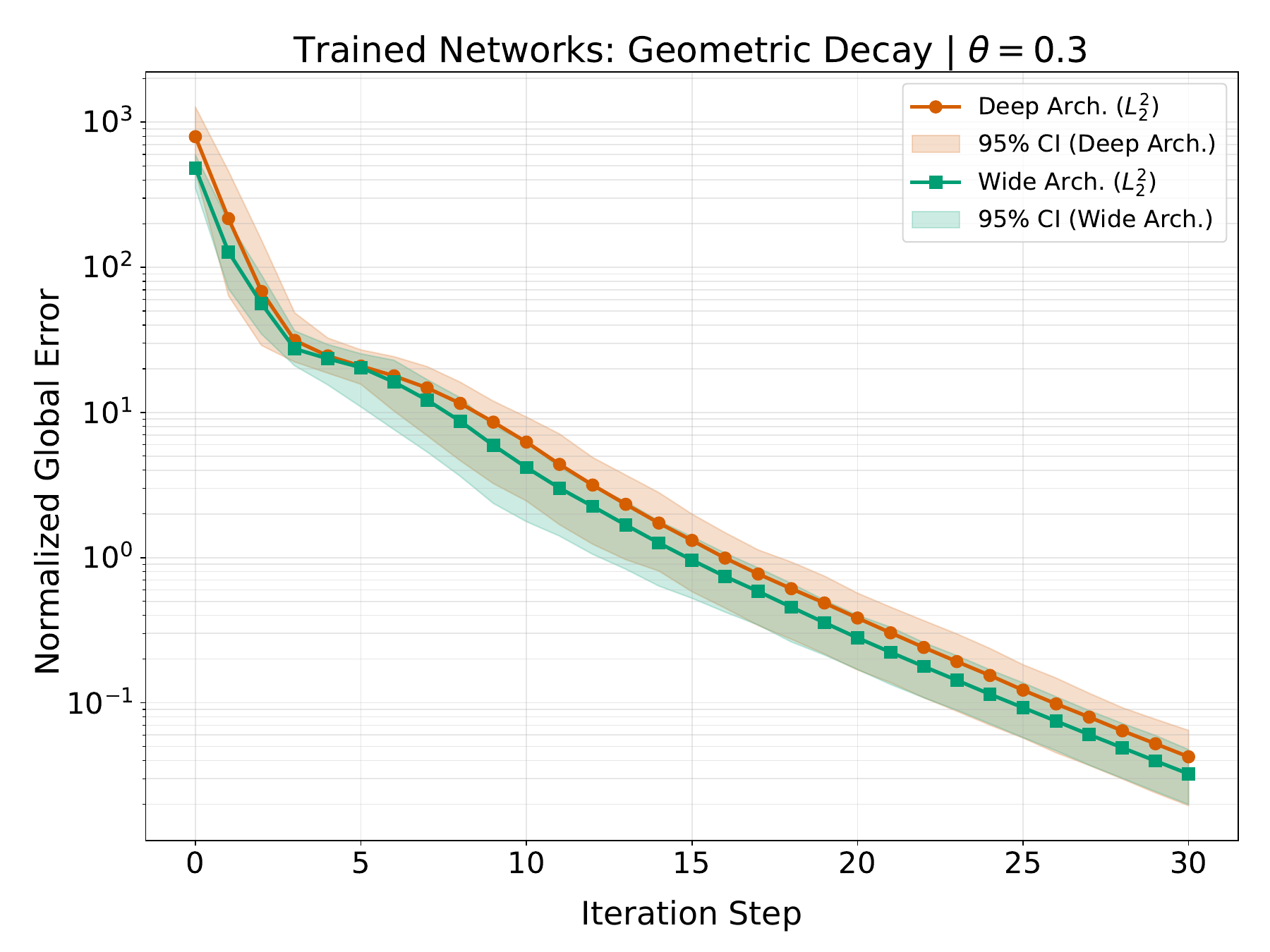}
			\caption{Trained ReLU networks.}
			\label{fig: lp trained waterfall}
		\end{subfigure}
		\caption{Mean and $95\%$ confidence interval of the normalized global error for the $L^p$ norm for 100 ReLU networks, comparing a deep and wide architecture. 
			The trained neural networks approximate the one-dimensional Gaussian-peak function~\eqref{eq:gaussian_peak}.}
		\label{fig: lp statistical waterfall combined}
	\end{figure}

	\subsubsection{Sobolev and $L^p$ bounds for trained networks (2d)}
	\label{sec: Sobolev and $L^p$ bounds for trained networks (2d)}
	The one-dimensional experiment indicates that architectural differences are most visible when the target has localized high curvature. 
	The smooth disk function, defined below, exhibits a localized curvature, making it an ideal fit for testing the adaptive integration of Algorithm~\ref{alg: adaptive integration} in two dimensions.
	The smooth disk function is defined as
	\begin{align}
		\label{eq: smooth disk}
		f(\mathbf{x}) = \begin{cases} 
			1 & \text{if } r \le 1, \\
			0 & \text{if } r \ge 1 + \epsilon, \\
			\frac{h(1+\epsilon - r)}{h(1+\epsilon - r) + h(r - 1)} & \text{otherwise},
		\end{cases}
	\end{align}
	where $r = \|\mathbf{x}\|_2$, $\epsilon = 0.5$ and $h(t) = e^{-1/t}$ for $t>0$ and zero otherwise.
	We train two architectures with Adam: a deep network with three hidden layers of width 32 and a wide network with one hidden layer of width 200.
	For $L^p$-bound evaluation, both architectures use ReLU and are trained for 2000 epochs without weight decay.
	For the Sobolev bounds, the neural networks were trained for 10000 epochs using the tanh activation function, applying a weight decay of $10^{-3}$ exclusively to the deep architecture. 
	Regularization by weight decay was used to avoid spiky approximations by the deep neural network, which would result in a very high Sobolev norm.  
	
	To bridge our earlier theory to these practical architectures, we first evaluate the normalized global error and the local error for the $L^2$ quantity on this domain.
	Figure~\ref{fig:2d_coin_lp_experiment} presents the adaptive refinement results for both the deep and wide architectures using $\theta=0.5$. 
	The convergence plots show that the normalized global error decays at the predicted geometric rate for both models. 
	The corresponding spatial heatmaps confirm that the marking strategy isolates and heavily refines transition regions with high local errors without wasting budget on the flat interior or exterior. 
	Notably, the deep neural network exhibits a sharper localization of the error compared to the wide neural network.
	\begin{figure}[htbp]
		\centering
		
		\begin{subfigure}[t]{0.48\linewidth}
			\centering
			\includegraphics[width=\linewidth]{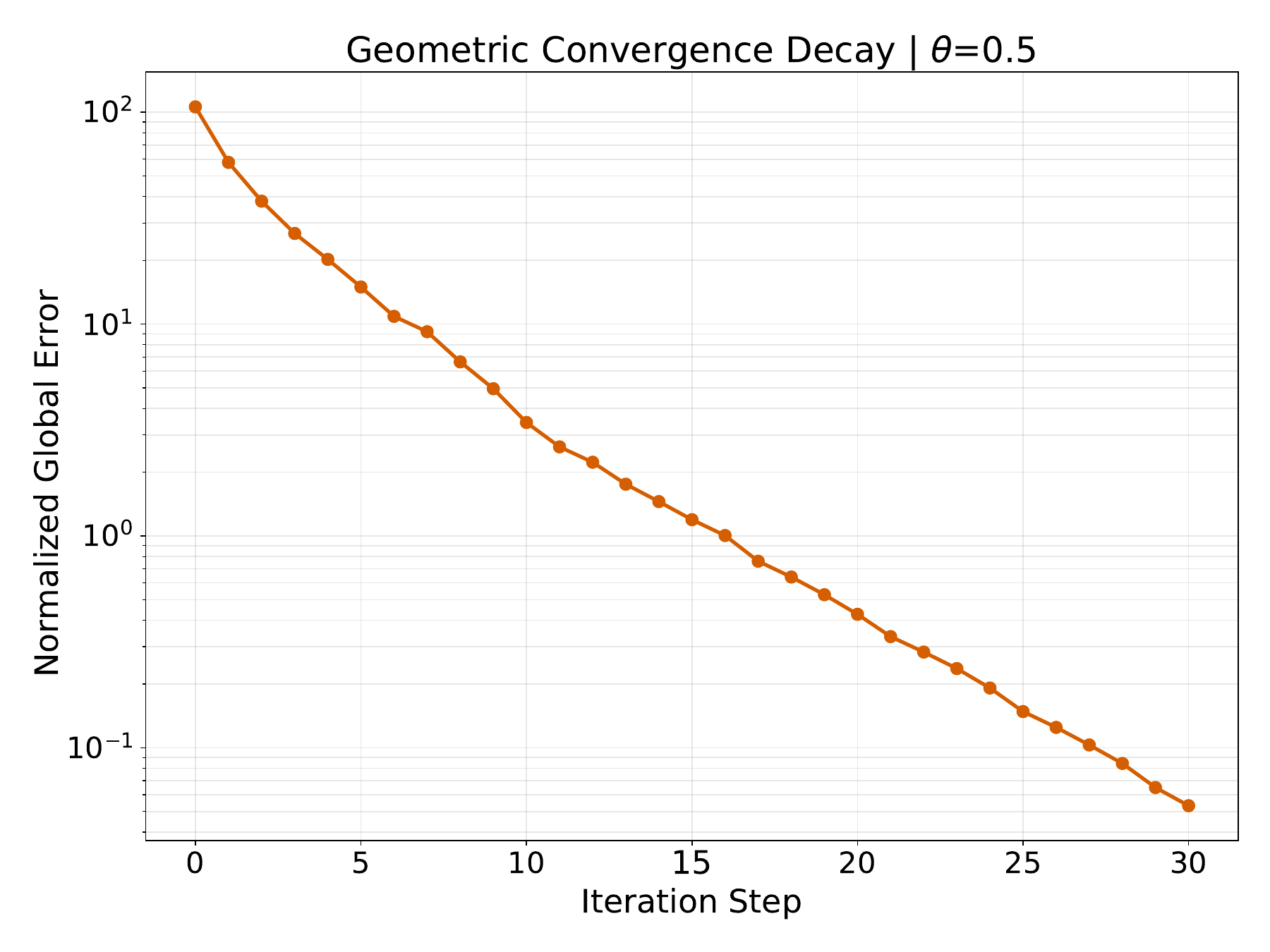}
			\caption{\small Geometric convergence ($L^2$ norm) for the deep architecture.}
			\label{fig:lp_rate_convergence_deep}
		\end{subfigure}
		\hfill
		\begin{subfigure}[t]{0.48\linewidth}
			\centering
			\includegraphics[width=\linewidth]{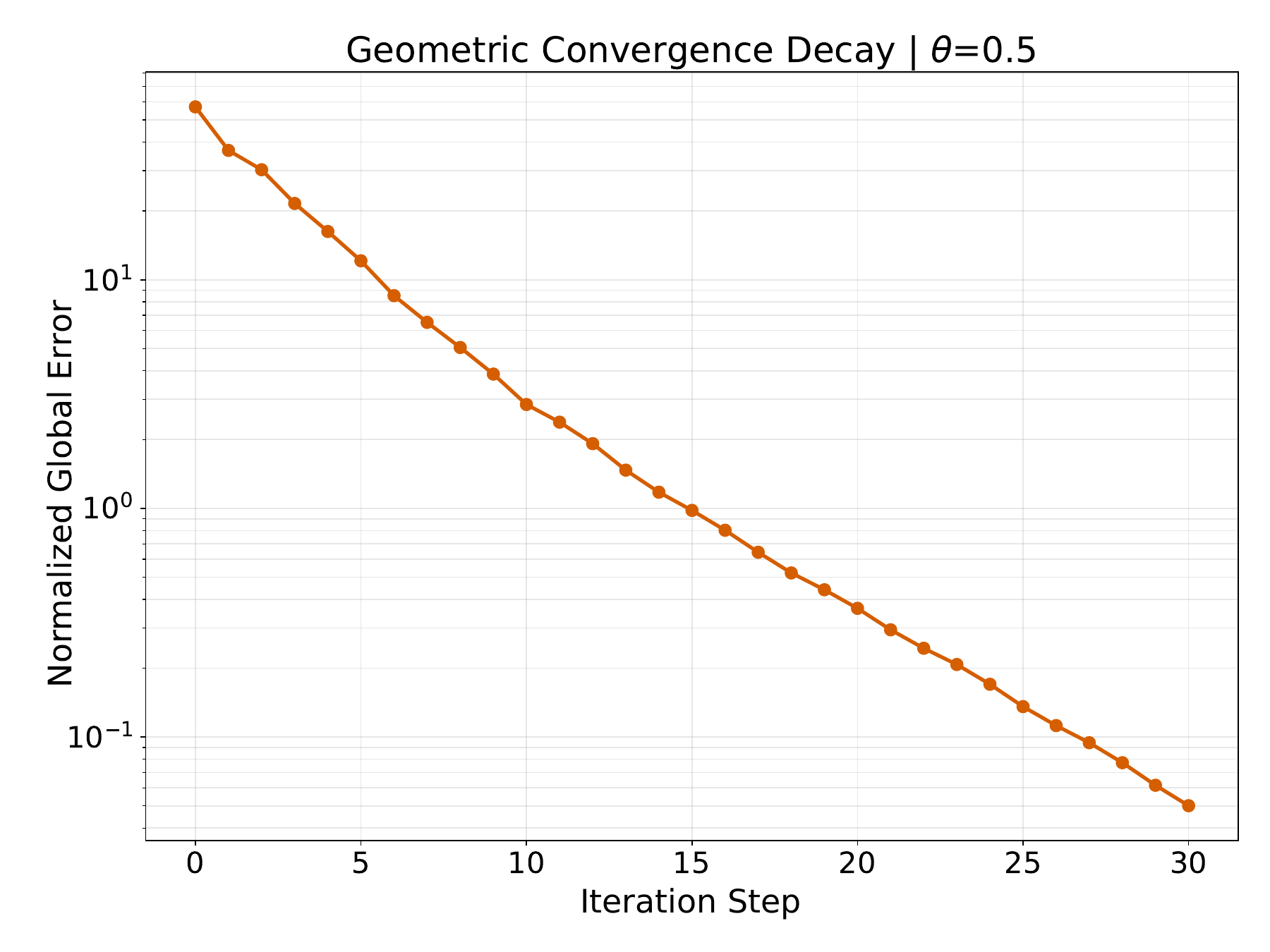}
			\caption{\small Geometric convergence ($L^2$ norm) for the wide architecture.}
			\label{fig:lp_rate_convergence_wide}
		\end{subfigure}
		
		\vspace{1em}
		
		\begin{subfigure}[t]{0.48\linewidth}
			\centering
			\includegraphics[width=\linewidth]{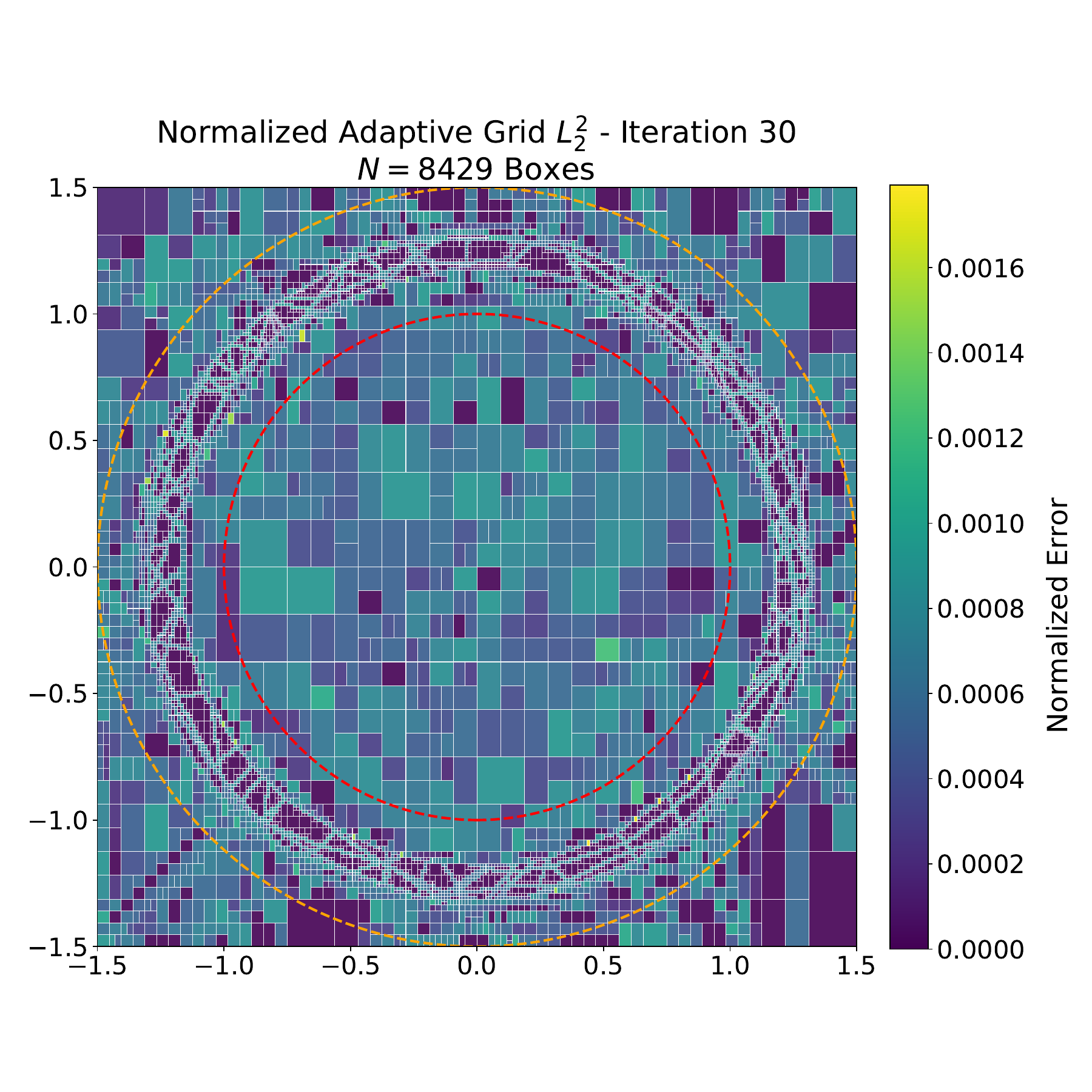}
			\caption{\small Heatmap of the local error for the $L^2$ norm (deep architecture).}
			\label{fig:lp_heatmap_deep}
		\end{subfigure}
		\hfill
		\begin{subfigure}[t]{0.48\linewidth}
			\centering
			\includegraphics[width=\linewidth]{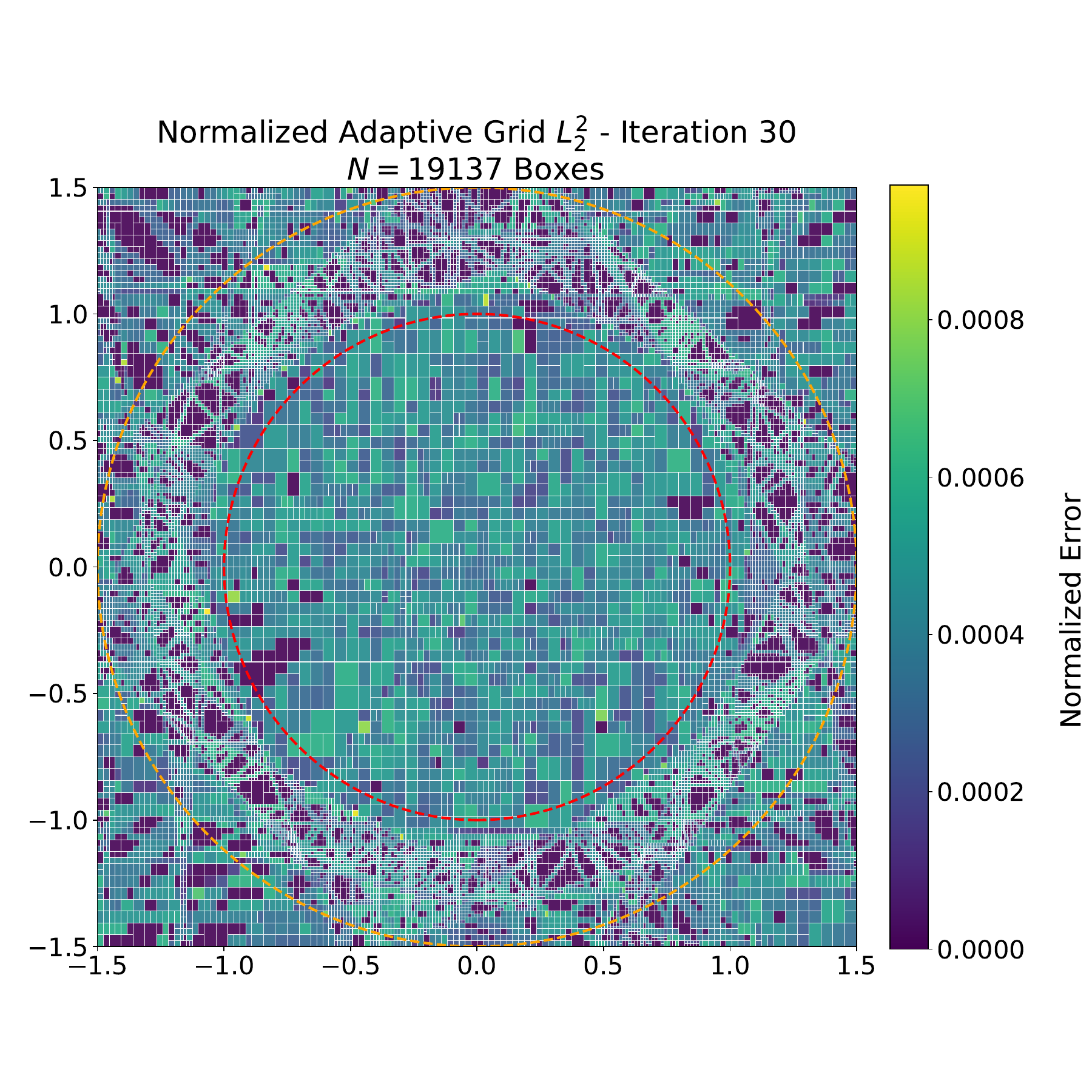}
			\caption{\small Heatmap of the local error for the $L^2$ norm (wide architecture).}
			\label{fig:lp_heatmap_wide}
		\end{subfigure}
		
		\caption{Two-dimensional disk experiment: adaptive evaluation of the normalized global error and the local errors for the $L^2$ quantity in deep and wide ReLU architectures ($\theta=0.5$).}
		\label{fig:2d_coin_lp_experiment}
	\end{figure}
	
	Building on this, Figure~\ref{fig:2d_coin_experiment_analysis_rates} presents convergence and refinement results for the normalized global error and the local error for Sobolev norms.
	The convergence closely follows the geometric convergence of Proposition~\ref{prop: convergence} shown in Figures~\ref{fig: deep rate convergence} and~\ref{fig: wide rate convergence}. 
	Figures~\ref{fig:w1p_deep_heatmap}, \ref{fig:w1p_wide_heatmap}, \ref{fig:w2p_deep_heatmap}, and \ref{fig:w2p_wide_heatmap} highlight how the adaptive framework targets localized large errors. 
	It is visible how the adaptive framework targets the transition areas with rapidly changing derivative (and therefore high bounds), which is very prominent in the heatmaps for the deep architecture. 
	The wide architecture cannot produce the same level of complexity, as indicated by the quasi-equidistant grid of the heatmap. 
	Figures~\ref{fig: deep local reduction rate convergence w1p}, \ref{fig: deep local reduction rate convergence w2p}, \ref{fig: wide local reduction rate convergence w1p}, and \ref{fig: wide local reduction rate convergence w2p} show that uniform refinement always results in a reduction of the local errors for all architectures. 
	This is essential for a provable convergence of the Dörfler marking and is discussed in Remark~\ref{rem:HalfRefinement}. Interval extensions reduce their error in $O(\text{diam}(\Omega))$, therefore one would expect a reduction of $1/4$ for the uniform refinement rule and the bounded norm to the second power.
	
	\begin{figure}[htbp]
		\centering
		
		\begin{subfigure}[t]{0.48\linewidth}
			\centering
			\includegraphics[width=\linewidth]{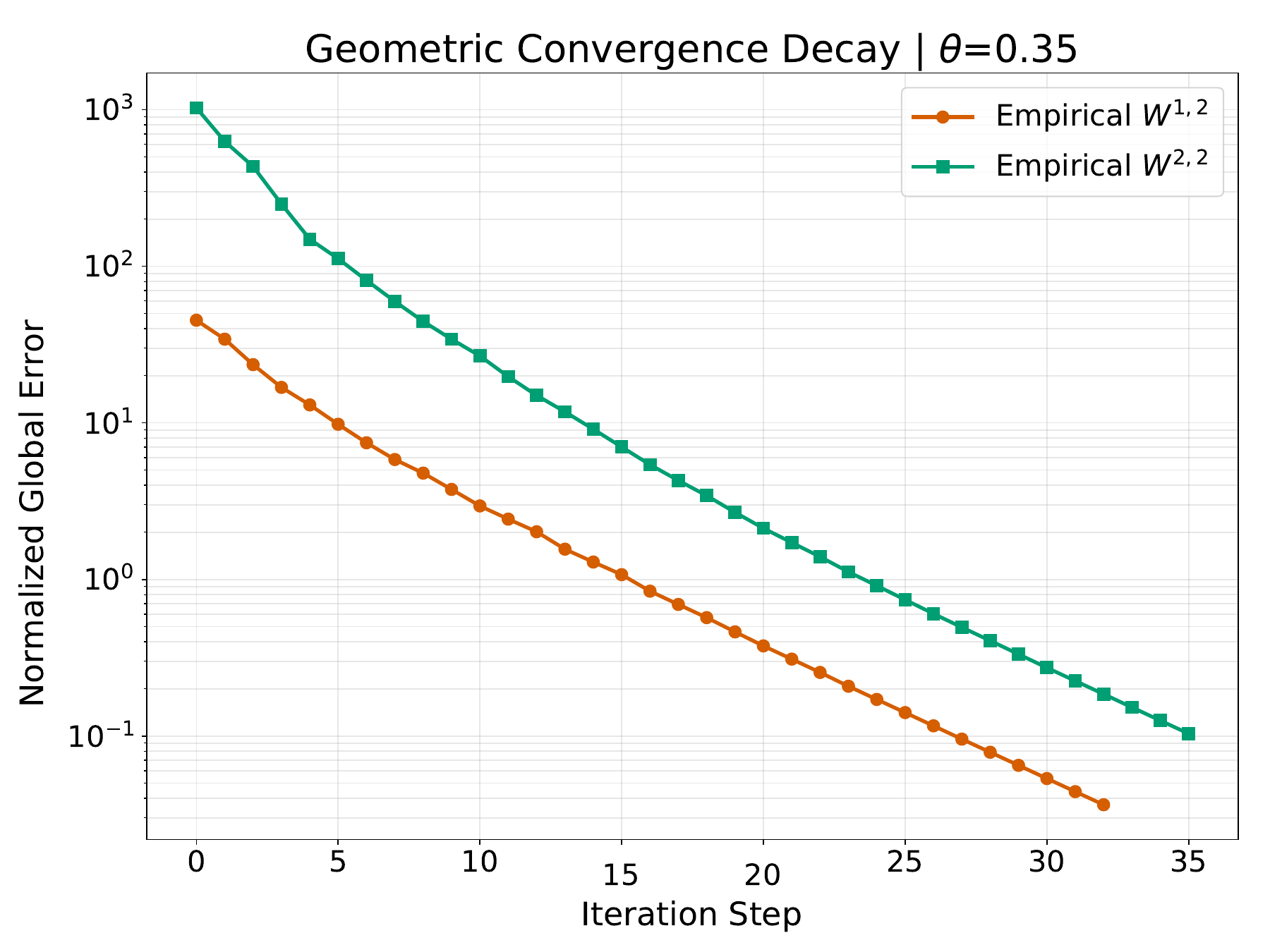}
			\caption{\small Convergence of the normalized global error for a deep architecture.}
			\label{fig: deep rate convergence}
		\end{subfigure}
		\hfill
		\begin{subfigure}[t]{0.48\linewidth}
			\centering
			\includegraphics[width=\linewidth]{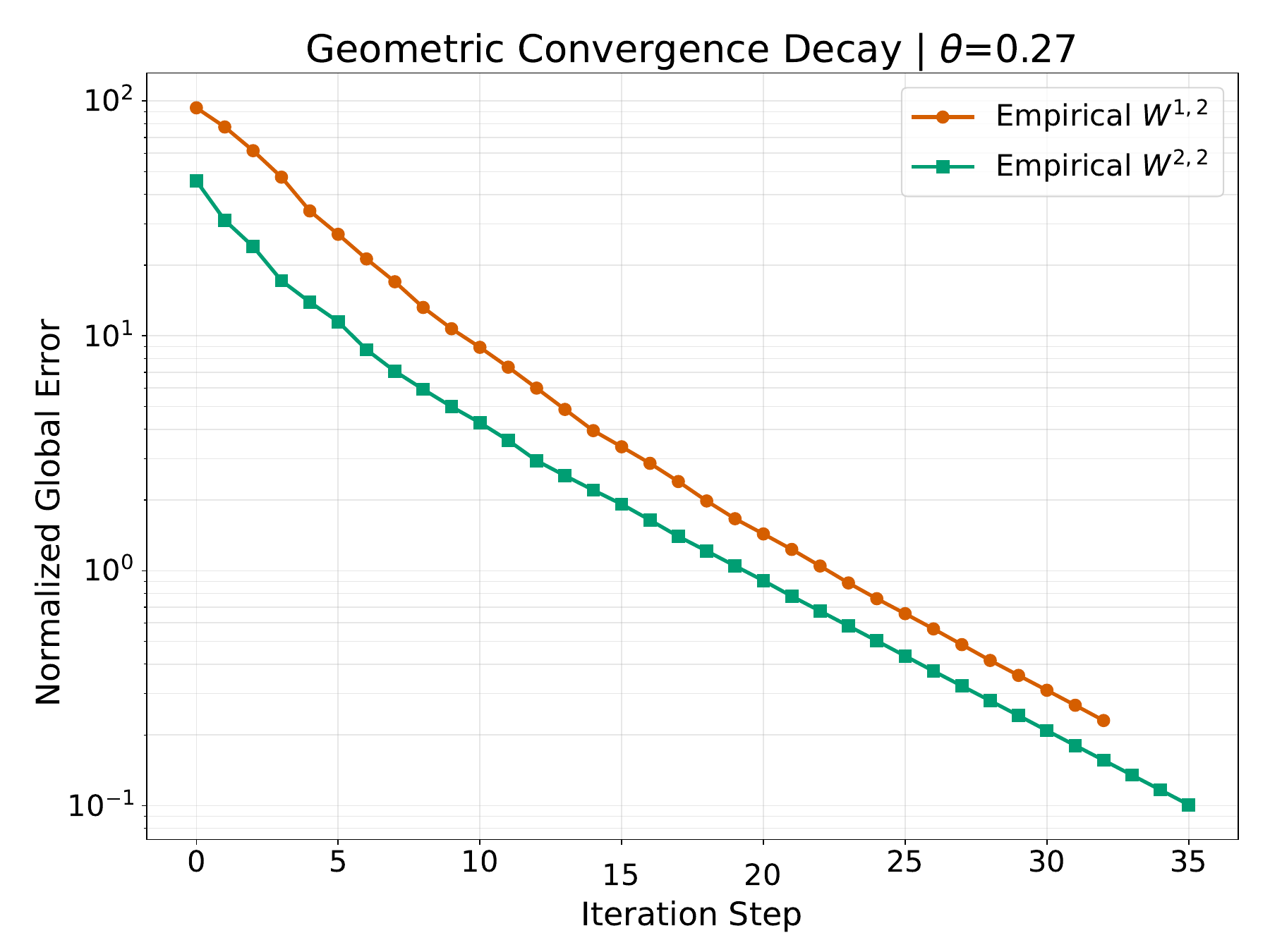}
			\caption{\small Convergence of the normalized global error for a wide architecture.}
			\label{fig: wide rate convergence}
		\end{subfigure}
		
		\vspace{1em}
		
		\begin{subfigure}[t]{0.48\linewidth}
			\centering
			\includegraphics[width=\linewidth]{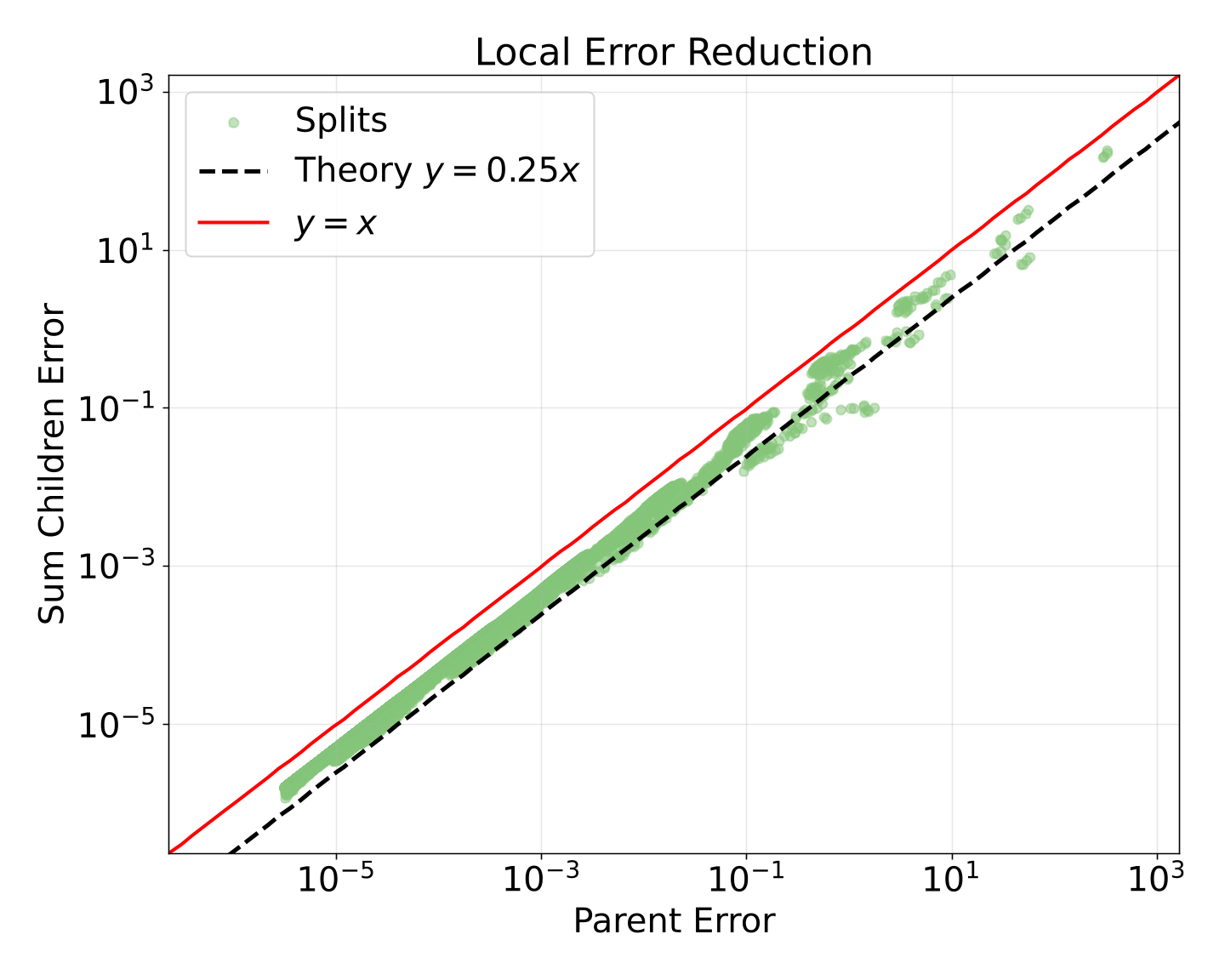}
			\caption{\small Local reduction of the $W^{1,p}$ norm bound for a deep architecture.}
			\label{fig: deep local reduction rate convergence w1p}
		\end{subfigure}
		\hfill
		\begin{subfigure}[t]{0.48\linewidth}
			\centering
			\includegraphics[width=\linewidth]{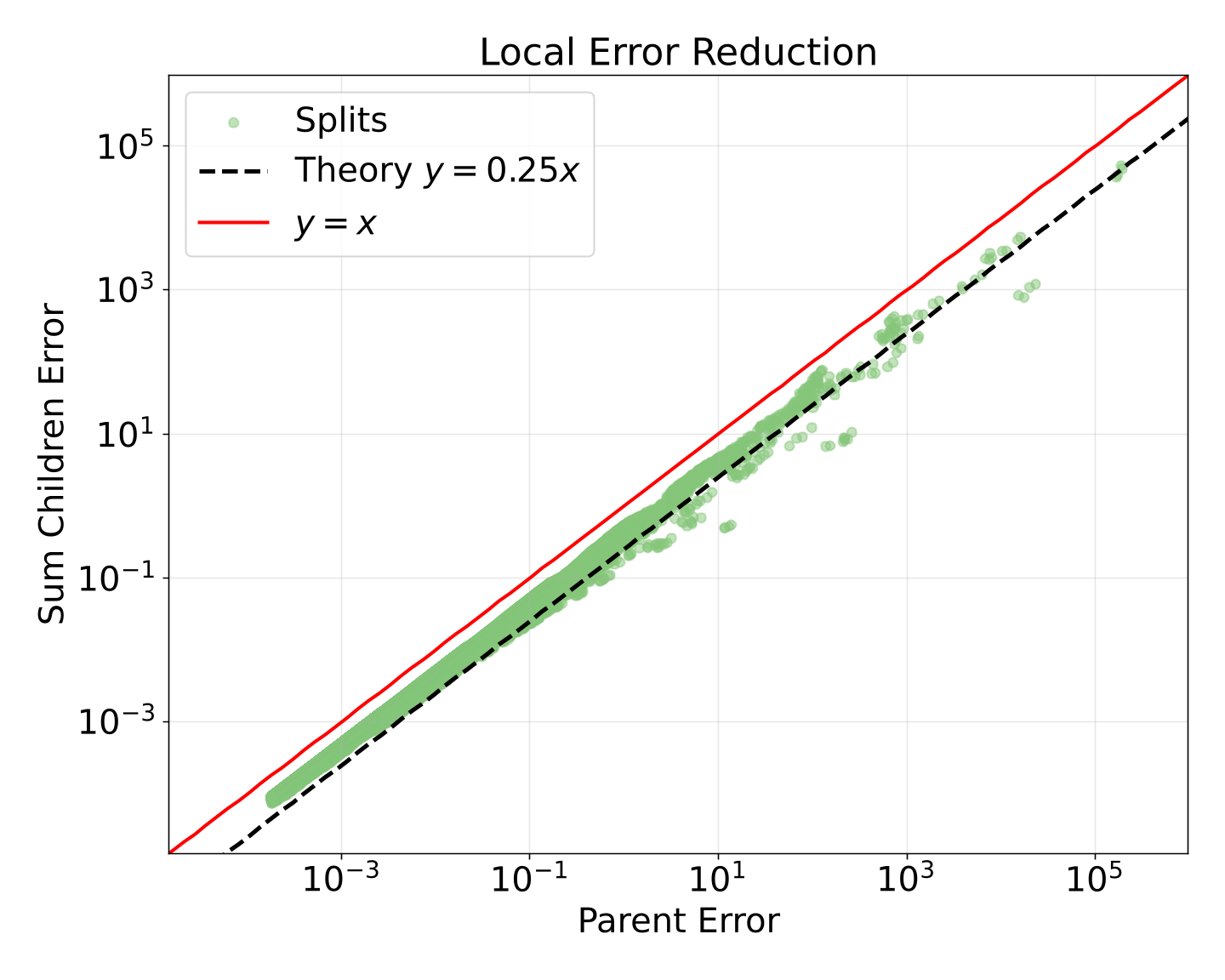}
			\caption{\small Local reduction of the $W^{2,p}$ norm bound for a deep architecture.}
			\label{fig: deep local reduction rate convergence w2p}
		\end{subfigure}
		
		\vspace{1em}
		
		\begin{subfigure}[t]{0.48\linewidth}
			\centering
			\includegraphics[width=\linewidth]{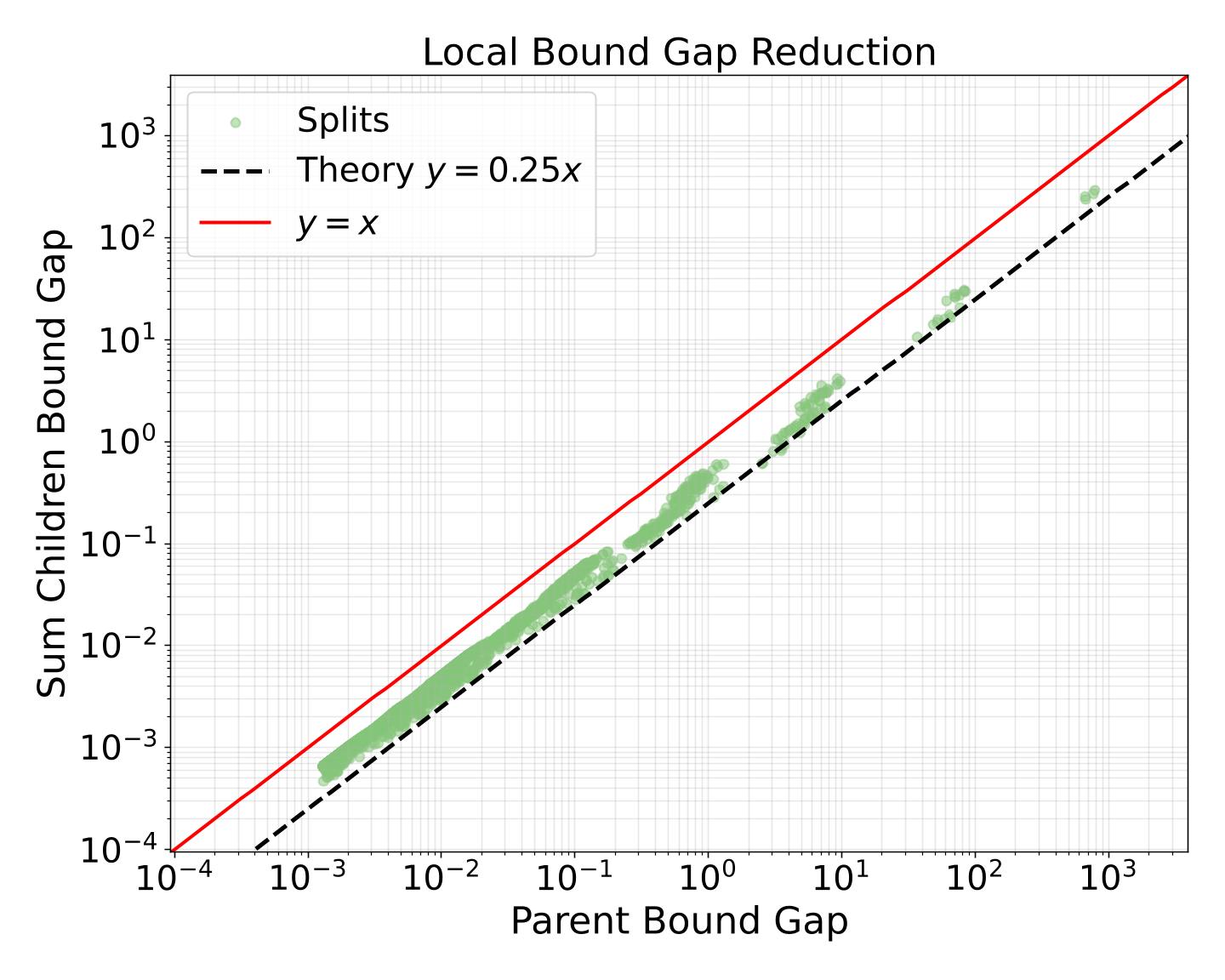}
			\caption{\small Local reduction of the $W^{1,p}$ norm bound for a wide architecture.}
			\label{fig: wide local reduction rate convergence w1p}
		\end{subfigure}
		\hfill
		\begin{subfigure}[t]{0.48\linewidth}
			\centering
			\includegraphics[width=\linewidth]{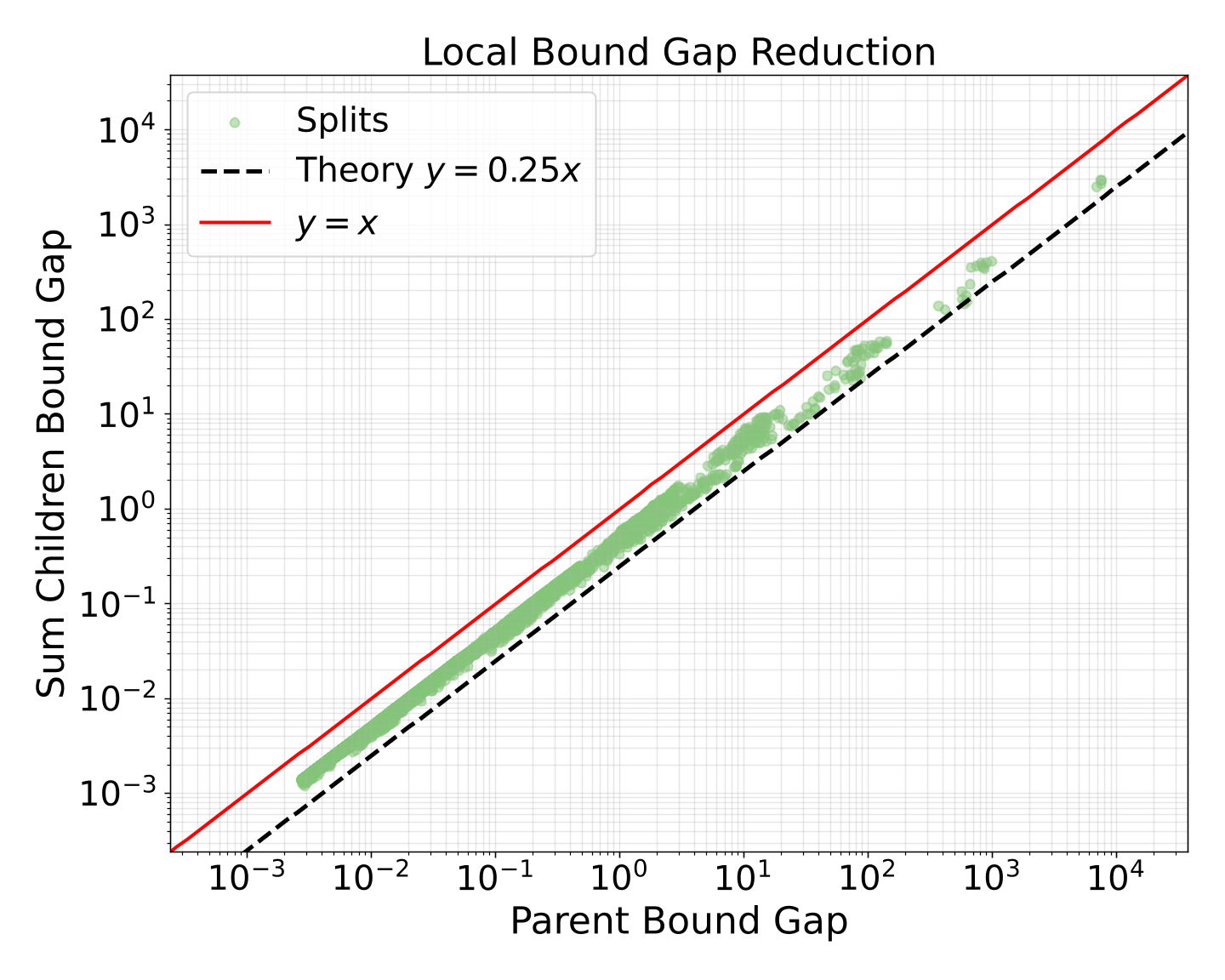}
			\caption{\small Local reduction of the $W^{2,p}$ norm bound for a wide architecture.}
			\label{fig: wide local reduction rate convergence w2p}
		\end{subfigure}
		
		\caption{Two-dimensional disk experiment: Convergence rate of the normalized error for the Sobolev norm and test of the local error reduction for the Dörfler marking (see Remark~\ref{rem:HalfRefinement}).}
		\label{fig:2d_coin_experiment_analysis_rates}
	\end{figure}
	
	\begin{figure}[htbp]
		\centering
		
		\begin{subfigure}[t]{0.48\linewidth}
			\centering
			\includegraphics[width=\linewidth]{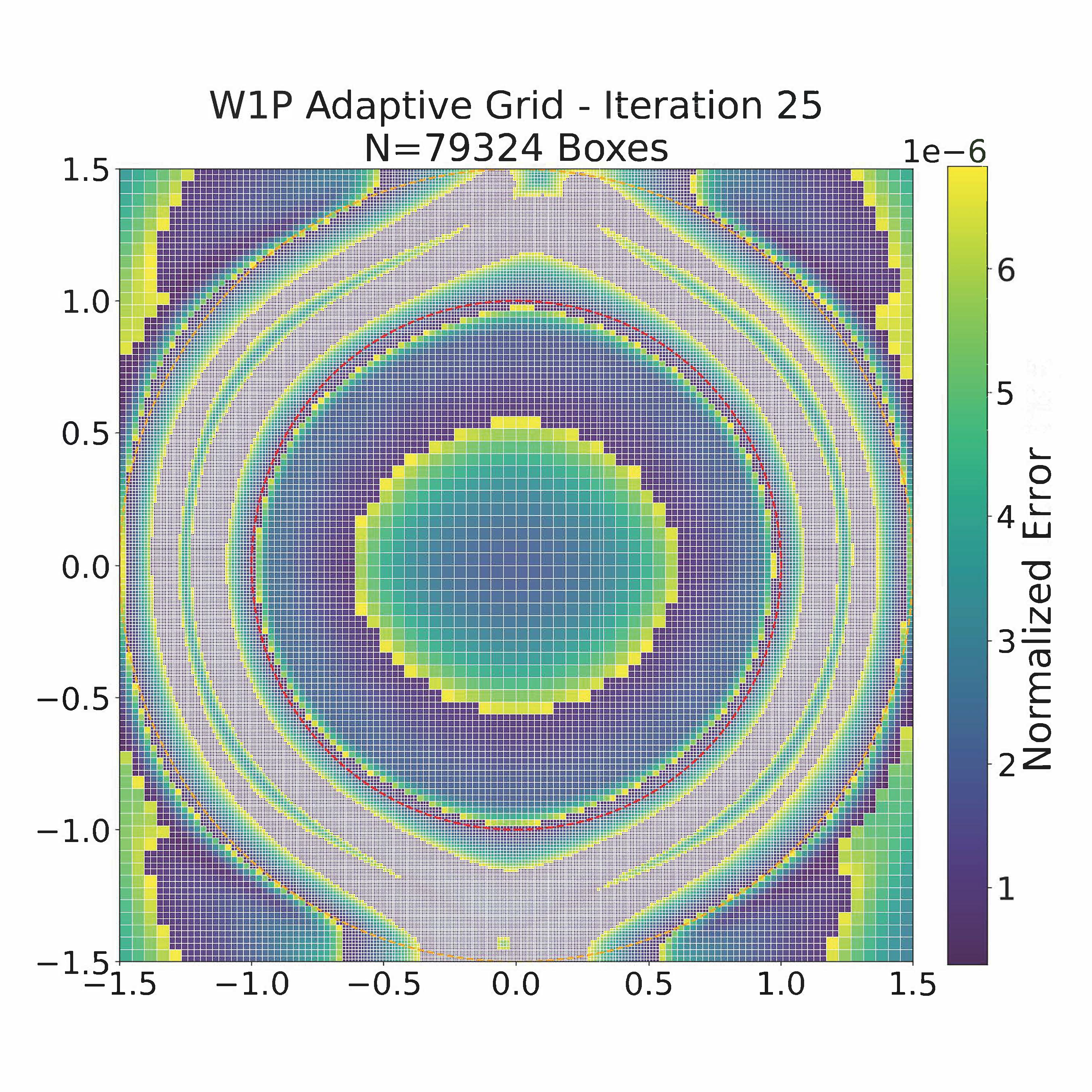}
			\caption{\small Heatmap for the deep architecture (W1P).} 
			\label{fig:w1p_deep_heatmap}
		\end{subfigure}
		\hfill
		\begin{subfigure}[t]{0.48\linewidth}
			\centering
			\includegraphics[width=\linewidth]{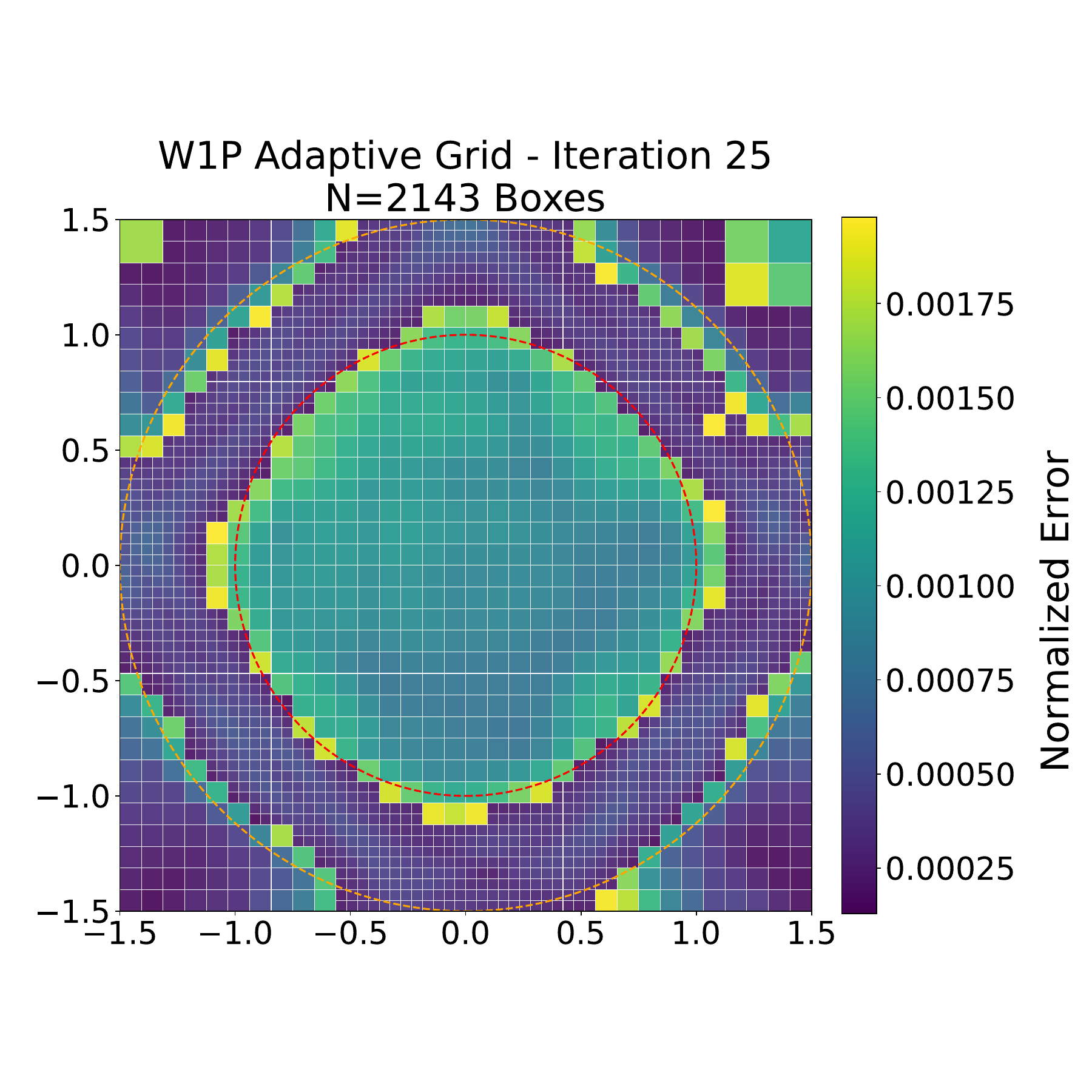}
			\caption{\small Heatmap for the wide architecture (W1P).} 
			\label{fig:w1p_wide_heatmap}
		\end{subfigure}
		
		\vspace{1em} 
		
		\begin{subfigure}[t]{0.48\linewidth}
			\centering
			\includegraphics[width=\linewidth]{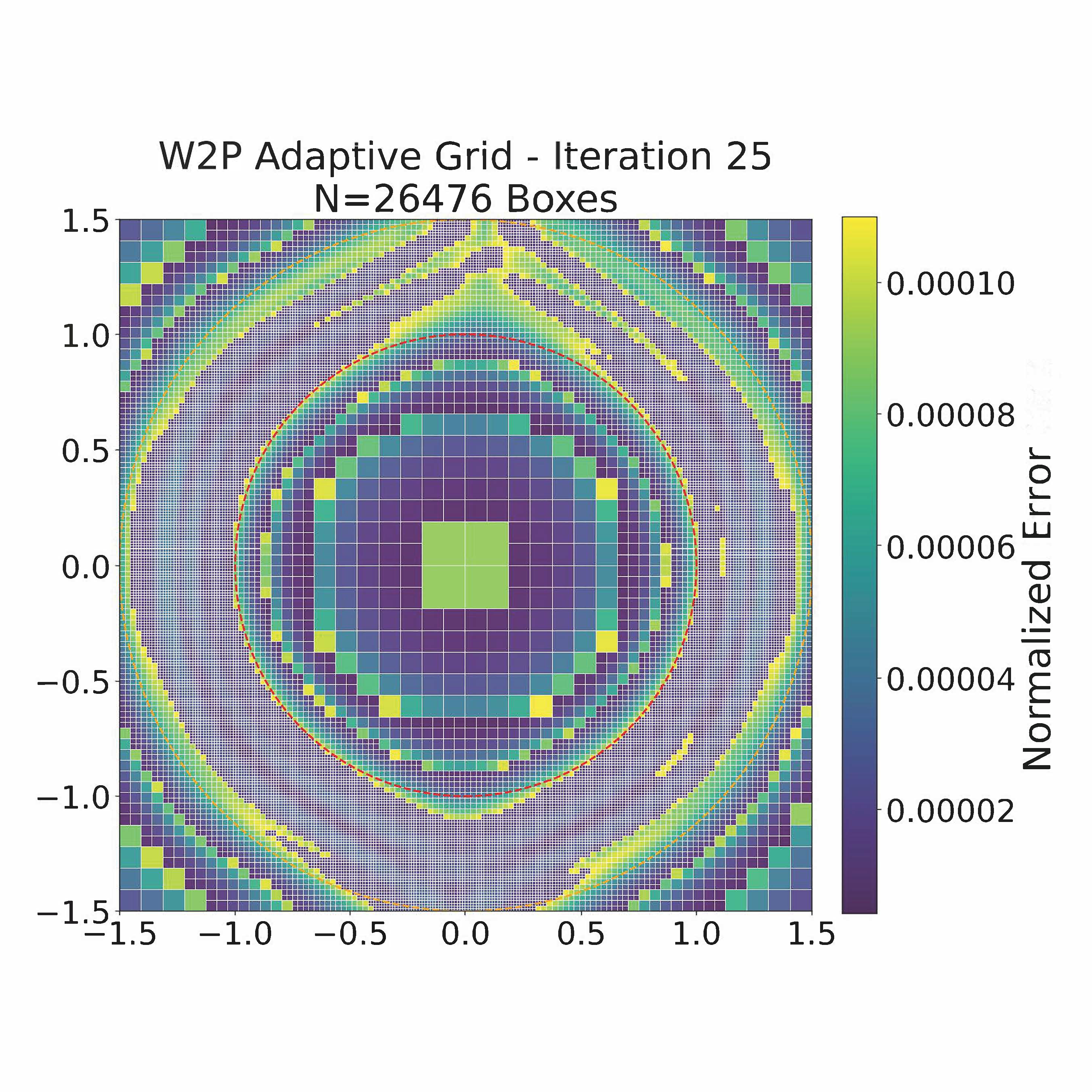}
			\caption{\small Heatmap for the deep architecture (W2P).} 
			\label{fig:w2p_deep_heatmap}
		\end{subfigure}
		\hfill
		\begin{subfigure}[t]{0.48\linewidth}
			\centering
			\includegraphics[width=\linewidth]{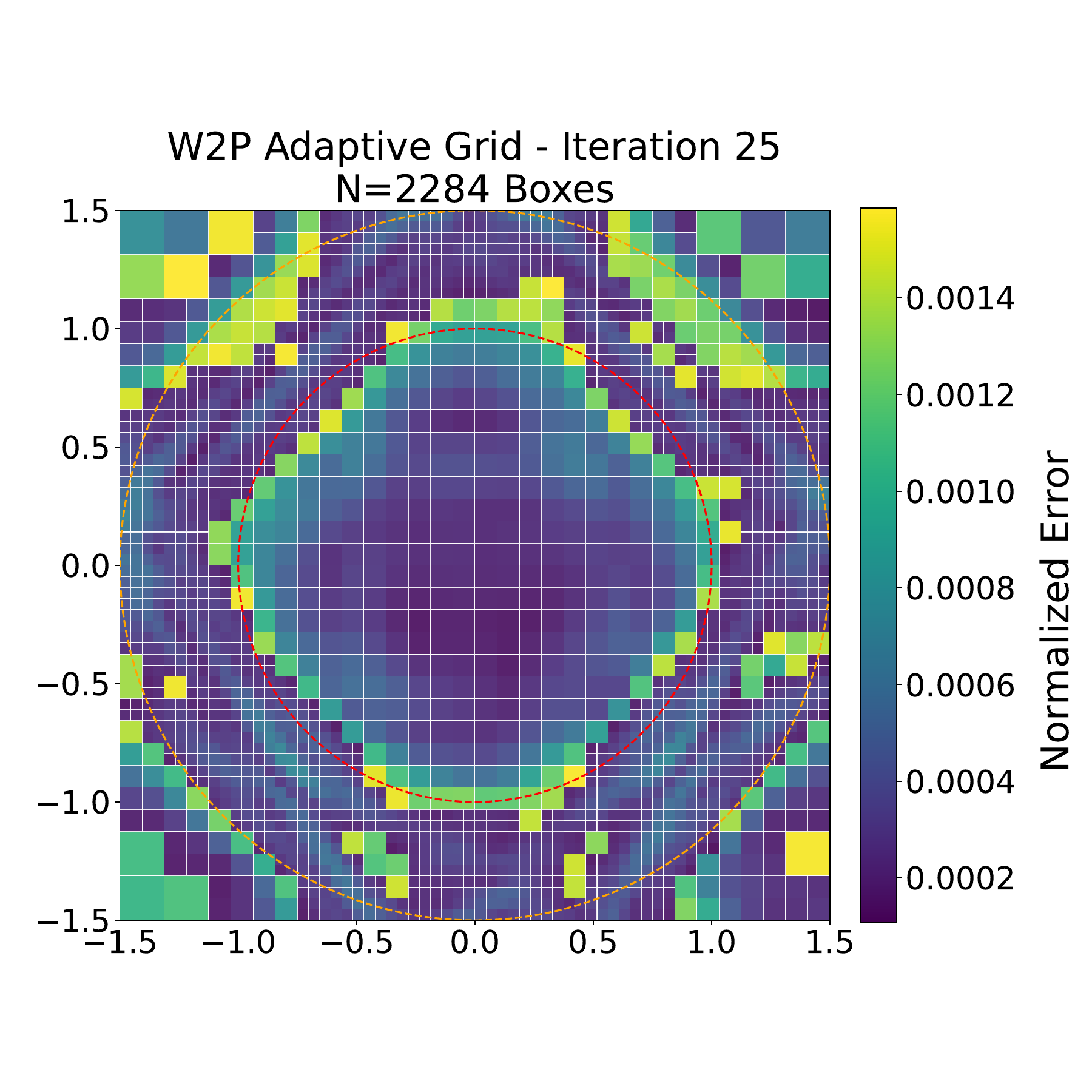}
			\caption{\small Heatmap for the wide architecture (W2P).} 
			\label{fig:w2p_wide_heatmap}
		\end{subfigure}
		
		\caption{Heatmaps of local errors for the Sobolev norm for the deep architecture (left) and wide architecture (right) from the 2d disk experiment.}
		\label{fig:2d_coin_heatmaps_combined}
	\end{figure}
	
	\subsubsection{Bounding the energy norm for elliptic PINNs}
	\label{sec:Bounding_Interior_Residuals_for_Elliptic_PINN}
	We analyze the interior residual bound calculated by bounding the energy norm as in Proposition~\ref{prop: energy interval enclosure} and its theoretical convergence from Proposition~\ref{prop: convergence}. 
	To approximate a weak solution of a PDE, we trained a neural network with 3 hidden layers of 32 neurons, the tanh activation function, and the Adam optimizer, calculating the loss as $|\Delta\Phi-f|$ via the Monte Carlo method as in \cite{Error_estimates_de2024error}.
	
	The residual is evaluated on 1000 sample points in the interior and on 800 sample points on the boundary. Let $\Omega = [0,1]^2$, $z = (x + y - 1)$, and the PDE of the form
	\begin{align}
		\label{eq: pde tanh}
		\Delta u &= -4 \left(1 - \tanh^2(z)\right) \tanh(z) & \text{in } \Omega, \\
		u &= \tanh(z) & \text{on } \partial\Omega.
	\end{align}
	The interior residual for the PINN approximation $u_\theta$ is defined as $r = \Delta u_\theta - f$.
	We chose the PDE~\eqref{eq: pde tanh} since fitting a PINN to this tanh-wavefront results in a highly localized error of the PINN, as shown in Figure~\ref{fig: energy norm error heatmap}. In Figure~\ref{fig: energy norm reduction rate}, we observe a reduction rate that matches the theoretical geometric error reduction rate of Proposition~\ref{prop: convergence}. The global error is not normalized in this plot, since we want to show the convergence towards zero, which is expected for a PINN fitting the Laplace operator. Proposition~\ref{prop: convergence} relies on the local reduction condition of the Dörfler marking, and we show by Figure~\ref{fig: deep local reduction energy error} that for the uniform refinement, this condition is satisfied for the interior residual bound.

	\begin{figure}[htb]
		\begin{subfigure}[t]{0.48\linewidth}
			\includegraphics[width=\linewidth]{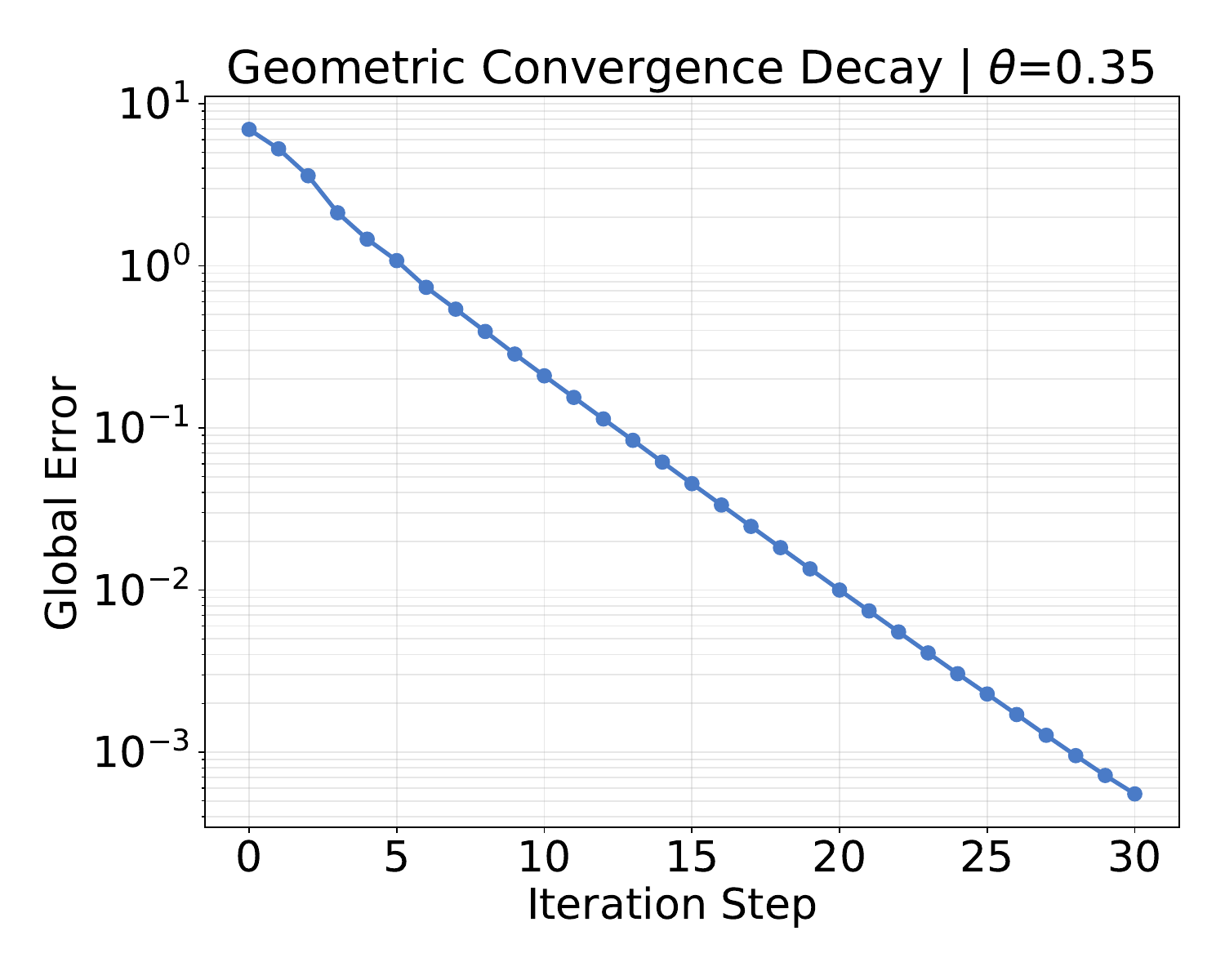}
			\caption{\small Convergence of the global error for the energy norm with $\theta=0.32$}
			\label{fig: energy norm reduction rate}
		\end{subfigure}
		\hfill
		\begin{subfigure}[t]{0.48\linewidth}
			\centering
			\includegraphics[width=\linewidth]{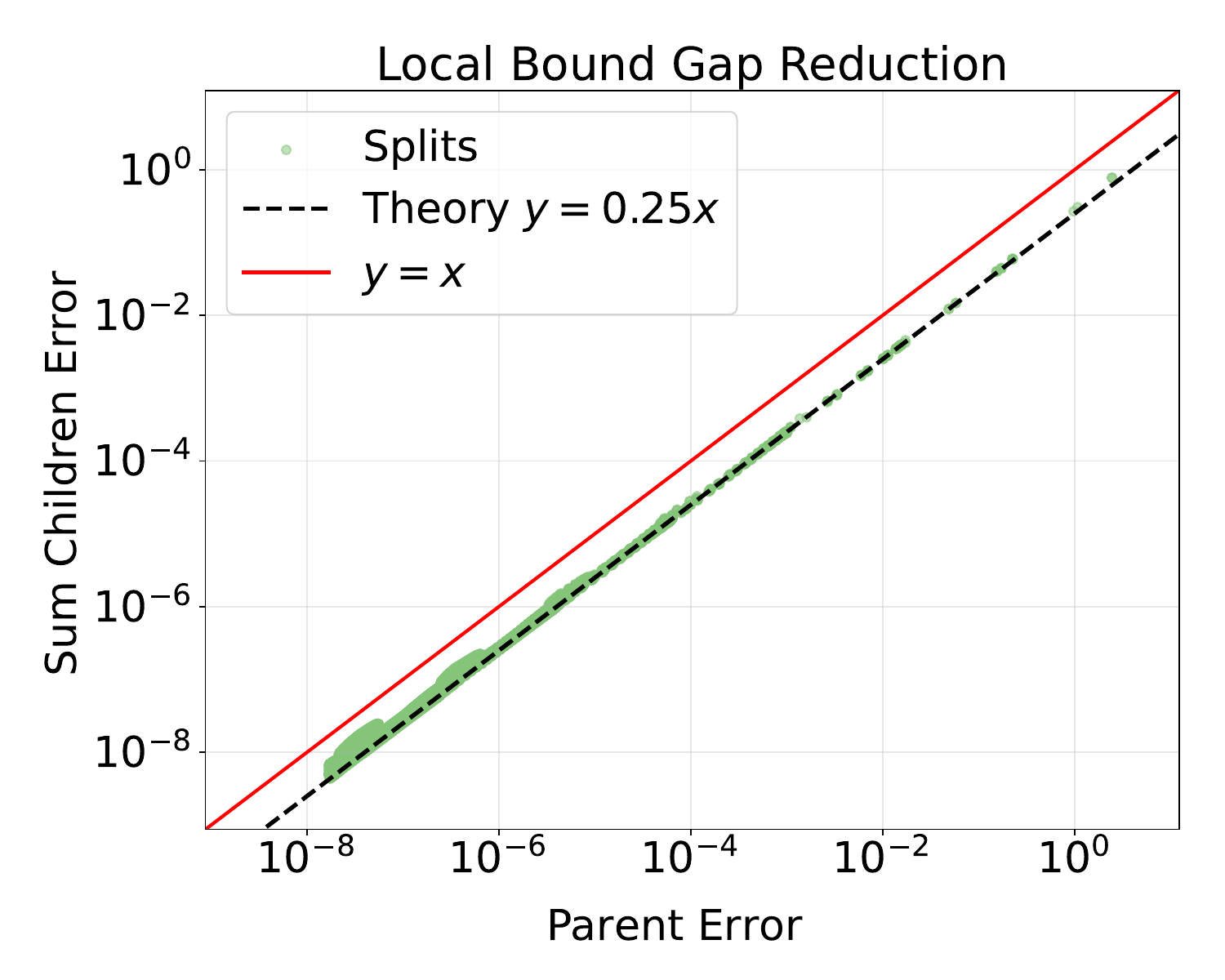}
			\caption{\small Local reduction}
			\label{fig: deep local reduction energy error}
		\end{subfigure}
		\hfill
		\begin{subfigure}[t]{0.48\linewidth}
			\centering
			\includegraphics[width=\linewidth]{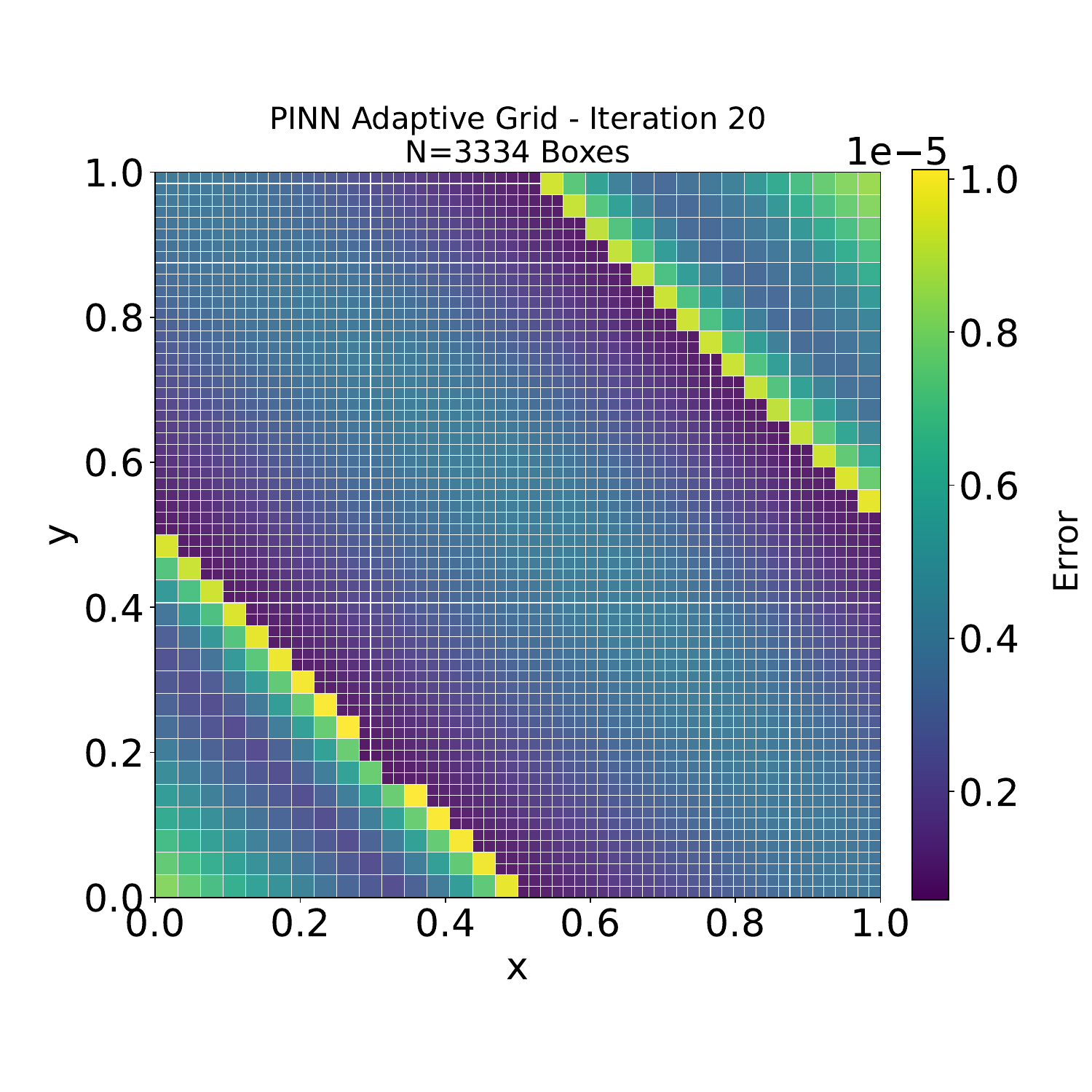}
			\caption{\small Heatmap of the local errors for the energy norm.}
			\label{fig: energy norm error heatmap}
		\end{subfigure}
		\hfill    
		\caption{PINN~\eqref{eq: pde tanh} experiment: convergence of the global error for the energy norm and local error reduction (Remark~\ref{rem:HalfRefinement}).}
		\label{fig: tanh PINN analysis}
	\end{figure}
	
	\section*{Acknowledgements}
	P.C.P. was supported by the Austrian Science Fund (FWF) Project P-37010.
	
	\bibliographystyle{abbrv}
	\bibliography{common/bib.bib}

@article{petersenzech,
  title        = {Mathematical Theory of Deep Learning},
  subtitle     = {Monograph},
  date         = {2025-09-03},
  journaltitle = {Springer Nature},
  author       = {Petersen, Philipp and Zech, Jakob},
  note         = {Monograph; bibliographic details incomplete}
}

@article{hanin2019deep,
  title={Deep relu networks have surprisingly few activation patterns},
  author={Hanin, Boris and Rolnick, David},
  journal={Advances in neural information processing systems},
  volume={32},
  year={2019}
}

@book{verfurth2013posteriori,
  title={A posteriori error estimation techniques for finite element methods},
  author={Verf{\"u}rth, R{\"u}diger},
  year={2013},
  publisher={OUP Oxford}
}

@article{Error_estimates_de2024error,
  title={Error estimates for physics-informed neural networks approximating the Navier--Stokes equations},
  author={De Ryck, Tim and Jagtap, Ameya D and Mishra, Siddhartha},
  journal={IMA Journal of Numerical Analysis},
  volume={44},
  number={1},
  pages={83--119},
  year={2024},
  publisher={Oxford University Press}
}

@article{dorfler1996convergent,
  title={A convergent adaptive algorithm for Poisson’s equation},
  author={D{\"o}rfler, Willy},
  journal={SIAM Journal on Numerical Analysis},
  volume={33},
  number={3},
  pages={1106--1124},
  year={1996},
  publisher={SIAM}
}

@article{babuska1978posteriori,
  title={A-posteriori error estimates for the finite element method},
  author={Babu{\v{s}}ka, Ivo and Rheinboldt, Werner C},
  journal={International journal for numerical methods in engineering},
  volume={12},
  number={10},
  pages={1597--1615},
  year={1978},
  publisher={Wiley Online Library}
}

@article{ainsworth1997posteriori,
  title={A posteriori error estimation in finite element analysis},
  author={Ainsworth, Mark and Oden, J Tinsley},
  journal={Computer methods in applied mechanics and engineering},
  volume={142},
  number={1-2},
  pages={1--88},
  year={1997},
  publisher={Elsevier}
}

@article{raissi2019pinns,
  title={Physics-informed neural networks: A deep learning framework for solving forward and inverse problems involving nonlinear partial differential equations},
  author={Raissi, Maziar and Perdikaris, Paris and Karniadakis, George E},
  journal={Journal of Computational physics},
  volume={378},
  pages={686--707},
  year={2019},
  publisher={Elsevier}
}

@article{sirignano2018dgm,
  title={DGM: A deep learning algorithm for solving partial differential equations},
  author={Sirignano, Justin and Spiliopoulos, Konstantinos},
  journal={Journal of computational physics},
  volume={375},
  pages={1339--1364},
  year={2018},
  publisher={Elsevier}
}

@article{lu2021deeponet,
  title={Deeponet: Learning nonlinear operators for identifying differential equations based on the universal approximation theorem of operators},
  author={Lu, Lu and Jin, Pengzhan and Karniadakis, George Em},
  journal={arXiv preprint arXiv:1910.03193},
  year={2019}
}

@article{kovachki2021neuraloperator,
  title={Neural operator: Learning maps between function spaces with applications to pdes},
  author={Kovachki, Nikola and Li, Zongyi and Liu, Burigede and Azizzadenesheli, Kamyar and Bhattacharya, Kaushik and Stuart, Andrew and Anandkumar, Anima},
  journal={Journal of Machine Learning Research},
  volume={24},
  number={89},
  pages={1--97},
  year={2023}
}

@book{evans2022partial,
  title={Partial differential equations},
  author={Evans, Lawrence C},
  volume={19},
  year={2022},
  publisher={American mathematical society}
}

@book{shalev2014understanding,
  title={Understanding machine learning: From theory to algorithms},
  author={Shalev-Shwartz, Shai and Ben-David, Shai},
  year={2014},
  publisher={Cambridge university press}
}

@article{neyshabur2018bounds,
  title={A {PAC}-bayesian approach to spectrally-normalized margin bounds for neural networks},
  author={Neyshabur, Behnam and Bhojanapalli, Srinadh and Srebro, Nathan},
  journal={arXiv preprint arXiv:1707.09564},
  year={2017}
}

@article{grohsvoigtlaendergap,
  title={Proof of the theory-to-practice gap in deep learning via sampling complexity bounds for neural network approximation spaces},
  author={Grohs, Philipp and Voigtlaender, Felix},
  journal={Found Comput Math},
  year={2023},
  doi={10.1007/s10208-023-09616-8}
}

@inproceedings{zhang2019recurjac,
  title={Recurjac: An efficient recursive algorithm for bounding jacobian matrix of neural networks and its applications},
  author={Zhang, Huan and Zhang, Pengchuan and Hsieh, Cho-Jui},
  booktitle={Proceedings of the AAAI Conference on Artificial Intelligence},
  volume={33},
  number={01},
  pages={5757--5764},
  year={2019}
}

@book{moore2009introduction,
  title={Introduction to interval analysis},
  author={Moore, Ramon E and Kearfott, R Baker and Cloud, Michael J},
  year={2009},
  publisher={SIAM}
}

@book{moore1966interval,
  title={Interval analysis},
  author={Moore, Ramon E},
  year={1966},
  publisher={Prentice-Hall}
}

@inproceedings{rump2010verification,
  title={Verification methods: Rigorous results using floating-point arithmetic},
  author={Rump, Siegfried M},
  booktitle={Proceedings of the 2010 International Symposium on Symbolic and Algebraic Computation},
  pages={3--4},
  year={2010}
}

@misc{Gowal2018,
  author        = {Gowal, Sven and Dvijotham, Krishnamurthy and Stanforth, Robert and
                   Bunel, Rudy and Qin, Chongli and Uesato, Jonathan and
                   Arandjelovic, Relja and Mann, Timothy and Kohli, Pushmeet},
  title         = {On the Effectiveness of Interval Bound Propagation for Training
                   Verifiably Robust Models},
  year          = {2018},
  eprint        = {1810.12715},
  archiveprefix = {arXiv},
  primaryclass  = {cs.LG},
  note          = {Preprint; published as ``Scalable Verified Training for Provably
                   Robust Image Classification'' at ICCV 2019}
}

@inproceedings{Mirman2018,
  author        = {Mirman, Matthew and Gehr, Timon and Vechev, Martin},
  title         = {Differentiable Abstract Interpretation for Provably Robust Neural
                   Networks},
  booktitle     = {Proceedings of the 35th International Conference on Machine Learning
                   (ICML)},
  pages         = {3575--3583},
  year          = {2018},
  volume        = {80},
  series        = {Proceedings of Machine Learning Research},
  publisher     = {PMLR},
  url           = {http://proceedings.mlr.press/v80/mirman18b.html}
}

@inproceedings{Zhang2018,
  author        = {Zhang, Huan and Weng, Tsui-Wei and Chen, Pin-Yu and
                   Hsieh, Cho-Jui and Daniel, Luca},
  title         = {Efficient Neural Network Robustness Certification with General
                   Activation Functions},
  booktitle     = {Advances in Neural Information Processing Systems 31 (NeurIPS)},
  pages         = {4939--4948},
  year          = {2018},
  url           = {https://proceedings.neurips.cc/paper/2018/hash/d04863f100d59b3eb688a11f95b0ae60-Abstract.html}
}

@inproceedings{Xu2020,
  author        = {Xu, Kaidi and Shi, Zhouxing and Zhang, Huan and Wang, Yihan and
                   Chang, Kai-Wei and Huang, Minlie and Kailkhura, Bhavya and
                   Lin, Xue and Hsieh, Cho-Jui},
  title         = {Automatic Perturbation Analysis for Scalable Certified Robustness
                   and Beyond},
  booktitle     = {Advances in Neural Information Processing Systems 33 (NeurIPS)},
  pages         = {1129--1141},
  year          = {2020},
  url           = {https://proceedings.neurips.cc/paper/2020/hash/0cbc5671ae26f67871cb914d81ef8fc1-Abstract.html}
}

@article{Singh2019,
  author        = {Singh, Gagandeep and Gehr, Timon and P{\"u}schel, Markus and
                   Vechev, Martin},
  title         = {An Abstract Domain for Certifying Neural Networks},
  journal       = {Proceedings of the {ACM} on Programming Languages},
  volume        = {3},
  number        = {{POPL}},
  pages         = {41:1--41:30},
  year          = {2019},
  doi           = {10.1145/3290354}
}

@inproceedings{Wang2021,
  author        = {Wang, Shiqi and Zhang, Huan and Xu, Kaidi and Lin, Xue and
                   Jana, Suman and Hsieh, Cho-Jui and Kolter, J. Zico},
  title         = {{$\beta$-CROWN}: Efficient Bound Propagation with Per-neuron Split
                   Constraints for Complete and Incomplete Neural Network Robustness
                   Verification},
  booktitle     = {Advances in Neural Information Processing Systems 34 (NeurIPS)},
  pages         = {29909--29921},
  year          = {2021},
  url           = {https://proceedings.neurips.cc/paper/2021/hash/fac7fead96dafceaf80c1daffeae82a4-Abstract.html}
}

@misc{Kapela2024,
  author        = {Kapela, Tomasz and Mrozek, Marian and Wilczak, Daniel and
                   Zgliczy{\'n}ski, Piotr and Spurek, Przemys{\l}aw},
  title         = {Make Interval Bound Propagation Great Again},
  year          = {2024},
  eprint        = {2410.03373},
  archiveprefix = {arXiv},
  primaryclass  = {cs.LG}
}

@inproceedings{Shi2022,
  author        = {Shi, Zhouxing and Wang, Yihan and Zhang, Huan and
                   Yi, Jinfeng and Hsieh, Cho-Jui},
  title         = {Efficiently Computing Local Lipschitz Constants of Neural Networks
                   via Bound Propagation},
  booktitle     = {Advances in Neural Information Processing Systems 35 (NeurIPS)},
  year          = {2022},
  url           = {https://proceedings.neurips.cc/paper_files/paper/2022/hash/7b58f55fa6f2d2ea77290d028ecc0e29-Abstract-Conference.html}
}

@inproceedings{Fazlyab2019,
  author        = {Fazlyab, Mahyar and Robey, Alexander and Hassani, Hamed and
                   Morari, Manfred and Pappas, George J.},
  title         = {Efficient and Accurate Estimation of {Lipschitz} Constants for Deep
                   Neural Networks},
  booktitle     = {Advances in Neural Information Processing Systems 32 (NeurIPS)},
  pages         = {11423--11434},
  year          = {2019},
  url           = {https://proceedings.neurips.cc/paper/2019/hash/95e1533eb1b20a97777749fb94fdb944-Abstract.html}
}

@inproceedings{Entesari2024,
  author        = {Entesari, Taha and Sharifi, Reza and Fazlyab, Mahyar},
  title         = {Certified Invertibility in Neural Networks via Mixed-Integer
                   Programming},
  booktitle     = {Proceedings of the 41st International Conference on Machine Learning
                   (ICML)},
  year          = {2024},
  note          = {Compositional curvature bounds paper; see also arXiv:2407.11498}
}

@misc{Sharifi2024,
  author        = {Sharifi, Reza and Fazlyab, Mahyar},
  title         = {Derivative-Preserving Reachability Analysis of Neural Networks},
  year          = {2024},
  eprint        = {2406.04476},
  archiveprefix = {arXiv},
  primaryclass  = {cs.LG}
}

@inproceedings{Laurel2022,
  author        = {Laurel, Jacob and Pailoor, Shankara and Sidrane, Claire and
                   Ugare, Shubham and Madhyastha, Sanjit},
  title         = {Toward Better Gradient Descent via Unified Abstract Interpretation
                   of Automatic Differentiation},
  booktitle     = {Proceedings of the ACM on Programming Languages ({OOPSLA})},
  volume        = {6},
  year          = {2022},
  doi           = {10.1145/3563325},
  note          = {Abstract interpretation of higher-order automatic differentiation
                   with certified derivative bounds via dual-number abstractions}
}

@article{Bunel2020,
  author        = {Bunel, Rudy and Lu, Jingyue and Turkaslan, Ilker and
                   Torr, Philip H. S. and Kohli, Pushmeet and Kumar, M. Pawan},
  title         = {Branch and Bound for Piecewise Linear Neural Network Verification},
  journal       = {Journal of Machine Learning Research},
  volume        = {21},
  number        = {42},
  pages         = {1--39},
  year          = {2020},
  url           = {https://jmlr.org/papers/v21/19-468.html}
}

@inproceedings{Wang2018ReluVal,
  author        = {Wang, Shiqi and Pei, Kexin and Whitehouse, Justin and
                   Yang, Junfeng and Jana, Suman},
  title         = {Formal Security Analysis of Neural Networks Using Symbolic
                   Intervals},
  booktitle     = {27th {USENIX} Security Symposium ({USENIX} Security 18)},
  pages         = {1599--1614},
  year          = {2018},
  address       = {Baltimore, MD},
  publisher     = {{USENIX} Association},
  isbn          = {978-1-939133-04-5},
  url           = {https://www.usenix.org/conference/usenixsecurity18/presentation/wang-shiqi}
}

@inproceedings{Ferrari2022,
  author        = {Ferrari, Claudio and M{\"u}ller, Mark Niklas and Jovanovi{\'c},
                   Nikola and Vechev, Martin},
  title         = {Complete Verification via Multi-Neuron Relaxation Guided
                   Branch-and-Bound},
  booktitle     = {Proceedings of the 10th International Conference on Learning
                   Representations (ICLR)},
  year          = {2022},
  url           = {https://openreview.net/forum?id=l_amHf1oaK}
}

@inproceedings{Zhang2022GCP,
  author        = {Zhang, Huan and Wang, Shiqi and Xu, Kaidi and Li, Linyi and
                   Li, Bo and Jana, Suman and Hsieh, Cho-Jui and Kolter, J. Zico},
  title         = {General Cutting Planes for Bound-Propagation-Based Neural Network
                   Verification},
  booktitle     = {Advances in Neural Information Processing Systems 35 (NeurIPS)},
  year          = {2022},
  url           = {https://proceedings.neurips.cc/paper_files/paper/2022/hash/20914db5742a5aa1b60e7c9d1d88e796-Abstract-Conference.html}
}

@inproceedings{Shi2024GenBaB,
  author        = {Shi, Zhouxing and Jin, Qirui and Kolter, J. Zico and Jana, Suman and
                   Hsieh, Cho-Jui and Zhang, Huan},
  title         = {Neural Network Verification with Branch-and-Bound for General
                   Nonlinearities},
  booktitle     = {Proceedings of the 31st International Conference on Tools and
                   Algorithms for the Construction and Analysis of Systems (TACAS)},
  year          = {2025},
  eprint        = {2405.21063},
  archiveprefix = {arXiv},
  primaryclass  = {cs.LG}
}

@inproceedings{Boetius2025,
  author        = {Boetius, Alexander and Leue, Stefan and Sutter, Tobias},
  title         = {Probabilistic Verification of Neural Networks against Group
                   Fairness},
  booktitle     = {Proceedings of the 42nd International Conference on Machine Learning
                   (ICML)},
  year          = {2025},
  eprint        = {2405.17556},
  archiveprefix = {arXiv},
  primaryclass  = {cs.LG},
  note          = {Computes guaranteed bounds on probabilities by aggregating interval
                   bounds over domain partitions}
}

@article{CorlissRall1987,
  author        = {Corliss, George F. and Rall, L. B.},
  title         = {Adaptive, Self-Validating Numerical Quadrature},
  journal       = {{SIAM} Journal on Scientific and Statistical Computing},
  volume        = {8},
  number        = {5},
  pages         = {831--847},
  year          = {1987},
  doi           = {10.1137/0908069}
}

@article{Petras2002,
  author        = {Petras, Knut},
  title         = {Autonomous and Reliable Numerical Integration of Functions with
                   Irregular Behavior},
  journal       = {Journal of Computational and Applied Mathematics},
  volume        = {145},
  number        = {2},
  pages         = {345--359},
  year          = {2002},
  doi           = {10.1016/S0377-0427(01)00588-8}
}

@inproceedings{Johansson2018,
  author        = {Johansson, Fredrik},
  title         = {Numerical Integration in Arbitrary-Precision Ball Arithmetic},
  booktitle     = {Mathematical Software -- {ICMS} 2018},
  series        = {Lecture Notes in Computer Science},
  volume        = {10931},
  pages         = {255--263},
  year          = {2018},
  publisher     = {Springer},
  doi           = {10.1007/978-3-319-96418-8_30},
  eprint        = {1802.07942},
  archiveprefix = {arXiv}
}

@article{BerzMakino1998,
  author        = {Berz, Martin and Makino, Kyoko},
  title         = {Verified Integration of {ODEs} and Flows Using Differential
                   Algebraic Methods on High-Order {Taylor} Models},
  journal       = {Reliable Computing},
  volume        = {4},
  number        = {4},
  pages         = {361--369},
  year          = {1998},
  doi           = {10.1023/A:1024467732637}
}

@article{BerzMakino2003,
  author        = {Berz, Martin and Makino, Kyoko},
  title         = {Performance of {Taylor} Model Methods for Validated Integration of
                   {ODEs}},
  journal       = {International Journal of Pure and Applied Mathematics},
  volume        = {6},
  number        = {3},
  pages         = {239--316},
  year          = {2003}
}

@book{Moore2009,
  author        = {Moore, Ramon E. and Kearfott, R. Baker and Cloud, Michael J.},
  title         = {Introduction to Interval Analysis},
  publisher     = {Society for Industrial and Applied Mathematics (SIAM)},
  address       = {Philadelphia},
  year          = {2009},
  doi           = {10.1137/1.9780898717716}
}

@book{Tucker2011,
  author        = {Tucker, Warwick},
  title         = {Validated Numerics: A Short Introduction to Rigorous Computations},
  publisher     = {Princeton University Press},
  address       = {Princeton, NJ},
  year          = {2011}
}

@article{Rump2010,
  author        = {Rump, Siegfried M.},
  title         = {Verification Methods: Rigorous Results Using Floating-Point
                   Arithmetic},
  journal       = {Acta Numerica},
  volume        = {19},
  pages         = {287--449},
  year          = {2010},
  doi           = {10.1017/S096249291000005X}
}

@book{Nakao2019,
  author        = {Nakao, Mitsuhiro T. and Plum, Michael and Watanabe, Yoshitaka},
  title         = {Numerical Verification Methods and Computer-Assisted Proofs for
                   Partial Differential Equations},
  series        = {Springer Series in Computational Mathematics},
  volume        = {53},
  publisher     = {Springer},
  address       = {Singapore},
  year          = {2019},
  doi           = {10.1007/978-981-13-7669-6}
}

@article{Dorfler1996,
  author        = {D{\"o}rfler, Willy},
  title         = {A Convergent Adaptive Algorithm for {Poisson's} Equation},
  journal       = {{SIAM} Journal on Numerical Analysis},
  volume        = {33},
  number        = {3},
  pages         = {1106--1124},
  year          = {1996},
  doi           = {10.1137/0733054}
}

@article{DeRyck2024,
  author        = {De Ryck, Tim and Mishra, Siddhartha},
  title         = {Generic Bounds on the Approximation Error for Physics-Informed
                   (and Other) Neural Networks},
  journal       = {Acta Numerica},
  volume        = {33},
  pages         = {633--713},
  year          = {2024},
  doi           = {10.1017/S0962492922000150}
}

@misc{Lorenz2024,
  author        = {Lorenz, Davide Antonio and Bacho, Rami and Kutyniok, Gitta},
  title         = {Certified Neural Network Approximations for Semilinear Wave
                   Equations},
  year          = {2024},
  eprint        = {2402.07153},
  archiveprefix = {arXiv},
  primaryclass  = {math.NA}
}

@misc{Doumeche2025,
  author        = {Doum{\`e}che, Nathan and Biau, G{\'e}rard and Boyer, Claire and
                   Moulines, {\'E}ric},
  title         = {Physics-Informed Machine Learning as a Kernel Method},
  year          = {2025},
  eprint        = {2402.11953},
  archiveprefix = {arXiv},
  primaryclass  = {stat.ML},
  note          = {A priori PINN error analysis paper cited as Doum\`eche et al. (2025)
                   in the reviewed manuscript}
}

@article{Hillebrecht2022,
  author        = {Hillebrecht, Birgit and Unger, Benjamin},
  title         = {Rigorous Numerical Analysis of Neural Networks},
  journal       = {{IEEE} Transactions on Neural Networks and Learning Systems},
  year          = {2022},
  note          = {A posteriori PINN error bounds via semigroup theory;
                   see also the 2025 extended version}
}

@article{Hillebrecht2025,
  author        = {Hillebrecht, Birgit and Unger, Benjamin},
  title         = {Certified Machine Learning: {A} Posteriori Error Estimation for
                   Physics-Informed Neural Networks},
  journal       = {{IEEE} Transactions on Neural Networks and Learning Systems},
  volume        = {36},
  number        = {1},
  pages         = {1583--1593},
  year          = {2025},
  doi           = {10.1109/TNNLS.2024.3349450}
}

@misc{Ernst2025,
  author        = {Ernst, Oliver and Rekatsinas, Athanasios and Urban, Karsten},
  title         = {Certified Error Bounds for Physics-Informed Neural Networks via
                   {Riesz} Representations},
  year          = {2025},
  eprint        = {2502.20336},
  archiveprefix = {arXiv},
  primaryclass  = {math.NA}
}

@article{Berrone2022,
  author        = {Berrone, Stefano and Canuto, Claudio and Pintore, Moreno},
  title         = {Variational Physics Informed Neural Networks: The Role of
                   Quadratures and Test Functions},
  journal       = {Journal of Scientific Computing},
  volume        = {92},
  pages         = {100},
  year          = {2022},
  doi           = {10.1007/s10915-022-01950-4},
  note          = {Ann. Univ. Ferrara 68:575--595 in an earlier version; cf.\
                   the published journal version}
}

@inproceedings{Eiras2024,
  author        = {Eiras, Francisco and Petrov, Aleksandar and Raina, Shubham and
                   Bhatt, Aditi and Sherborne, Tom and Cebere, Bogdan and
                   Bhatt, Usman and Torr, Philip H. S. and Kumar, M. Pawan and
                   Bunel, Rudy},
  title         = {{$\partial$-CROWN}: Certified Bound Propagation for Differentiable
                   Programs},
  booktitle     = {Proceedings of the 41st International Conference on Machine Learning
                   (ICML)},
  year          = {2024},
  eprint        = {2305.10157},
  archiveprefix = {arXiv},
  primaryclass  = {cs.LG}
}

@misc{Mukherjee2026,
  author        = {Mukherjee, Arnab and others},
  title         = {Generalization Bounds for Physics-Informed Neural Networks via
                   Residual Control},
  year          = {2026},
  eprint        = {2603.19165},
  archiveprefix = {arXiv},
  primaryclass  = {math.NA}
}

@misc{Tanaka2026,
  author        = {Tanaka, Ryo and Yatabe, Kohei},
  title         = {Learn and Verify: Interval-Arithmetic-Based Verification of
                   Physics-Informed Neural Networks for {ODEs}},
  year          = {2026},
  eprint        = {2601.19818},
  archiveprefix = {arXiv},
  primaryclass  = {cs.LG}
}

@article{E2022,
  author        = {E, Weinan and Ma, Chao and Wu, Lei},
  title         = {Barron Spaces and the Compositional Function Spaces for Neural
                   Network Models},
  journal       = {Constructive Approximation},
  volume        = {55},
  pages         = {369--406},
  year          = {2022},
  doi           = {10.1007/s00365-021-09549-y}
}

@misc{LiaoMing2023,
  author        = {Liao, Zhiqiang and Ming, Pingbing},
  title         = {Spectral {Barron} Spaces and Neural Network Approximation in the
                   High-Frequency Regime},
  year          = {2023},
  note          = {Cited as Liao and Ming (2023) in the reviewed manuscript;
                   see also related work on spectral Barron spaces in JMLR}
}

@article{Guhring2020,
  author        = {G{\"u}hring, Ingo and Kutyniok, Gitta and Petersen, Philipp},
  title         = {Error Bounds for Approximations with Deep {ReLU} Neural Networks in
                   {$W^{s,p}$} Norms},
  journal       = {Analysis and Applications},
  volume        = {18},
  number        = {5},
  pages         = {803--859},
  year          = {2020},
  doi           = {10.1142/S0219530519410021}
}

@article{SiegelXu2023,
  author        = {Siegel, Jonathan W. and Xu, Jinchao},
  title         = {Characterization of the Variation Spaces Corresponding to Shallow
                   Neural Networks},
  journal       = {Constructive Approximation},
  volume        = {57},
  pages         = {1017--1078},
  year          = {2023},
  doi           = {10.1007/s00365-023-09626-4}
}

@article{DeVore2021,
  author        = {DeVore, Ronald and Hanin, Boris and Petrova, Guergana},
  title         = {Neural Network Approximation},
  journal       = {Acta Numerica},
  volume        = {30},
  pages         = {327--444},
  year          = {2021},
  doi           = {10.1017/S0962492921000052}
}

@inproceedings{Bartlett2017,
  author        = {Bartlett, Peter L. and Foster, Dylan J. and Telgarsky, Matus},
  title         = {Spectrally-Normalized Margin Bounds for Neural Networks},
  booktitle     = {Advances in Neural Information Processing Systems 30 (NIPS)},
  pages         = {6240--6249},
  year          = {2017},
  url           = {https://proceedings.neurips.cc/paper/2017/hash/b22b257ad0519d4500539da3c8895f41-Abstract.html}
}

@inproceedings{Neyshabur2015,
  author        = {Neyshabur, Behnam and Tomioka, Ryota and Srebro, Nathan},
  title         = {Norm-Based Capacity Control in Neural Networks},
  booktitle     = {Proceedings of the 28th Annual Conference on Learning Theory (COLT)},
  series        = {Proceedings of Machine Learning Research},
  volume        = {40},
  pages         = {1376--1401},
  year          = {2015},
  publisher     = {PMLR},
  url           = {http://proceedings.mlr.press/v40/Neyshabur15.html}
}

@inproceedings{Gonon2024,
  author        = {Gonon, Lukas and Grigoryeva, Lyudmila and Ortega, Juan-Pablo and
                   Ratsaby, Joel},
  title         = {Path-Norm Optimization for Residual Networks},
  booktitle     = {Proceedings of the 12th International Conference on Learning
                   Representations (ICLR)},
  year          = {2024},
  eprint        = {2310.01225},
  archiveprefix = {arXiv},
  primaryclass  = {cs.LG},
  note          = {Extends path norms to general DAG ReLU networks with biases,
                   skip connections, and pooling}
}
\end{document}